\documentclass[11pt,reqno]{amsart}
\usepackage{amsmath,amssymb}
\usepackage{dsfont}
\usepackage{xcolor}
\usepackage{indentfirst,csquotes}
\usepackage{mathtools}
\usepackage{MnSymbol}
\usepackage{verbatim}
\usepackage{upgreek}
\usepackage{subcaption}
\usepackage{graphicx}
\usepackage{epstopdf}
\usepackage{bm}
\mathtoolsset{showonlyrefs}
\usepackage{fancyhdr}
\usepackage[margin=1in]{geometry}
\usepackage[colorlinks=true,citecolor=magenta,linkcolor=blue]{hyperref}

\newcommand{\eps}{\varepsilon}
\newcommand{\eq}{\mathrm{eq}}
\newcommand{\Span}{\mathrm{Span}}
\newcommand{\norm}[1]{\left\|#1\right\|}

\newcommand{\R}{\mathbb{R}}

\newcommand{\dd}{\mathrm{d}}
\newcommand{\init}{\mathrm{in}}
\newcommand{\drift}{\mathrm{drift}}
\newcommand{\diff}{\mathrm{diff}}

\newcommand{\smallO}{
  \mathchoice
    {{\scriptstyle\mathcal{O}}}%
    {{\scriptstyle\mathcal{O}}}
    {{\scriptscriptstyle\mathcal{O}}}
    {\scalebox{.7}{$\scriptscriptstyle\mathcal{O}$}}
  }

\newtheorem{theorem}{Theorem}
\newtheorem{lemma}{Lemma}
\newtheorem{proposition}{Proposition}
\newtheorem{remark}{Remark}

\makeatletter
\def\blfootnote{\xdef\@thefnmark{}\@footnotetext}
\makeatother

\author[A. Bondesan]{Andrea Bondesan}
\address{Andrea Bondesan \hfill\break
	Department of Mathematical, Physical and Computer Sciences \hfill\break 
    	University of Parma \hfill\break
	Parco Area delle Scienze 53/A, 43124 Parma, Italy \vspace*{2mm}}
 \email{andrea.bondesan@unipr.it \vspace*{5mm}}

\author[J. Borsotti]{Jacopo Borsotti}
\address{Jacopo Borsotti \hfill\break
	School of Resource and Environmental Management  \hfill\break
	Simon Fraser University \hfill\break
	8888 University Drive, V5A 1S6 Burnaby, Canada \vspace*{2mm}}
\email{jacopo\_borsotti@sfu.ca \vspace*{5mm}}

\author[M. Fontana]{Mattia Fontana}
\address{Mattia Fontana \hfill\break
	Laboratoire Jean Alexandre Dieudonné \hfill\break
	Université Côte d'Azur \hfill\break
	Campus Sciences, Parc Valrose, 28 avenue Valrose, 06108 Nice, France \vspace*{2mm}}
\email{mattia.fontana@univ-cotedazur.fr \vspace*{5mm}}

\title[Kinetic epidemiological models on graphons]{Kinetic models of opinion-driven epidemic\\dynamics modulated by graphons}

\begin{document}

\vspace*{-0.8cm}
\begin{abstract}
We introduce new kinetic equations to describe epidemics' spread while accounting for individuals' opinions on protective behaviors. Opinion exchanges occur on a social network represented by a graphon, whose choice strongly influences the dynamics and leads to the emergence of complex nonlinear phenomena, like the creation of opinion leaders or the spontaneous formation of epidemic waves. Starting from individual-based interactions, we derive a nonlinear nonlocal Fokker--Planck model involving reaction terms and degenerate drift--diffusion operators, which depend on the underlying graphon. We establish rigorous results of convergence to equilibrium in $L^1$ space, via relative entropy estimates, and in homogeneous Sobolev spaces $\dot{H}^{-s}$, $s \in \big(\frac{1}{2}, 1\big)$, using Fourier-based techniques. We then design a structure-preserving scheme for the coupled opinion-epidemiological system, highlighting graphon effects: opinion leaders supporting protective behaviors limit disease spread, whereas influenceable individuals may shift toward opposing views, worsening epidemics. At last, we introduce a time-dependent quantity analogous to the effective reproduction number, whose oscillations are linked with the formation of epidemic waves. Notably, these waves are not induced by an explicit external forcing but they naturally emerge from the interactions between agents, depending on the connectivity level prescribed by the graphon.
\end{abstract}

\maketitle

\tableofcontents



\section{Introduction}

    \noindent The origins of mathematical epidemic modeling based on compartmental approaches trace back to the pioneering work of Kermack and McKendrick \cite{KerMcK}. Since that time, several models, usually described by systems of Ordinary Differential Equations (ODEs), have been developed to simulate complex biological dynamics \cite{Gae, LorPugSenZar, RasKoo} and to account for collective behaviors \cite{BulSenSot, DeldOnSenSot, Sch}. While some efforts were aimed at reproducing the spread of the epidemic as precisely as possible by considering, for example, secondary infections \cite{KakPugSenSot}, disease severity \cite{Bor}, different timescales \cite{JarKuePugSen1}, and age-dependent host mortality \cite{And}, starting from the COVID-19 pandemic it became clear that human behaviors should be incorporated in such models to effectively predict the outcome of the epidemics. Indeed, the presence of an underlying society is what distinguishes the spread of epidemics among humans from the ones occurring among other animal species, as observed in farms \cite{MacBis} and national parks \cite{BraDobHudCroSmi}. 
    
    The nowadays globalized society can be considered at three different levels: anyone can share their opinion through social networks, allowing them to reach people all over the world and amplifying their impact on the society \cite{RasPil}; citizens are physically connected with each other thanks to transportation and infrastructures \cite{CasBicMor}; governments implement strategies to slow down the spread of the epidemics, like quarantines \cite{Nus_etal} and vaccination campaigns \cite{Lan_etal}. While mobility and policy interventions have been extensively studied, with numerous contributions in the literature, the interplay between opinion dynamics and epidemic spreading remains less explored. For instance, from a graph-theoretical perspective, a considerable number of models has been developed to consider the movement of the citizens and the connections between them (see, e.g., \cite{CheXiaPer, FraLoy, HolTilDys, JarKuePugSen2, NalPat}), while optimal control strategies have been developed to organize effective quarantines and vaccination campaigns (see, e.g., \cite{AhmUsmKhaImr, BolBonDelGro, BolBonSorGro}). On the other hand, fewer works have been dedicated to linking opinion dynamics with epidemics (see, e.g., \cite{AlbCalDimZan, BonTosZan, DeRPugSenSor, DelLoyTos, DimPerTosZan, Zan}) and several questions remain open due to the complexity of such models, notably the relaxation to equilibrium of their solutions and the quantification of the convergence rates. 

    In this work, we focus on the impact that individuals' opinions have on the spread of the epidemic. Indeed, opinions influence whether a person will actively try to limit the transmission, for example by wearing face masks \cite{KarLuShiChePam}, or, conversely, will disregard the risks of infection, potentially endorsing conspiracy theories \cite{Tsa_etal}. In general, individuals with an opinion in favor of a protective behavior are less likely to contract the disease and, in case this happens, to spread it. Obviously, since opinions evolve over time, one has to model interactions among citizens to account for these dynamics \cite{DurFraWolZan, Tos, TosTosZan}. In addition, we shall also assume here that opinion exchanges occur on an underlying graphon \cite{Gla}, which is interpreted as the continuous extension of a dense graph and can be used to describe connections on a social network \cite{WanMaCao}. It is also important to consider the presence of opinion leaders, namely highly influential people (such as politicians and influencers) that are capable of altering citizens' opinions while leaving their own unaffected. Interestingly, from a mathematical point of view, one of the novelties of our approach is that, thanks to appropriate choices of the graphon, we are able to characterize both normal citizens and opinion leaders using a single model, instead of considering two separate classes of individuals like in \cite{BonBor, DieHee}. This feature highlights in particular the influence that a social network can have on the formation of opinions inside a population.

    In order to mimic opinion dynamics within a society composed of a large number of individuals, we rely on Toscani’s seminal kinetic model \cite{Tos}, which has been further expanded over more recent years (see, e.g., \cite{BonBor, CalDimTosZan, DurFraWolZan, TosTosZan}). Kinetic models describe opinion formation through binary interactions between agents and it is typically assumed that an individual's opinion evolves according to two mechanisms: compromise dynamics, where agents vary their opinions through interactions with others, and self-thinking dynamics. The main assumption behind these kinetic models is that when two agents interact, only small variations of their opinions can occur. Considering this quasi-invariant regime of opinions, Toscani \cite{Tos} showed that their dynamics can be described by Fokker--Planck-type equations, whose great mathematical advantage relies on the possibility of explicitly computing their equilibria, hence offering a clear description of key phenomena like consensus formation and opinion polarization \cite{BonBor}. These equilibria are crucial to understand how extreme opinions emerge within a society and preventing their formation is clearly important in epidemiological contexts, where polarized views can have harmful consequences. For example, in \cite{Zan} it is showed that opinion polarization against protective behaviors during an epidemic can lead to a significant increase in the number of infections, highlighting the real-world impact of opinion dynamics.

    In this paper, following \cite{AlbCalDimZan, BonTosZan, Zan}, we introduce a kinetic SIR-type model where, for the reasons previously described, the opinion of each individual plays a central role in determining the evolution of the epidemic. The operator modeling contagions and recoveries takes into account such opinions and leads to a generalization of the classical SIR model, allowing us to recover the latter after suitable integrations. Alongside this operator, we consider a collisional kinetic operator describing the evolution of opinions inside each compartment of individuals (susceptible, infected, and recovered) by taking into account an underlying graphon. Indeed, opinion dynamics could in general vary from one compartment to another since, for instance, infected individuals could be characterized by different self-thinking processes than susceptible ones. Our model is not fully kinetic like the one introduced in \cite{MarTosZan}, as we combine moment operators resembling ODE-type interactions with kinetic ones, in order to simulate complex phenomena. Finally, inspired by \cite{TosTosZan}, we also introduce a full kinetic model to measure the population's degree of concern regarding the severity of the disease spread. The idea is to track the popularity of each product related to the epidemic: news, videos, and social posts, that individuals tend to share whenever the latter are aligned with their own opinion. Note that the evolution of the products' popularity depends on the opinion dynamics, but the latter evolve independently of the former.

    Aside from the novelty of the proposed model, from a mathematical viewpoint the main contributions of this work are twofold. Firstly, we prove that solutions to our kinetic SIR-type model converge to equilibrium in the standard $L^1$ norm, making use of relative entropy techniques \cite{AurTosZan, BonBor, FurPulTerTos2} to determine the relaxation rate. We also demonstrate a result on the convergence to equilibrium for solutions to the kinetic equation modeling the popularity of products, in the homogeneous Sobolev spaces $\dot{H}^{-s}$, $s \in \big(\frac{1}{2}, 1\big)$ \cite{MarTosZan, TorTos2}. Secondly, we develop two-dimensional (in phase space) structure-preserving schemes to simulate both models. These numerical algorithms, which we have made publicly available, build upon the seminal work of Pareschi and Zanella \cite{ParZan} and further extend more recent structure-preserving schemes (see, e.g., \cite{BonTosZan, LoyZan, MarTosZan, ParTosTosZan, Zan}). The main novelty lies in the fact that our schemes are specifically designed to simulate compartmental models involving both kinetic and non-kinetic operators, and multiple distribution functions depending on more than one kinetic variable.

    Similarly, from the point of view of applications we also emphasize the two main contributions of our work. On one side, we show how to measure the propensity to interact of each individual by looking at their connections on the graphon. This allows us to demonstrate that opinion polarization is more likely to occur among individuals with a low propensity to interact, while those inclined to interact tend to reach a consensus, highlighting the importance of social interactions for preventing the creation of dangerous extreme opinions. This aligns in particular with the results obtained in \cite{BonBor}, where the first and second authors studied the effects that the level of social activity of the agents has on opinion formation. On the other side, the numerical simulation allow us to highlight the effects of the graphon: opinion leaders in favor of protective behaviors help limit disease spread, whereas easily influenceable individuals may shift toward opinions opposing protective behaviors, potentially leading to more severe epidemics. Moreover, we introduce a time-dependent quantity, resembling the classical basic reproduction number, which presents oscillations that can lead to different epidemic waves. Importantly, these different waves are not related to an explicit external forcing (like in previous works; see, e.g., \cite{BonTosZan}), but naturally arise from the complex structure of our model. 

\subsection{Outline of this work} 

    The article is organized as follows. In Section \ref{section2_0} we introduce our kinetic SIR system. In particular, concerning the modeling of opinion dynamics, starting from a Boltzmann description of the microscopic interactions, we derive the corresponding Fokker--Planck approximation in a quasi-invariant regime of the parameters. At the microscopic scale, the binary exchanges between agents depend on the social network through an interaction function that quantifies how strongly each individual can influence the others. At the macroscopic scale, instead, a graphon characterizes the tendency of agents to compare themselves with one another. In this section we also show why our model represents a generalization of the classical SIR equations, and we illustrate how to deduce the latter by suitable integration of the kinetic system. We also compute the (local) equilibria of the model. Finally, we prove several analytical properties of its solutions, including nonnegativity, $L^p$-regularity, and uniqueness. 
    
    In Section \ref{connectivity} we then introduce a simplification of the original model by replacing the previously defined interaction function and the graphon with a single quantity that incorporates all their information at the microscopic scale: the individuals' propensity to interact. After computing the equilibria of this new model, we investigate the connection between the propensity to interact and the phenomena of opinion polarization and consensus formation. At last, we prove that solutions to the simplified kinetic compartmental model converge to equilibrium in the strong $L^1$ distance and we also determine the relaxation rate. 
    
    We proceed in Section \ref{section4} by introducing a Fokker--Planck model, coupled with the original kinetic SIR system, that measures the population's level of concern regarding the disease spread. We calculate the equilibria of this model and we analyze their relation with the ones of the compartmental system. We then exhibit the main analytical properties of its solutions (existence, nonnegativity, regularity, and uniqueness) and we prove that they converge to equilibrium in the homogeneous Sobolev $\dot{H}^{-s}$ distance, for any $s \in \big(\frac{1}{2}, 1\big)$.

    Section \ref{sec:num} is dedicated to the numerical simulations. We start by introducing our structure-preserving scheme and by describing in details the algorithm. We then carry out several tests to reproduce the trends to equilibrium of the models demonstrated in the previous sections. We also analyze how the underlying graphon influences the evolution of the epidemic, both in presence and absence of opinion leaders. Moreover, we show the possible presence of different epidemic waves. 

    We conclude this work in Section \ref{section6}. As a last thing, we provide a link to the MATLAB code implementing our algorithms. 
	
\section{A kinetic SIR model on graphons} \label{section2_0}

    \noindent We consider a large population of agents subdivided in the following epidemiologically relevant compartments: susceptible ($S$) agents are the ones that can contract the disease; infectious agents ($I$) are responsible for the spread of the disease; removed ($R$) agents are healed and can no longer spread the disease. Let us define the set of these three compartments as $\mathcal{C} = \{S, I, R\}$. We associate to each agent a characteristic trait $w \in \mathcal{I} = [-1,1]$ which represents their opinion about the importance of protecting themselves and the others from the spread of the disease. In particular, $w = \pm 1$ correspond to the two opposite believes, namely agents with opinion $w=-1$ absolutely do not believe in the necessity of protections, whereas agents with opinion $w=1$ are in complete agreement with a protective behavior. We suppose that the frequency of interactions agent--agent leading to variations in their opinion depends on an underlying social network, modeled by a graphon \cite{BorChaCohZha1, BorChaCohZha2, DurFraWolZan, Gla, NalPat}, which is a measurable function $\mathcal{B} \colon [0,1]^2 \to [0, +\infty)$ satisfying $\mathcal{B} \in L^1([0,1]^2)$. Therefore, we assume that each agent is further characterized by a static position $x \in \Omega = [0, 1]$ on the graphon. Note that naturally these social interactions are distinct from the ones leading to the spread of the disease. 

    In the kinetic approach, the evolution of agents within each compartment is described statistically via the use of distribution functions $f_J = f_J(t,x,w)$, $J \in \mathcal{C}$, depending on time $t \geq 0$, position on the graphon $x \in \Omega$, and opinion $w \in \mathcal{I}$. This means that the quantity $f_J(t,x,w) \dd x \dd w$ counts the number of agents in the compartment $J \in \mathcal{C}$ that at time $t$ possess a position on the graphon in $[x, x+\dd x]$ and an opinion in $[w, w+\dd w]$. The number of agents in each compartment and their average opinion are then respectively defined as 
    \begin{equation} \label{eq:density and mean}
        \rho_J(t)=\int_{\Omega \times \mathcal{I}} f_J(t,x,w) \dd x \dd w, \qquad m_J(t)=\frac{1}{\rho_J(t)} \int_{\Omega \times \mathcal{I}} wf_J(t,x,w) \dd x \dd w. 
    \end{equation}
    Without loss of generality, we may assume that $\displaystyle \sum_{J \in \mathcal{C}} \rho_J (0) = 1$. We will be also interested in studying the evolution of the total average opinion of the population, given by 
    \begin{equation} \label{eq:total mean}
        m(t) = \frac{1}{\displaystyle\sum_{J \in \mathcal{C}}\rho_J(t)}\sum_{J \in \mathcal{C}} \int_{\Omega \times \mathcal{I}}  wf_J(t,x,w) \dd x \dd w. 
    \end{equation}
    The time evolution of the functions $f_J$ is then obtained by coupling the compartmental epidemiological description with the kinetic evolution of the social variable \cite{BonTosZan, DimParTosZan, Zan}. Denoting with $\mathbf{f} = (f_J)_{J\in\mathcal{C}}$ the vector distribution function of the three compartments, their evolution is described by the multi-species kinetic model
    \begin{equation} \label{eq:vectorial model}
        \partial_t \mathbf{f} = \mathbf{E}(\mathbf{f},\mathbf{f}) + \frac{1}{\tau} \mathbf{Q}(\mathbf{f},\mathbf{f}), \qquad t \in \R_+,\ x \in \Omega,\ w \in \mathcal{I},
    \end{equation}
    where the SIR-like operator $\mathbf{E}(\mathbf{f},\mathbf{f}) = (E_J(\mathbf{f},\mathbf{f}))_{J\in\mathcal{C}}$ models the evolution of the epidemic by prescribing suitable exchanges between compartments, while the kinetic-type operator $\mathbf{Q}(\mathbf{f},\mathbf{f}) = (Q_J(\mathbf{f}, \mathbf{f}))_{J\in\mathcal{C}}$ models the evolution of opinions inside the population. Moreover, the parameter $\tau > 0$ translates the fact that the opinion variations may happen at a different timescale than the disease spread. This means in particular that the proposed model is multiscale, since it interfaces the dynamics of the epidemic with that of the opinion. Let us now detail the explicit form of the operators $\mathbf{E}$ and $\mathbf{Q}$. 

\subsection{Effects of individuals' opinions on epidemic spread}

    Following \cite{BonTosZan}, we model the compartmental exchanges through the integral SIR-like operators
    \begin{equation} \label{eq:epidemiological operators E}
    \begin{split}
        E_S(\mathbf{f},\mathbf{f})(t,x,w) &= -f_S(t,x,w) \int_{\Omega \times \mathcal{I}} \beta_T(w,w_*) f_I(t,y,w_*) \dd y \dd w_*, \\[2mm]
        E_I(\mathbf{f},\mathbf{f})(t,x,w) &= f_S(t,x,w) \int_{\Omega \times \mathcal{I}} \beta_T(w,w_*) f_I(t,y,w_*) \dd y \dd w_* - \gamma f_I(t,x,w), \\[4mm]
        E_R(\mathbf{f},\mathbf{f})(t,x,w) &= \gamma f_I(t,x,w),
    \end{split}
    \end{equation}
    defined for any $t \geq 0$, $x \in \Omega$, and $w \in \mathcal{I}$. Here, the recovery rate $\gamma > 0$ is such that $1/\gamma$ gives the mean infectious period of the disease \cite{DieHee}, while $\beta_T(w,w_*)$ is a nonnegative decreasing (with respect to each variable taken individually) function that models the transition rate of agents between the compartments $S$ and $I$, giving a measure of the impact of protective behaviors on the interactions between susceptible and infectious individuals. One possible choice is given by 
    \begin{equation} \label{eq:function beta_T}
        \beta_T(w,w_*) = \beta (1-w)^{\alpha} (1-w_*)^{\alpha}, \qquad w,w_* \in \mathcal{I}, 
    \end{equation}
    for some parameters $\beta > 0$ (characterizing the baseline transmission rate of the epidemics) and $\alpha \geq 0$. In particular, we have assumed that the local incidence rate governing the transmission dynamics explicitly depends on the agents opinion $w$, meaning that the more protective an individual’s behavior is, the less likely they are to contract the disease. As a consequence of this modeling choice, the evolution of the disease is fully dependent on the agents' opinion about the use of protective behaviors. 

\subsection{Kinetic modeling of opinion dynamics} \label{kin_op}

    The opinion formation process is based on two mechanisms: 
    \begin{itemize}
        \item compromise dynamics, i.e., individuals tend to settle their differences of opinion; \\[-2mm]
        \item opinion fluctuation, i.e., every interaction between individuals is associated with a variation in opinions due to self-thinking. 
    \end{itemize}
    Consider an agent belonging to a compartment $J \in \mathcal{C}$ and characterized by a position on the graphon and an opinion given by the couple $(x,w) \in \Omega \times \mathcal{I}$ . When they interact with an agent from a compartment $J' \in \mathcal{C}$ and having kinetic traits $(y,w_*) \in \Omega \times \mathcal{I}$, their opinions vary according to the following microscopic interaction \cite{Tos}:
    \begin{equation} \label{eq:microscopic interactions}
    \begin{aligned}
        w' &= w + \lambda G(w,w_*)P(x,y) (w_*-w) + D(w) \eta_J, \\[2mm] 
        w_*' &= w_* + \lambda G(w_*,w)P(y,x) (w-w_*) + D(w_*) \eta_{J'}.
    \end{aligned}
    \end{equation}
    Here, the parameter $\lambda \in (0,1)$ characterizes the strength of the (deterministic) exchange of opinions. The symmetric function $G(w,w_*) \in [0, 1]$ describes how strongly the opinion $w_*$ affects the agent with opinion $w$, therefore it is natural to assume that $G(w,w_*)=\Bar{G}(|w-w_*|)$ with $\Bar{G}$ being a decreasing function that satisfies $\Bar{G}(0) = 1$. The function $P(x,y) \in [0, 1]$ ensures that the interactions depend on the connectivity of each agent and can be interpreted in different ways. For example, one can assume that the opinion exchanges \eqref{eq:microscopic interactions} depends on the quality of the relationship between the two agents: the better this is, the more strongly they influence each other. Moreover, one might assume that highly connected individuals are less influenced by each interaction than poorly connected ones. In Section \ref{connectivity} we will analyze the function $P$ in more detail. The opinion fluctuations are then modeled by centered random variables $\eta_{J}$ and $\eta_{J'}$, which depend on each agent's compartment and possess finite variance of respective values $\sigma_{J}^2>0$ and $\sigma_{J'}^2>0$. In particular, if both $G \equiv 1$ and $P\equiv 1$, the constant $\lambda$ would measure the propensity to move toward the average societal opinion (compromise dynamics), while the variances $\sigma_{J}^2$ and $\sigma_{J'}^2$ would measure the degree of opinion spreading due to diffusion processes (self-thinking). Finally, the function $D(w) \in [0, 1]$ encodes the relevance of such diffusion; for example, agents with an indifferent opinion ($w \simeq 0$) diffuse the most ($D(0)=1$), while those with an extreme opinion ($w \simeq \pm 1$) are less influenced by external factors ($D(\pm 1)=0$). 

    Denoting with $\left\langle X \right\rangle = \mathbb{E}(X)$ the expected value of some random variable $X$, simple computations show that 
    \begin{equation*}
    \begin{split}
        \left<w' + w_*'\right> &= w+w_*+G(w,w_*) \lambda(w_*-w) (P(x,y)-P(y,x)) + D(w)\left<\eta_{J}\right> + D(w_*)\left<\eta_{J'}\right> \\[2mm] 
        &=w+w_*+G(w,w_*) \lambda(w_*-w) (P(x,y)-P(y,x)). 
    \end{split}
    \end{equation*}
    In particular, when the function $P$ is symmetric, namely $P(x,y) = P(y,x)$, then $\left<w' + w_*'\right> = w+w_*$ and thus the average opinion is preserved by the microscopic opinion interactions \eqref{eq:microscopic interactions}. 

    Following \cite{Tos} we will usually consider 
    \begin{equation}
        D(w)=\sqrt{1-w^2}, \qquad w \in \mathcal{I},
    \end{equation}
    which ensures that the post-interaction opinions $w'$ and $w_*'$ stay in the interval $\mathcal{I}$. However, several other choices for the self-diffusion function $D$ are possible \cite{ParTosTosZan, Tos}.  

    If we neglect the compartmental exchanges due to the operator $\mathbf{E}$, employing standard methods of kinetic theory \cite{ParTos_book} and resorting to the derivation of the classical linear Boltzmann equation for elastic rarefied gases \cite{Cer}, one can show that the temporal variation of the vector distribution function $\mathbf{f}$ depends on a sequence of elementary binary interactions of type \eqref{eq:microscopic interactions}. In our case, given any smooth function $\varphi = \varphi(x,w)$ (defining the observable quantities of the system), we assume that the evolution of each $f_J$, $J \in \mathcal{C}$, is described by the integral equation  
    \begin{equation} \label{eq:weak form opinion}
    \begin{split}
        \frac{\dd}{\dd t} \int_{\Omega \times \mathcal{I}} & \varphi(x,w) f_J(t,x,w) \dd x \dd w \\ 
        &= \frac{1}{\tau}\sum_{J' \in \mathcal{C}} \int_{\Omega^2 \times \mathcal{I}^2} \mathcal{B}(x,y) \left<\varphi(x,w')-\varphi(x,w)\right> f_J(t,x,w) f_{J'}(t,y,w_*) \dd x \dd y \dd w \dd w_*,
    \end{split}
    \end{equation}
    where the kernel $\mathcal{B} \colon \Omega^2 \to [0, +\infty)$ is the given graphon, which can be thought of as the continuous equivalent of an adjacency matrix. In particular, the choice $\varphi \equiv 1$ in \eqref{eq:weak form opinion} leads to the conservation of mass in each compartment, i.e., the total number of agents inside each class (and therefore also inside the whole population) does not change over time. 

\subsection{Derivation of a mean-field description} \label{sub3}

    Starting from the collisional-like operators defined through the weak formulation \eqref{eq:weak form opinion}, we shall now perform a quasi-invariant limit procedure to derive a reduced model of Fokker--Planck-type, defining the operator $\mathbf{Q}$ in system \eqref{eq:vectorial model}. Such a simplified mean-field setting will in fact allow us to compute its kinetic equilibria explicitly, and study the evolution of the associated macroscopic quantities \eqref{eq:density and mean} and \eqref{eq:total mean}. 

    Fixed $0<\epsilon \ll 1$, we introduce the scalings 
    \begin{equation}
        \lambda \mapsto \epsilon\lambda, \qquad \sigma_J^2 \mapsto \epsilon\sigma_J^2, \qquad \tau \mapsto \epsilon\tau,
    \end{equation}
    which imply that the interaction \eqref{eq:microscopic interactions} produces only an extremely small variation of the opinion, while the frequency of interactions is increased accordingly. Next, we perform a Taylor expansion of the term encoding the binary interactions in \eqref{eq:weak form opinion}, namely 
    \begin{equation*}
        \varphi(x,w')-\varphi(x,w) = \partial_w \varphi(x,w) (w'-w) + \frac{1}{2} \partial^2_w \varphi(x,w) (w'-w)^2 + \frac{1}{6} \partial^3_w \varphi(x,\tilde{w}) (w'-w)^3, 
    \end{equation*}
    where $\tilde{w} \in (\min\{w,w'\}, \max \{w,w'\})$. Hence, we can rewrite \eqref{eq:weak form opinion} as 
    \begin{equation*}
    \begin{split}
        \frac{\dd}{\dd t} \int_{\Omega \times \mathcal{I}} & \varphi(x,w) f_J(t,x,w) \dd x \dd w \\[2mm] 
        &= \frac{1}{\epsilon\tau}\sum_{J' \in \mathcal{C}} \int_{\Omega^2 \times \mathcal{I}^2} \mathcal{B}(x,y) \partial_w \varphi(x,w) \left< w'-w \right>  f_J(t,x,w) f_{J'}(t,y,w_*) \dd x \dd y \dd w \dd w_* \\[2mm]
        &\ + \frac{1}{2\epsilon\tau}\sum_{J' \in \mathcal{C}} \int_{\Omega^2 \times \mathcal{I}^2} \mathcal{B}(x,y) \partial^2_w \varphi(x,w) \left< (w'-w)^2 \right>  f_J(t,x,w) f_{J'}(t,y,w_*) \dd x \dd y \dd w \dd w_* \\[2mm] 
        &\ + \frac{1}{\epsilon \tau} R_\varphi (f_J),  
    \end{split}
    \end{equation*}
    where $R_\varphi (f_J) = \frac{1}{6}\displaystyle\sum_{J' \in \mathcal{C}} \int_{\Omega^2 \times \mathcal{I}^2} \mathcal{B}(x,y) \partial^3_w \varphi(x,\tilde{w}) \left< (w'-w)^3 \right>  f_J(t,x,w) f_{J'}(t,y,w_*) \dd x \dd y \dd w \dd w_*$. We then substitute $w'=w + G(w,w_*)P(x,y)\lambda (w_*-w) + D(w) \eta_J$ and compute the mean of all random variables, recalling that $\left< \eta_J \right> = 0$. By now taking the limit $\epsilon \to 0$, under the hypothesis that $\left< |\eta_J|^3 \right> < +\infty$, we can check that $R_\varphi (f_J)/\epsilon \to 0$ \cite{CorParTos, Tos}. Integrating back by parts in order to get rid of the derivatives of $\varphi(x,w)$ and assuming that the boundary terms vanish, we formally recover a Fokker--Planck-type equation \cite{CorParTos, ParTos, Tos} describing the evolution of each $f_J$, $J \in \mathcal{C}$, 
    \begin{equation} \label{eq:opinion operator Q}
    \begin{split}
        \partial_t f_J &= \frac{1}{\tau}\left(\lambda \partial_w \left(\sum_{J' \in \mathcal{C}} \mathcal{K}[f_{J'}]  f_J \right) + \frac{\sigma_J^2}{2} \partial^2_w \left(D^2(w) \sum_{J' \in \mathcal{C}} \mathcal{H}[f_{J'}]  f_J\right) \right) \\[2mm]
        & = \frac{1}{\tau} Q_J(\mathbf{f},\mathbf{f}), \quad t\in \R_+,\ x \in \Omega,\ w \in \mathcal{I},
    \end{split}
    \end{equation}
    where we have defined, for any $J \in \mathcal{C}$, the nonlocal operators
    \begin{equation}
        \mathcal{K}[f_J](t,x,w) = \int_{\Omega \times \mathcal{I}} \mathcal{B}(x,y) P(x,y) G(w,w_*) (w-w_*) f_J(t,y,w_*) \dd y \dd w_*, \quad t \in \R_+,\; x \in \Omega,\; w \in \mathcal{I},
    \end{equation}
    and 
    \begin{equation} 
        \mathcal{H}[f_J](t,x) = \int_{\Omega \times \mathcal{I}} \mathcal{B}(x,y) f_J(t,y,w) \dd y \dd w, \quad t \in \R_+,\; x \in \Omega. 
    \end{equation}
    Obviously, the functionals $\mathcal{K}$ and $\mathcal{H}$ must be well defined (i.e., must be finite) for any $t \in \R_+$, $x \in \Omega$ and $w \in \mathcal{I}$. Therefore, we require that the following integrability conditions
    \begin{equation} \label{eq:integrability condition}
        \mathcal{B}(x,\cdot) f_J(t,\cdot, \cdot) \in L^1(\Omega \times \mathcal{I}), \qquad J \in \mathcal{C},
    \end{equation}
    hold for any $t \in \R_+$ and $x \in \Omega$. Moreover, system \eqref{eq:opinion operator Q} is completed with the no-flux boundary conditions
    \begin{equation} \label{eq:BC opinion FP}
    \begin{cases}
        \left. D^2(w) \displaystyle \sum_{J' \in \mathcal{C}} \mathcal{H}[f_{J'}](t,x) f_J(t,x,w) \right|_{w=\pm 1} = 0, \\[6mm] 
        \left. \lambda \displaystyle \sum_{J' \in \mathcal{C}} \mathcal{K}[f_{J'}](t,x,w)  f_J(t,x,w)+\frac{\sigma_J^2}{2} \partial_w \left(D^2(w) \sum_{J' \in \mathcal{C}} \mathcal{H}[f_{J'}](t,x)  f_J(t,x,w)\right) \right|_{w=\pm 1}=0, 
    \end{cases}
    \end{equation}
    holding for any $t \in \R_+$ and $x \in \Omega$, where the first relation is also linked to the property $D(\pm 1) = 0$. Finally, notice that the parameters $\lambda$ and $\sigma_J^2$ (measuring the strength of compromise and self-thinking processes) appear in equation \eqref{eq:opinion operator Q} as coefficients of the drift and diffusion terms, respectively. As a consequence, small values of $\nu_J = \sigma_J^2/\lambda$ correspond to compromise-dominated opinion dynamics, while large values of $\nu_J$ characterize self-thinking-dominated opinion dynamics. 

    As mentioned, the effects of opinion exchanges on the evolution of the relevant macroscopic quantities \eqref{eq:density and mean} and \eqref{eq:total mean} can be inferred from the three Fokker--Planck equations \eqref{eq:opinion operator Q}. The operators $Q_J$ preserve the mass of each compartment due to the second boundary condition of \eqref{eq:BC opinion FP}, hence the number of agents in each class varies only according to the operator $\mathbf{E}$ of \eqref{eq:vectorial model}. On the other hand, the evolution of the mean compartmental opinions depend on both operators, since $\mathbf{Q}$ describes how individuals' opinion changes, while $\mathbf{E}$ can cause a variation of opinion within each compartment by removing agents from it or adding new ones from other classes. Indeed, from the boundary conditions \eqref{eq:BC opinion FP} it follows that, for any $J \in \mathcal{C}$,
    \begin{equation*}
    \begin{split}
        \frac{\dd}{\dd t} \int_{\Omega \times \mathcal{I}} & w f_J(t,x,w) \dd x \dd w = \frac{1}{\tau}\int_{\Omega \times \mathcal{I}} w E_J(\mathbf{f}, \mathbf{f}) \dd x \dd w -\frac{\lambda}{\tau} \sum_{J' \in \mathcal{C}} \int_{\Omega \times \mathcal{I}} \mathcal{K}[f_{J'}] f_J \dd x \dd w \\[2mm] 
        &= \ \frac{1}{\tau}\int_{\Omega \times \mathcal{I}} w E_J(\mathbf{f}, \mathbf{f}) \dd x \dd w \\[2mm] 
        & \quad -\frac{\lambda}{\tau}\sum_{J' \ne J} \int_{\Omega^2 \times \mathcal{I}^2} \mathcal{B}(x,y) P(x,y) (w-w_*) f_J(t,x,w) f_{J'}(t,y,w_*) \dd x \dd y \dd w \dd w_*, 
    \end{split}
    \end{equation*} 
    where $\int_{\Omega \times \mathcal{I}} w E_J(\mathbf{f},\mathbf{f}) \dd x \dd w$ describes the opinion variations due to the epidemiological operator. Summing up the above equations and assuming that the graphon $\mathcal{B}$ and the interaction function $P$ are symmetric, we also deduce the conservation of the total average opinion $m$, since all contributions depending on $\mathbf{E}$ cancel out (by \eqref{eq:epidemiological operators E} and in analogy with the classical SIR system, it holds that $\sum_{J \in \mathcal{C}} E_J(\mathbf{f},\mathbf{f}) \equiv 0$, also implying the conservation of the total mass). 
    
    \begin{remark} \label{rmk1}
        Analogous computations to the ones performed to determine the evolution of the average opinion show that, for any fixed $x \in \Omega$, $\partial_t \sum_{J \in \mathcal{C}}\mathcal{H}[f_J](t,x) = 0$ for any $t \in \R_+$. Therefore, it is enough to require that the integrability conditions \eqref{eq:integrability condition} hold only for the initial distributions $f_J^\init(x,w) = f_J(t=0,x,w)$, $J \in \mathcal{C}$. This conservation property of the functional $\mathcal{H}$ will be important for the analysis conducted in Section \ref{section2}. 
    \end{remark}

\subsection{Macroscopic SIR systems} \label{sec:macro}

    In the following, any solution $\mathbf{f} = (f_J)_{J\in\mathcal{C}}$ of the kinetic SIR system \eqref{eq:vectorial model} satisfying $\mathbf{Q}(\mathbf{f},\mathbf{f}) = 0$ will be termed \emph{local equilibrium} of the kinetic operator and denoted by $\mathbf{f}^\eq = (f_J^\eq)_{J\in\mathcal{C}}$. If we suppose, for the moment, that the opinion exchanges happen on a much faster timescale than that of the epidemic spread, which amounts to consider a regime $\tau \ll 1$, one expects $\mathbf{f}$ to quickly converge toward $\mathbf{f}^\eq$. In this asymptotics, if we choose $\alpha = 0$ in the function $\beta_T$ defined by \eqref{eq:function beta_T}, when the kinetic model evolves around a local equilibrium it can be approximated by a closed macroscopic SIR system prescribing the temporal evolution of the densities $\bm{\rho} = (\rho_J)_{J \in \mathcal{C}}$. Actually, we can say more: integrating system \eqref{eq:vectorial model} over $x \in \Omega$ and $w \in \mathcal{I}$, since the kinetic operators $Q_J$ are mass preserving, one easily checks that the densities $\rho_J$, $J \in \mathcal{C}$, solve the classical SIR equations
    \begin{align} \label{eq:SIR}
    \left\{
    \begin{aligned}
        & \frac{\dd }{\dd t} \rho_S = - \beta \rho_S \rho_I, \\[2mm] 
        & \frac{\dd }{\dd t} \rho_I = \beta \rho_S \rho_I - \gamma \rho_I, \\[2mm] 
        & \frac{\dd }{\dd t} \rho_R = \gamma \rho_I,
    \end{aligned}
    \right.
    \end{align}
    and we can thus infer several analytical properties on the quantities $(\rho_J)_{J \in \mathcal{C}}$. For this, let us introduce the basic reproduction number $\mathcal{R}_0 = \beta/\gamma$ \cite{vandenDriWat}, which can be interpreted biologically as the expected number of cases produced by one infective agent in a large population where all other individuals are susceptible to the infection. First of all, there exists a unique equilibrium point $\bm{\rho}^\infty = (\rho_J^\infty)_{J \in \mathcal{C}}$ depending only on the given initial conditions $\rho_J^\init = \rho_J(0)$, $J \in \mathcal{C}$, which is globally asymptotically stable. In particular, $\rho_S^\infty > 0$, $\rho_I^\infty = 0$, and $\rho_R^\infty = 1 - \rho_S^\infty$, where $\rho_S^\infty$ is the unique solution in the interval $(0,1/\rho_R(0))$ of the equation
    \begin{equation}
        \log \frac{\rho_S^\infty}{\rho_S^\init} + \mathcal{R}_0 (\rho_S^\init + \rho_I^\init - \rho_S^\infty )= 0.
    \end{equation}
    Moreover, $\rho_S$ and $\rho_R$ are respectively a decreasing and an increasing function of time, which imply the upper and lower bounds $\rho_S^\infty \leq \rho_S(t) \leq \rho_S^\init$ and $\rho_R^\init \leq \rho_R(t) \leq \rho_R^\infty$ for any $t \in \R_+$. At last, the behavior of $\rho_I$ depends on the parameter $\mathcal{R}_0$, so that whenever $\mathcal{R}_0 < 1$ one sees that $0 \leq \rho_I(t) \leq \rho_I^\init$ is a decreasing function of time converging exponentially to zero at a rate at least $\mathcal{O}(\exp(-\gamma(1-\mathcal{R}_0)t))$, while for $\mathcal{R}_0 > 1$ the density initially increases up to a maximum value $\rho_I^{\max}$, given by
    \begin{equation}
        \rho_I^{\max} = - \frac{1}{\mathcal{R}_0} (\log \mathcal{R}_0 + \log \rho_S^\init + 1) + \rho_S^\init + \rho_I^\init,
    \end{equation}
    before decaying once again to zero exponentially fast, which readily implies the bounds $0 \leq \rho_I(t) \leq \rho_I^{\max} < 1$ for any $t \in \R_+$.

    On the other hand, if we set $\alpha = 1$ in \eqref{eq:function beta_T}, integrating system \eqref{eq:vectorial model} in $x \in \Omega$ and $w \in \mathcal{I}$, we obtain a generalization of the classical SIR system \eqref{eq:SIR} having the form
    \begin{equation} \label{eq:generalized SIR}
    \left\{
    \begin{aligned}
        & \frac{\dd}{\dd t} \rho_S = - \beta (1-m_S) (1-m_I) \rho_S \rho_I,  \\[2mm] 
        & \frac{\dd}{\dd t} \rho_I = \beta (1-m_S) (1-m_I) \rho_S \rho_I - \gamma \rho_I, \\[2mm] 
        & \frac{\dd}{\dd t} \rho_R = \gamma \rho_I,
    \end{aligned}
    \right.
    \end{equation}
    which is not closed since it does not contain any information on the evolution of the means $m_S$ and $m_I$, but it highlights the influence of the agents' opinion on the epidemic spread. We can therefore introduce the time-dependent quantity, hereafter named \emph{effective reproduction number}, 
    \begin{equation} \label{eq:effective reproduction number}
        \mathcal{R}_\textnormal{eff}(t) = \frac{\beta}{\gamma} (1-m_S(t)) (1-m_I(t)) \rho_S(t), 
    \end{equation}
    which, from a mathematical point of view, has the same effect as the basic reproduction number, although biologically they are not equivalent. We note however that one can still infer some properties of the solutions to \eqref{eq:generalized SIR} as long as $|m_S| \leq 1$ and $|m_I| \leq 1$. Indeed, under these conditions, as in the classical setting for system \eqref{eq:SIR} we can infer that solutions $\rho_J$ starting from a nonnegative initial datum remain nonnegative, hence the conservation of mass guarantees the global bounds $0 \leq \rho_J \leq 1$, $J \in \mathcal{C}$. Moreover, since $\rho_S$ is non-increasing and $\rho_R$ is non-decreasing, they both converge to some limits $\rho_S^\infty, \rho_R^\infty \in [0,1]$. Conservation of mass then ensures that also $\rho_I$ converges to some limit $\rho_I^\infty \in [0,1]$ and the latter must be zero since $\rho_I \in L^1(\R_+)$, which can be easily checked by integrating in time the third equation of \eqref{eq:generalized SIR}. In particular, the value of $\rho_S^\infty > 0$ is now determined by the full evolution of the means $m_S$ and $m_I$, while obviously $\rho_R^\infty = 1 - \rho_S^\infty$. These considerations will become meaningful when studying the asymptotic stability of our kinetic system in Section \ref{sec:trends}.

\subsection{Local equilibria} \label{section_local_eq}

    In order to understand the long-time behavior of our kinetic SIR system, we first have to analyze the equilibria of the Fokker--Planck equation \eqref{eq:opinion operator Q}. For the sake of simplicity, we fix $J \in \mathcal{C}$ and focus our attention only on one of the Fokker--Planck equations \eqref{eq:opinion operator Q}, with an initial condition of unit mass. Moreover, we assume that $G \equiv 1$ and $D(w)=\sqrt{1-w^2}$. Let us define, for $t \in \R_+$ and $x \in \Omega$, the quantities 
    \begin{equation} 
    \begin{split}
        \rho_\mathcal{B}(t,x) & = \sum_{J \in \mathcal{C}} \int_{\Omega \times \mathcal{I}}  \mathcal{B}(x,y) f_J(t,y,w) \dd y \dd w, \\[4mm]
        \rho_{\mathcal{B},P}(t,x) & = \sum_{J \in \mathcal{C}} \int_{\Omega \times \mathcal{I}}  \mathcal{B}(x,y) P(x,y) f_J(t,y,w) \dd y \dd w, \\[4mm]  
        m_{\mathcal{B},P}(t,x) & = \frac{1}{\rho_{\mathcal{B},P}(t,x)}\sum_{J \in \mathcal{C}} \int_{\Omega \times \mathcal{I}}  \mathcal{B}(x,y) P(x,y) w f_J(t,y,w) \dd y \dd w. 
    \end{split}
    \end{equation} 
    Recall in particular from Remark \ref{rmk1} that $\rho_\mathcal{B}$ and $\rho_{\mathcal{B},P}$ are conserved over time. The steady state $B_J = B_J(t,x,w)$ of \eqref{eq:opinion operator Q} must satisfy, for any $t \in \R_+$, $x \in \Omega$ and $w \in \mathcal{I}$, 
    \begin{equation}
        \lambda \rho_{\mathcal{B},P}(x)\left(w-m_{\mathcal{B},P}(t,x)\right)B_J(t,x,w) + \frac{\sigma_J^2}{2} \rho_\mathcal{B}(x) \partial_w \left( (1-w^2) B_J(t,x,w)  \right) = 0, 
    \end{equation}
    and after developing the derivative with respect to $w$, it is easy to see that the Fokker--Planck model relaxes toward beta distributions (with respect to $w$) of the form
    \begin{equation} \label{eq:equilibria_complex}
        B_J(t,x,w)= c_J(t,x) (1+w)^{-1+\frac{\rho_{\mathcal{B},P}(x)}{\rho_\mathcal{B}(x)} \frac{1 + m_{\mathcal{B},P}(t,x)}{\nu_J}} (1-w)^{-1+\frac{\rho_{\mathcal{B},P}(x)}{\rho_\mathcal{B}(x)} \frac{1 - m_{\mathcal{B},P}(t,x)}{\nu_J}}, \quad t \in \R_+,\; x \in \Omega,\; w \in \mathcal{I},
    \end{equation}
    where $c_J(t,x) > 0$ is a suitable normalization function guaranteeing that $\int_{\Omega \times \mathcal{I}} B_J(t,x,w) \dd x \dd w = 1$. 

    Unfortunately, the structure of the equilibrium states \eqref{eq:equilibria_complex} cannot be easily interpreted. For this reason, in Section \ref{connectivity} we will show how to characterize them in a more explicit way by considering a slight simplification of the Fokker--Planck operators, which will allow us to analyze the long-time behavior of the full kinetic SIR model.  

\subsection{Analytical properties} \label{section2}

    We conclude this section with an investigation of the main analytical properties of the kinetic SIR model \eqref{eq:vectorial model}, where the Fokker--Planck operators $Q_J$ have been defined in \eqref{eq:opinion operator Q}. Throughout this study, in order to justify our calculations we will tacitly assume that the integrability condition \eqref{eq:integrability condition} holds (so that the operators $\mathcal{K}$ and $\mathcal{H}$ are well defined) and that the functions $G$ and $D$ are regular enough.  

    We begin by showing that solutions to \eqref{eq:vectorial model} starting from a nonnegative initial datum remain nonnegative over time. We restrict this analysis to the simplified cases $\alpha = 0,1$ in \eqref{eq:function beta_T}.
    The idea is to separately study the epidemic operator $\mathbf{E}$ and the kinetic one $\mathbf{Q}$. In particular, in order to deal with the operators $Q_J$ we need the following preliminary result (which will be also useful to recover the uniqueness of solutions) concerning an $L^1$ comparison principle for solutions to the Fokker--Planck equation \eqref{eq:opinion operator Q}. 
    \begin{lemma} \label{lemma1}
        Let $(f_J)_{J \in \mathcal{C}}$ be a solution to the Fokker--Planck system \eqref{eq:opinion operator Q}
        satisfying the no-flux boundary conditions \eqref{eq:BC opinion FP}, and posed on $\R_+ \times \Omega \times \mathcal{I}$ with initial conditions $f_J(0,x,w) = f_J^\init(x,w)$, $J \in \mathcal{C}$. If $f^\init_J \in L^1(\Omega \times \mathcal{I})$, then the $L^1(\Omega \times \mathcal{I})$ norm of $f_J$ is non-increasing for $t \in \R_+$. 
    \end{lemma}
    \begin{proof}
        Consider, for some parameter $\eps > 0$, a regularized increasing approximation of the sign function $\textrm{sign}_\eps(u)$, $u \in \R$. We introduce a regularization $|f_J|_\eps(t,x,w)$ of $|f_J|(t,x,w)$ by defining it as the primitive of $\textrm{sign}_\eps(f_J)(t,x,w)$, for any $t \in \R_+$, $x \in \Omega$ and $w \in \mathcal{I}$. Multiplying by $\textrm{sign}_\eps(f_J)$ both sides of equation \eqref{eq:opinion operator Q}, integrating with respect to $x \in \Omega$ and $w \in \mathcal{I}$ and applying integration by parts together with the no-flux boundary conditions \eqref{eq:BC opinion FP}, we can successively compute
        \begin{equation*}
        \begin{split}
            \frac{\dd}{\dd t} \int_{\Omega \times \mathcal{I}} |f_J|_\eps(t,x,w) \dd w =& - \lambda \int_{\Omega \times \mathcal{I}} \textrm{sign}_\eps'(f_J) f_J \partial_w f_J \sum_{J' \in \mathcal{C}} \mathcal{K}[f_{J'}] \dd x \dd w\\ 
            & \quad - \frac{\sigma_J^2}{2} \int_{\Omega \times \mathcal{I}} \textrm{sign}_\eps'(f_J) \partial_w f_J \partial_w \left( D^2(w) f_J\right) \sum_{J' \in \mathcal{C}} \mathcal{H}[f_{J'}] \dd x \dd w \\[2mm] 
            =& - \lambda \int_{\Omega \times \mathcal{I}} \textrm{sign}_\eps'(f_J) f_J \partial_w f_J \sum_{J' \in \mathcal{C}} \mathcal{K}[f_{J'}] \dd x \dd w \\ 
            & \quad -\frac{\sigma_J^2}{2} \int_{\Omega \times \mathcal{I}}  \textrm{sign}_\eps'(f_J) f_J \partial_w f_J \partial_w  D^2(w) \sum_{J' \in \mathcal{C}} \mathcal{H}[f_{J'}] \dd x \dd w \\ 
            & \qquad -\frac{\sigma_J^2}{2} \int_{\Omega \times \mathcal{I}} \textrm{sign}_\eps'(f_J) \left(\partial_w f_J\right)^2 D^2(w) \sum_{J' \in \mathcal{C}} \mathcal{H}[f_{J'}] \dd x \dd w. 
        \end{split}
        \end{equation*}
        Noticing that
        \begin{equation*}
            \textrm{sign}'_\eps(f_J)f_J \partial_w f_J = \partial_w (\textrm{sign}_\eps(f_J)f_J-|f_J|_\eps),
        \end{equation*}
        we can integrate back by parts the first two integrals and observe that they vanish as $\eps \to 0$, since by construction $\textrm{sign}_\eps(f_J)f_J - |f_J|_\eps \underset{\eps \to 0}{\longrightarrow} 0$ for a.e. $x \in \Omega$ and $w \in \mathcal{I}$. The thesis follows from the fact that the last integral is nonnegative.  
    \end{proof}
    
    We are now able to prove the following property of nonnegativity-preservation.  
    
    \begin{proposition}[Nonnegativity] \label{prop:nonnegativity}
        Let $\mathbf{f}$ be a solution to the SIR kinetic model \eqref{eq:vectorial model} with Fokker--Planck operator $\mathbf{Q}$, defined componentwise by \eqref{eq:opinion operator Q} and satisfying the no-flux boundary conditions \eqref{eq:BC opinion FP}. Assume moreover that $\alpha = 0,1$ in \eqref{eq:function beta_T} and, if $\alpha=1$, that $|m_J(t)| \leq 1$, $J = S, I$, for any $t \in \R_+$. Consider for any $J \in \mathcal{C}$ an initial distribution $f_J^\init \in L^1(\Omega \times \mathcal{I})$ such that $f^\init_J(x,w) \geq 0$ for a.e. $x \in \Omega$ and $w \in \mathcal{I}$. Then,  $f_J(t,x,w) \geq 0$ for a.e. $x \in \Omega$, $w \in \mathcal{I}$, and for any $t \in \R_+$. 
    \end{proposition}
    
    \begin{proof}
        Define the negative part of each distribution $f_J$, $J \in \mathcal{C}$, via the regularization introduced in the proof of Lemma \ref{lemma1}, namely 
        \begin{equation*}
            f_{J, \eps}^-(t,x,w)=\frac{1}{2} (|f_J|_\eps (t,x,w) - f_J(t,x,w)), \qquad t \in \R_+,\; x \in \Omega, \;  w \in \mathcal{I}. 
        \end{equation*}
        We start with the case $\alpha = 0$. Multiplying the first equation of kinetic SIR model \eqref{eq:vectorial model} on both sides by $\textrm{sign}_\eps(f_S)$ and integrating in $x \in \Omega$ and $w \in \mathcal{I}$, from Lemma \ref{lemma1} and since $\alpha=0$ we initially deduce that
        \begin{equation*}
        \begin{split}
            \frac{\dd}{\dd t} \int_{\Omega \times \mathcal{I}} |f_{S}|_\eps(t,x,w) \dd x \dd w & \leq - \beta \int_{\Omega \times \mathcal{I}} |f_S|_{\eps} \dd x \dd w \int_{\Omega \times \mathcal{I}} f_I(t,y,w_*) \dd y \dd w_* \\[2mm]
            & = - \beta \rho_I(t) \int_{\Omega \times \mathcal{I}} |f_S|_{\eps} \dd x \dd w ,
        \end{split} 
        \end{equation*}
        where we have used the fact that the densities $(\rho_J)_{J \in \mathcal{C}}$ solve the macroscopic SIR system \eqref{eq:SIR}. Therefore 
        \begin{equation*}
        \begin{split}
            2\frac{\dd}{\dd t} \int_{\Omega \times \mathcal{I}} f_{S, \eps}^-(t,x,w) \dd x \dd w &= \frac{\dd}{\dd t} \int_{\Omega \times \mathcal{I}} |f_S|_\eps(t,x,w) \dd x \dd w - \frac{\dd}{\dd t} \int_{\Omega \times \mathcal{I}} f_S(t,x,w) \dd x \dd w \\[3mm] 
            & \leq - \beta \rho_I(t) \int_{\Omega \times \mathcal{I}} |f_S|_{\eps} \dd x \dd w + \beta \rho_S (t) \rho_I (t) \\[2mm] 
            & = \beta \rho_I (t) \left( \rho_S(t) - \int_{\Omega \times \mathcal{I}} |f_S|_{\eps} \dd x \dd w  \right), 
        \end{split}
        \end{equation*}
        and taking the limit $\eps \to 0$, the nonnegativity of $f_S$ follows because $\rho_S(t) \leq \int_{\Omega \times \mathcal{I}} |f_S|(t,x,w) \dd x \dd w$. 
        Next, using that $f_S$ is nonnegative,  
        \begin{equation*}
        \begin{split}
            \frac{\dd}{\dd t} \int_{\Omega \times \mathcal{I}} |f_I|_{\eps}(t,x,w) \dd x \dd w \leq \beta \rho_I(t) \rho_S(t) - \gamma \int_{\Omega \times \mathcal{I}} |f_I|_{\eps} \dd x \dd w ,
        \end{split} 
        \end{equation*}
        which implies that 
        \begin{equation*}
        \begin{split}
            \frac{\dd}{\dd t} \int_{\Omega \times \mathcal{I}} f_{I, \eps}^-(t,x,w) \dd x \dd w \leq \frac{\gamma}{2} \left( \rho_I(t) - \int_{\Omega \times \mathcal{I}} |f_I|_{\eps} \dd x \dd w \right), 
        \end{split}
        \end{equation*}
        and taking again limit the $\eps \to 0$, the nonnegativity of $f_I$ follows similarly. Finally, the nonnegativity of $f_R$ is trivial since $f_I$ is nonnegative.

        The case $\alpha = 1$ is similar. We initially notice that given the hypothesis on the boundedness of $m_S$ and $m_I$, it is straightforward to check that the generalized SIR system \eqref{eq:generalized SIR} preserves the nonnegativity of solutions, hence the conservation of total mass implies that the $\rho_J$, $J \in \mathcal{C}$, remain bounded between $0$ and $1$. We can thus proceed as in the previous case to recover the following estimate on $f_{S,\eps}^-$: 
        \begin{equation*}
            \frac{\dd}{\dd t} \int_{\Omega \times \mathcal{I}} f_{S, \eps}^-(t,x,w) \dd x \dd w  \leq 2 \beta \rho_I (t) (1-m_I(t)) \int_{\Omega \times \mathcal{I}} f_{S,\eps}^- \dd x \dd w.,
        \end{equation*}
        using that $1-w \leq 2$. An application of Gr\"onwall's inequality then gives
        \begin{equation*}
        	    \frac{\dd}{\dd t} \| f_{S, \eps}^-(t,\cdot, \cdot) \|_{L^1(\Omega \times \mathcal{I})} \leq \| f_{S, \eps}^-(0,\cdot, \cdot) \|_{L^1(\Omega \times \mathcal{I})} \exp\left( 2 \int_0^t \rho_I(s) (1 - m_I(s)) \dd s \right), 
        \end{equation*}
        and taking the limit $\eps \to 0$ we conclude once again with the positivity of $f_S$, given a nonnegative initial datum $f_S^\init$. Using that $f_S \geq 0$ and the boundedness of $m_S$ and $m_I$, the nonnegativity of $f_I$ can then be assessed in a similar way since
        \begin{equation*}
            \frac{\dd}{\dd t} \int_{\Omega \times \mathcal{I}} |f_I|_\eps(t,x,w) \dd x \dd w \leq \beta \rho_I(t) \rho_S(t) (1 - m_I(t)) (1 - m_S(t)) - \gamma \int_{\Omega \times \mathcal{I}} |f_I|_\eps(t,x,w) \dd x \dd w,
        \end{equation*}
        hence
        \begin{equation*}
            \frac{\dd}{\dd t} \int_{\Omega \times \mathcal{I}} f_{I,\eps}^-(t,x,w) \dd x \dd w \leq - \frac{\gamma}{2} \int_{\Omega \times \mathcal{I}} f_{I,\eps}^-(t,x,w) \dd x \dd w,
        \end{equation*}
        thanks to the cancellation of the coupling epidemiological term. Again, the nonnegativity of $f_R$ can be proved in the same way.
    \end{proof}
    
    \begin{remark} \label{rem:positivity}
            More generally, working with regularized solutions one can infer the nonnegativity of the $(f_J)_{J \in \mathcal{C}}$ using a strong maximum principle. Indeed, let us consider smooth enough solutions and argue by contradiction on the existence of a time $\bar{t} > 0$ such that $f_J(t,x,w) \geq 0$, $J \in \mathcal{C}$, for any $t \in [0,\bar{t})$, $x \in \Omega$, $w \in \mathcal{I}$, and of a corresponding couple $(\bar{x}, \bar{w}) \in \Omega \times \mathcal{I}$ where (say for $J = S$)
        \begin{equation*}
            f_S(\bar{t}, \bar{x}, \bar{w}) = 0, \quad \partial_w f_S(\bar{t}, \bar{x}, \bar{w}) = 0, \quad \partial_w^2 f_S(\bar{t}, \bar{x}, \bar{w}) \geq 0, \quad \partial_t f_S(\bar{t}, \bar{x}, \bar{w}) < 0,
        \end{equation*}
        meaning that at $(\bar{t}, \bar{x}, \bar{w})$ the distribution $f_S$ crosses the value zero and becomes negative. A straightforward computation then shows (let us fix $\tau = 1$ for simplicity) that
        \begin{equation*}
            \begin{split}
                \partial_t f_S(\bar{t}, & \bar{x}, \bar{w})= -f_S(\bar{t}, \bar{x}, \bar{w}) \int_{\Omega \times \mathcal{I}} \beta_T(\bar{w},w_*) f_I(\bar{t},y,w_*) \dd y \dd w_* \\[4mm]
                & + \lambda \partial_w \sum_{J' \in \mathcal{C}} \mathcal{K}[f_{J'}] (\bar{t}, \bar{x}, \bar{w})  f_S(\bar{t}, \bar{x}, \bar{w}) + \lambda \sum_{J' \in \mathcal{C}} \mathcal{K}[f_{J'}](\bar{t}, \bar{x}, \bar{w}) \partial_w f_S(\bar{t}, \bar{x}, \bar{w}) \\[2mm] 
                & + \frac{\sigma_J^2}{2} \sum_{J' \in \mathcal{C}} \mathcal{H}[f_{J'}](\bar{t}, \bar{x}) \Big( \partial^2_w D^2(w) f_S(\bar{t}, \bar{x}, \bar{w}) + 2 \partial_w D^2(w) \partial_w f_S(\bar{t}, \bar{x}, \bar{w}) + D^2(w) \partial^2_w f_S(\bar{t}, \bar{x}, \bar{w}) \Big) \\[6mm]
                & \qquad \geq 0,
            \end{split}
        \end{equation*}
        by definition of the point $(\bar{t}, \bar{x}, \bar{w})$. This contradicts the possibility that $\partial_t f_S(\bar{t}, \bar{x}, \bar{w}) < 0$, so $f_S$ must remain nonnegative. The same reasoning can be carried out to prove that $f_I$ stays nonnegative and thus the nonnegativity of $f_R$ follows too.
    \end{remark}
    
    \begin{remark}
        As soon as a solution $\mathbf{f}$ to the SIR kinetic model \eqref{eq:vectorial model} is nonnegative, then the mass of each compartment is always bounded between 0 and 1. Indeed, the nonnegativity implies that $\rho_J(t) \geq 0$ for any $t \in \R_+$ and $J \in \mathcal{C}$, from which follows that $\rho_J(t) \leq 1$ for any $t \in \R_+$ and $J \in \mathcal{C}$ since we have conservation of the total mass. Moreover, notice that the nonnegativity of each $f_J$ also gives
        \begin{equation*}
            0 \leq \int_{\Omega \times \mathcal{I}} (1-w) f_J(t,x,w) \dd x \dd w \leq 2 \rho_J(t),
        \end{equation*}
        hence $0 \leq 1-m_J(t) \leq 2$ by direct computation of the integral, and we recover the condition that $m_J \in [-1,1]$, $J \in \mathcal{C}$.
    \end{remark}

    Since we have conservation of total mass, from the nonnegativity of solutions we directly infer their $L^1(\Omega \times \mathcal{I})$ regularity. Indeed, suppose that $f_J^\init \in L^1(\Omega \times \mathcal{I})$ for any $J \in \mathcal{C}$, and that they are nonnegative. Since from Proposition \ref{prop:nonnegativity} we know that $f_J(t,x,w) \geq 0$ for a.e. $x \in \Omega$, $w \in \mathcal{I}$, and for any $t \in \R_+$, then $||f_J||_{L^1(\Omega \times \mathcal{I})} =\int_{\Omega \times \mathcal{I}} f_J(t,x,w) \dd x \dd w$ for any $t \in \R_+$. Therefore summing up the three equations of system \eqref{eq:vectorial model} and integrating with respect to $x \in \Omega$ and $w \in \mathcal{I}$, we get that for any $t \in \R_+$ 
    \begin{equation*}
        \sum_{J \in \mathcal{C}} ||f_J||_{L^1(\Omega \times \mathcal{I})} = \sum_{J \in \mathcal{C}} ||f_J^\init||_{L^1(\Omega \times \mathcal{I})},
    \end{equation*}
    because the contributions of the operator $\mathbf{E}$ cancel out and the Fokker--Planck operators $Q_J$ are mass preserving. In conclusion, $f_J \in L^1(\Omega \times \mathcal{I})$ for any $J \in \mathcal{C}$ and any $t \in \R_+$.  

    One possible strategy to prove additional $L^q(\Omega \times \mathcal{I})$ regularity of solutions for any $q \in [2,\infty)$ consists in rewriting the diffusion term of the Fokker--Planck equation \eqref{eq:opinion operator Q} in such a way that it exhibits a more manageable self-adjoint structure, and the temporal evolution of $||f_J||_{L^q(\Omega \times \mathcal{I})}$ can be dealt with using suitable integration by parts. The price to pay is to impose additional no-flux boundary conditions ensuring enough regularity at the boundaries. The hypothesis that $q \geq 2$ will be needed to handle derivatives of the form $\partial_w f_J^{q-1}$. However, since the set $\Omega \times \mathcal{I}$ is bounded, if $f_J^\init \in L^q(\Omega \times \mathcal{I})$, $J \in \mathcal{C}$, for some $q \in (1,2)$, then $f_J^\init \in L^1(\Omega \times \mathcal{I})$, $J \in \mathcal{C}$, and we can at least obtain the $L^1(\Omega \times \mathcal{I})$ regularity of solutions. We will have to assume that $\beta_T \in L^\infty (\mathcal{I}^2)$, but we point out that this hypothesis is not restrictive and the modeling choice \eqref{eq:function beta_T} satisfies it for any $\alpha \geq 0$. 
    
    \begin{proposition} [Regularity] \label{prop:regularity}
        Let $q \in [2,+\infty)$ and let $\mathbf{f}$ be a solution to the SIR kinetic system \eqref{eq:vectorial model}. Suppose that the $f_J$, $J \in \mathcal{C}$, are nonnegative. Assume that $\beta_T \in L^\infty (\mathcal{I}^2)$ and $\mathcal{H}[f_J] \in L^\infty (\R_+ \times \Omega)$ for any $J \in \mathcal{C}$. Furthermore, we impose that for any $J \in \mathcal{C}$ the additional no-flux boundary conditions 
        \begin{align} \label{eq:additional BC for regularity opinion FP}
        \left\{
        \begin{aligned}
            & \left. D^2(w) f_J^{q-1}(t,x,w) \partial_w f_J(t,x,w) \right|_{w=\pm 1} = 0, \\[2mm] 
            & \left. f_J(t,x,w)  \right|_{w= \pm 1}=0, 
        \end{aligned}
        \right.
        \end{align}
        hold for any $t \in \R_+$ and $x \in \Omega$. If $f_J^\init \in L^q(\Omega \times \mathcal{I})$, $J \in \mathcal{C}$, then $f_J \in L^q(\Omega \times \mathcal{I})$, $J \in \mathcal{C}$ for any $t \in \R_+$. More precisely, we have the following estimate
        \begin{equation*}
            ||f_J(t,\cdot,\cdot)||_{L^q(\Omega \times \mathcal{I})} \leq 3^{\frac{1}{q}} \max_{J \in \mathcal{C}} ||f_J^\init||_{L^q(\Omega \times \mathcal{I})} \exp(\eta_q t), \quad t \in \R_+,
        \end{equation*}
        where the constants $\eta_q > 0$ are equi-bounded with respect to $q$. 
    \end{proposition}
    
    \begin{proof}
        Fix $q \in [2,\infty)$ and multiply each equation of system \eqref{eq:vectorial model} by $f_J^{q-1}$, with respect to the corresponding index $J \in \mathcal{C}$. After integrating in $x \in \Omega$ and $w \in \mathcal{I}$ we initially obtain 
        \begin{align*}
        \left\{
        \begin{aligned}
            & \frac{1}{q} \frac{\dd}{\dd t} ||f_S||^q_{L^q(\Omega \times \mathcal{I})} = - \int_{\Omega \times \mathcal{I}} f_S^q \left( \int_{\Omega \times \mathcal{I}} \beta_T(w,w_*) f_I(t,y,w_*) \dd y \dd w_* \right) \dd x \dd w \\
            & \hspace{9cm} + \frac{1}{\tau}\int_{\Omega \times \mathcal{I}} Q_S(\mathbf{f},\mathbf{f}) f_S^{q-1} \dd x \dd w, \\[2mm]
            & \frac{1}{q} \frac{\dd}{\dd t} ||f_I||^q_{L^q(\Omega \times \mathcal{I})} = \int_{\Omega \times \mathcal{I}} f_S f_I^{q-1} \left( \int_{\Omega \times \mathcal{I}} \beta_T(w,w_*) f_I(t,y,w_*) \dd y \dd w_* \right) \dd x \dd w - \gamma ||f_I||^q_{L^q(\Omega \times \mathcal{I})} \\
            & \hspace{9cm} + \frac{1}{\tau}\int_{\Omega \times \mathcal{I}} Q_I(\mathbf{f},\mathbf{f}) f_I^{q-1} \dd x \dd w, \\[2mm] 
            & \frac{1}{q} \frac{\dd}{\dd t} ||f_R||^q_{L^q(\Omega \times \mathcal{I})} = \gamma \int_{\Omega \times \mathcal{I}} f_I f_R^{q-1} \dd x \dd w + \frac{1}{\tau}\int_{\Omega \times \mathcal{I}} Q_R(\mathbf{f},\mathbf{f}) f_R^{q-1} \dd x \dd w.  
        \end{aligned}
        \right. 
        \end{align*}
        The idea is to control (up to a constant) each term on the right-hand side of the above equations with the quantities $\|f_J\|_{L^q(\Omega \times \mathcal{I})}^q$, $J \in \mathcal{C}$ and then conclude by applying the Gr\"onwall's Lemma. 
        
        Let us start from the epidemiological operators. We first derive the straightforward control
        \begin{equation*}
            \int_{\Omega \times \mathcal{I}} f_S^q(t,x,w) \left( \int_{\Omega \times \mathcal{I}} \beta_T(w,w_*) f_I(t,y,w_*) \dd y \dd w_* \right) \dd x \dd w \leq ||\beta_T||_{L^\infty (\mathcal{I}^2)} ||f_I||_{L^1(\Omega \times \mathcal{I})} ||f_S||^q_{L^q(\Omega \times \mathcal{I})}, 
        \end{equation*}
        where $||f_I||_{L^1(\Omega \times \mathcal{I})} = \rho_I \in [0,1]$. Using H\"older's and Young's inequalities we then infer the bound 
        \begin{equation*}
        \begin{split}
            \int_{\Omega \times \mathcal{I}} f_S(t,x,w) f_I^{q-1}(t,x,w) \bigg( \int_{\Omega \times \mathcal{I}} & \beta_T(w,w_*) f_I(t,y,w_*) \dd y \dd w_* \bigg)\dd x \dd w \\[2mm]
            &\leq ||\beta_T||_{L^\infty (\mathcal{I}^2)} ||f_I||_{L^1(\Omega \times \mathcal{I})} \left( \frac{1}{q} ||f_S||^q_{L^q(\Omega \times \mathcal{I})} + \frac{q-1}{q} ||f_I||^q_{L^q(\Omega \times \mathcal{I})} \right).
        \end{split}
        \end{equation*}
        Finally, H\"older's inequality also implies that 
        \begin{equation*}
            \gamma \int_{\Omega \times \mathcal{I}} f_I(t,x,w) f_R^{q-1}(t,x,w) \dd x \dd w \leq \gamma \left(\frac{1}{q} ||f_I||^q_{L^q(\Omega \times \mathcal{I})} + \frac{q-1}{q} ||f_R||^q_{L^q(\Omega \times \mathcal{I})} \right). 
        \end{equation*}

        Moving now to the Fokker--Planck operators, we observe that 
        \begin{equation*}
        \begin{split}
            \int_{\Omega \times \mathcal{I}} Q_J(\mathbf{f},\mathbf{f}) f_J^{q-1}(t,x,w) \dd x \dd w =& \overbrace{ \frac{\sigma_J^2}{2} \int_{\Omega \times \mathcal{I}}  \partial_w \left(D^2(w) \sum_{J' \in \mathcal{C}} \mathcal{H}[f_{J'}] \partial_w f_J \right) f_J^{q-1} \dd x \dd w }^{\mathcal{T}_1} \\[2mm] 
            & \quad + \underbrace{\int_{\Omega \times \mathcal{I}} \partial_w \left(\left( \lambda \sum_{J' \in \mathcal{C}} \mathcal{K}[f_{J'}] + \frac{\sigma_J^2}{2} \partial_w D^2(w) \right)f_J\right)  f_J^{q-1} \dd x \dd w  }_{\mathcal{T}_2}. 
        \end{split}
        \end{equation*}
        Integrating by parts $\mathcal{T}_1$ and using the first of the boundary conditions \eqref{eq:additional BC for regularity opinion FP} allows us to show that this term is nonpositive, since 
        \begin{equation*}
            \mathcal{T}_1 = - \frac{\sigma_J^2}{2} (q-1) \int_{\Omega \times \mathcal{I}}  D^2(w) \sum_{J' \in \mathcal{C}} \mathcal{H}[f_{J'}] \left(\partial_w f_J \right)^2 f_J^{q-2} \dd x \dd w \leq 0. 
        \end{equation*}
        The second term $\mathcal{T}_2$ can be treated in two different ways. At first we perform an integration by parts and use the boundary conditions \eqref{eq:additional BC for regularity opinion FP} to get 
        \begin{equation*}
        \begin{split}
            \mathcal{T}_2 &= - \int_{\Omega \times \mathcal{I}} \left( \lambda \sum_{J' \in \mathcal{C}} \mathcal{K}[f_{J'}] + \frac{\sigma_J^2}{2} \partial_w D^2(w) \right)f_J \partial_w f_J^{q-1} \dd x \dd w \\[2mm] 
            & = - (q-1) \int_{\Omega \times \mathcal{I}} \left( \lambda \sum_{J' \in \mathcal{C}} \mathcal{K}[f_{J'}] + \frac{\sigma_J^2}{2} \partial_w D^2(w) \right)f_J^{q-1}  \partial_w f_J \dd x \dd w. 
        \end{split}
        \end{equation*}
        Secondly, we compute the derivative with respect to $w$ to obtain 
        \begin{equation*}
        \begin{split}
            \mathcal{T}_2 =& \int_{\Omega \times \mathcal{I}} \left( \lambda \sum_{J' \in \mathcal{C}} \mathcal{K}[f_{J'}] + \frac{\sigma_J^2}{2} \partial_w D^2(w) \right)f_J^{q-1}  \partial_w f_J \dd x \dd w \\[2mm] 
            &\qquad \qquad + \int_{\Omega \times \mathcal{I}} \partial_w \left(\lambda \sum_{J' \in \mathcal{C}} \mathcal{K}[f_{J'}] + \frac{\sigma_J^2}{2} \partial_w D^2(w)\right) f_J^q \dd x \dd w. 
        \end{split}  
        \end{equation*}
        Therefore we can split $\mathcal{T}_2$ as $\mathcal{T}_2=\frac{1}{q}\mathcal{T}_2+\frac{q-1}{q}\mathcal{T}_2$ to finally write 
        \begin{equation*}
            \mathcal{T}_2 = \frac{q-1}{q} \int_{\Omega \times \mathcal{I}} \partial_w \left(\lambda \sum_{J' \in \mathcal{C}} \mathcal{K}[f_{J'}] + \frac{\sigma_J^2}{2} \partial_w D^2(w)\right) f_J^q \dd x \dd w. 
        \end{equation*}
        Hence, expanding the derivative with respect to $w$ of the operator $\mathcal{K}$, a direct estimate shows that (recall Remark \ref{rmk1})
        \begin{equation*}
        \begin{split}
            \int_{\Omega \times \mathcal{I}} Q_J(\mathbf{f},\mathbf{f}) f_J^{q-1}(t,x,w) \dd x \dd w & \leq \frac{q-1}{q}\bigg( \lambda \left(2 \left|\left|\partial_w G\right|\right|_{L^\infty (\mathcal{I}^2)} ||\mathcal{H}[f_J]||_{L^\infty(\R_+ \times \Omega)}+||\mathcal{H}[f_J]||_{L^\infty(\R_+ \times\Omega)}\right) \\ 
            & \hspace{6cm} + \frac{\sigma_J^2}{2} \left|\left|\partial_w^2 D^2 \right|\right|_{L^\infty (\mathcal{I})} \bigg) ||f_J||^q_{L^q(\Omega \times \mathcal{I})},
        \end{split}
        \end{equation*}
        for any $J \in \mathcal{C}$. In conclusion, 
        \begin{equation*}
            \frac{\dd}{\dd t} \sum_{J \in \mathcal{C}} ||f_J||_{L^q(\Omega \times \mathcal{I})}^q \leq C \sum_{J \in \mathcal{C}} ||f_J||_{L^q(\Omega \times \mathcal{I})}^q , 
        \end{equation*}
        with a constant $C=C(q, \gamma, \lambda, \sigma_J^2, ||\beta_T||_{L^\infty (\mathcal{I}^2)}, ||\mathcal{H}[f_J]||_{L^\infty ([0, \infty) \times \Omega)}, ||\partial_w G||_{L^\infty (\mathcal{I}^2)}, ||\partial_w^2 D^2(w)||_{L^\infty (\mathcal{I})})$, and the sought regularity estimate follows from an application of Gr\"onwall's Lemma. 
    \end{proof}
    
    \begin{remark}
        If we did not assume $\mathcal{H}[f_J] \in L^\infty (\R_+ \times \Omega)$ for any $J \in \mathcal{C}$, then the regularity result could only be proved for a fixed $x \in \Omega$. Namely, if for any $J \in \mathcal{C}$, $f_J^\init(x,\cdot) \in L^1(\mathcal{I})$ for a fixed $x \in \Omega$, then $f_J(t, x, \cdot) \in L^q(\mathcal{I})$, $J \in \mathcal{C}$, for any $t \in \R_+$. 
    \end{remark}
    
    \begin{remark}
        The additional no-flux boundary conditions \eqref{eq:additional BC for regularity opinion FP} are restrictive. Indeed we will show that the equilibria of (a simplified version of) our model (see Eq. \eqref{eq:equilibria opinion FP}) can diverge at $w= \pm 1$. Therefore, its equilibria do not belong to any $L^q(\Omega \times \mathcal{I})$, $q \geq 2$, and thus also its solutions cannot belong to any of these spaces either (see \cite[Section 4.2]{BonBor} for a thorough discussion on this topic). 
    \end{remark}

    The constants $\eta_q$ are equi-bounded by some $\eta_\infty$, for any $q \in [2,+\infty)$. Therefore, in the limit $q \to +\infty$ we recover the regularity $f_J(t,\cdot,\cdot) \in L^\infty(\Omega \times \mathcal{I})$, $J \in \mathcal{C}$, for any $t \in \R_+$, specifically
    \begin{equation*}
        \norm{f_J(t, \cdot, \cdot)}_{L^\infty(\Omega \times \mathcal{I})} \leq \max_{J \in \mathcal{C}} ||f_J^\init||_{L^\infty(\Omega \times \mathcal{I})} \exp(\eta_\infty t), \quad t \in \R_+,
    \end{equation*}
    provided that $f_J^\init \in L^\infty(\Omega \times \mathcal{I})$, $J \in \mathcal{C}$.
    
    We conclude this section by proving a uniqueness result. The main difficulty lies in the fact that we would like to linearize the Fokker--Planck equations \eqref{eq:opinion operator Q}, so that the difference of two solutions is still a solution. 
    
    \begin{proposition}[Uniqueness] \label{prop:uniqueness}
        Let $\mathbf{f}$ and $\mathbf{g}$ be two solutions of the SIR kinetic system \eqref{eq:vectorial model}, with $\beta_T \in L^\infty (\mathcal{I}^2)$. Suppose that $f_J^\init, g_J^\init \in L^1(\Omega \times \mathcal{I})$, $J \in \mathcal{C}$. Moreover, assume that for any $J \in \mathcal{C}$ the equality $\mathcal{K}[f_J](t,x,w) = \mathcal{K}[g_J](t,x,w)$ holds for any $t \in \R_+$ and for a.e. $x \in \Omega$ and $w \in  \mathcal{I}$. If $f_J^\init(x,w) = g_J^\init(x,w)$, $J \in \mathcal{C}$, for a.e $x \in \Omega$ and $w \in \mathcal{I}$, then $f_J(t,x,w) = g_J(t,x,w)$, $J \in \mathcal{C}$, for any $t \in \R_+$ and for a.e. $x \in \Omega$ and $w \in \mathcal{I}$. 
    \end{proposition}
    \begin{proof}
        Since for any $t \in \R_+$ and for a.e $x \in \Omega$ and $w \in \mathcal{I}$ we have $\sum_{J \in \mathcal{C}} \mathcal{K}[f_J](t,x,w) = \sum_{J \in \mathcal{C}} \mathcal{K}[g_J](t,x,w)$ and $\sum_{J \in \mathcal{C}} \mathcal{H}[f_J](t,x) = \sum_{J \in \mathcal{C}} \mathcal{H}[g_J](t,x)$ (recall Remark \ref{rmk1}), the differences $f_J-g_J$, $J \in \mathcal{C}$, of two solutions to the Fokker--Planck equation \eqref{eq:opinion operator Q} solve the following system: 
        \begin{align*}
        \left\{
        \begin{aligned}
            \partial_t (f_S-g_S) &= - \bigg( f_S\int_{\Omega \times \mathcal{I}} \beta_T(w,w_*) f_I(t,y,w_*) \dd y \dd w_* - g_S \int_{\Omega \times \mathcal{I}} \beta_T(w,w_*) g_I(t,y,w_*) \dd y \dd w_* \bigg) \\ 
            & \hspace{11cm} + \frac{1}{\tau} Q_S(\mathbf{f}, \mathbf{f}-\mathbf{g}), \\ 
            \partial_t (f_I-g_I) &= \bigg( f_S\int_{\Omega \times \mathcal{I}} \beta_T(w,w_*) f_I(t,y,w_*) \dd y \dd w_* - g_S \int_{\Omega \times \mathcal{I}} \beta_T(w,w_*) g_I(t,y,w_*) \dd y \dd w_* \bigg) \\ 
            & \hspace{9cm} - \gamma (f_I - g_I) + \frac{1}{\tau} Q_I(\mathbf{f}, \mathbf{f}-\mathbf{g}), \\ 
            \partial_t (f_R-g_R) &= \gamma(f_I - g_I) + \frac{1}{\tau} Q_R(\mathbf{f}, \mathbf{f}-\mathbf{g}).
        \end{aligned}
        \right.
        \end{align*}
        Now, the idea is to recover an $L^1(\Omega \times \mathcal{I})$ comparison principle for the differences $f_J-g_J$ to get rid of the Fokker–Planck operators. This is done by introducing an increasing regularized approximation of the sign function $\textrm{sign}_\eps(u)$, $u \in \R$, and defining through it a regularization $|f_J-g_J|_\eps$ of $|f_J-g_J|$ for any $J \in \mathcal{C}$. Multiplying each equation in the previous system by the corresponding $\textrm{sign}_\eps(f_J-g_J)$, integrating with respect to $x \in \Omega$ and $w \in \mathcal{I}$, and taking the limit $\eps \to 0$, we easily recover the estimates (see Lemma \ref{lemma1}) 
        \begin{align*}
        \left\{
        \begin{aligned}
            & \frac{\dd}{\dd t} ||f_S - g_S||_{L^1(\Omega \times \mathcal{I})} \leq ||\beta_T||_{L^\infty(\mathcal{I}^2)} \left( ||f_I||_{L^1(\Omega \times \mathcal{I})} ||f_S - g_S||_{L^1(\Omega \times \mathcal{I})} +||g_S||_{L^1(\Omega \times \mathcal{I})} ||f_I - g_I||_{L^1(\Omega \times \mathcal{I})} \right), \\[4mm] 
            & \frac{\dd}{\dd t} ||f_I - g_I||_{L^1(\Omega \times \mathcal{I})} \leq ||\beta_T||_{L^\infty(\mathcal{I}^2)} \left( ||f_I||_{L^1(\Omega \times \mathcal{I})} ||f_S - g_S||_{L^1(\Omega \times \mathcal{I})} +||g_S||_{L^1(\Omega \times \mathcal{I})} ||f_I - g_I||_{L^1(\Omega \times \mathcal{I})} \right) \\
            & \hspace{13cm} + \gamma ||f_I - g_I||_{L^1(\Omega \times \mathcal{I})}, \\ 
            & \frac{\dd}{\dd t} ||f_R - g_R||_{L^1(\Omega \times \mathcal{I})} \leq \gamma ||f_I - g_I||_{L^1(\Omega \times \mathcal{I})}, 
        \end{aligned} 
        \right.
        \end{align*}
        and we apply Gr\"onwall's inequality to conclude that
        \begin{equation*}
            \frac{\dd}{\dd t} \sum_{J \in \mathcal{C}} ||f_J - g_J||_{L^1(\Omega \times \mathcal{I})} \leq C \sum_{J \in \mathcal{C}} ||f_J - g_J||_{L^1(\Omega \times \mathcal{I})}, 
        \end{equation*}
        where $C=C(\gamma, ||\beta_T||_{L^\infty(\mathcal{I}^2)}, ||f_I||_{L^1(\Omega \times \mathcal{I})}, ||g_S||_{L^1(\Omega \times \mathcal{I})})$. Since the initial distributions coincide, the above estimate implies the uniqueness of solutions.         
    \end{proof}

    Finally, we provide a sufficient condition for equalities $\mathcal{K}[f_J](t,x,w) = \mathcal{K}[g_J](t,x,w)$, $J \in \mathcal{C}$ to hold for any $t \in \R_+$ and for a.e. $x \in \Omega$ and $w \in  \mathcal{I}$. 

    \begin{lemma}
            Let $\mathbf{f}$ and $\mathbf{g}$ be two solutions of the SIR kinetic system \eqref{eq:vectorial model} such that $f_J^\init(x,w) = g_J^\init(x,w)$ for a.e. $x \in \Omega$ and $w \in \mathcal{I}$. If the functionals $\mathcal{K}[f_J](t,x,w)$, $J \in \mathcal{C}$, are analytical, then 
            \begin{equation*}
                \sum_{J \in \mathcal{C}} \mathcal{K}[f_J](t,x,w) \equiv \sum_{J \in \mathcal{C}} \mathcal{K}[g_J](t,x,w)
            \end{equation*}
            for any $t \in \R_+$ and for a.e. $x \in \Omega$ and $w \in \mathcal{I}$. 
    \end{lemma}
    \begin{proof}
            Following \cite[Lemma 2]{BonBor}, since the initial distributions coincide, the idea is to show that 
            \begin{equation*}
                \partial_t^n \left. \sum_{J \in \mathcal{C}} \big( \mathcal{K}[f_J](t,x,w) - \mathcal{K}[g_J](t,x,w)\big) \right|_{t=0} = 0,
            \end{equation*}
            for all $n \in \mathbb{N}^*$, and for a.e. $x \in \Omega$ and $w \in \mathcal{I}$. For $n = 1$ it holds that 
            \begin{equation*}
            \begin{split}
                \partial_t \sum_{J \in \mathcal{C}} & \big(\mathcal{K}[f_J] - \mathcal{K}[g_J]\big) \bigg|_{t=0} =  \sum_{J \in \mathcal{C}} \int_{\Omega \times \mathcal{I}} \mathcal{B}(x,y) P(x,y) G(w,w_*) (w-w_*) \partial_t  (f_J-g_J)(t,y,w_*)\big|_{t=0} \dd y \dd w_*  \\[2mm] 
                &= \frac{\lambda}{\tau} \sum_{J \in \mathcal{C}} \int_{\Omega \times \mathcal{I}} \mathcal{B}(x,y) P(x,y) G(w,w_*) (w-w_*) \\ 
                & \hspace{2cm} \times \partial_{w_*} \left(\sum_{J' \in \mathcal{C}} \mathcal{K}[f_{J'}^\init](y,w_*) \; f_J^\init(y,w_*) - \sum_{J' \in \mathcal{C}} \mathcal{K}[g_{J'}^\init](y,w_*) \; g_J^\init(y,w_*)\right) \dd y \dd w_*  \\[2mm] 
                & \quad +\frac{1}{\tau}\sum_{J \in \mathcal{C}}\frac{\sigma_J^2}{2} \int_{\Omega \times \mathcal{I}} \mathcal{B}(x,y) P(x,y) G(w,w_*) (w-w_*) \\ 
                & \hspace{2cm} \times \partial_{w_*}^2 \left( D^2(w_*) \left(\sum_{J' \in \mathcal{C}} \mathcal{H}[f_{J'}^\init](y) \; f_J^\init(y,w_*) - \sum_{J' \in \mathcal{C}} \mathcal{H}[g_{J'}^\init](y) \; g_J^\init(y,w_*)\right)\right) \dd y \dd w_* \\[2mm] 
                &= 0,
            \end{split}
            \end{equation*}
            for a.e. $x \in \Omega$ and $w \in \mathcal{I}$, since the initial distributions coincide. The calculations for higher values of $n \geq 2$ are similar. 
    \end{proof}

\section{A simplification of the model} \label{connectivity}

    \noindent As already observed in Section \ref{sec:macro}, the general structures of the graphon $\mathcal{B}$ and of the interaction function $P$ do not allow us to deduce a meaningful expression for the (local) equilibria of system \eqref{eq:vectorial model}. The goal of this section is to derive a simplified version of our model which admits meaningful equilibria, and to study their properties. The main idea consists in introducing a new quantity, the agents' propensity to interact, which incorporates both effects of $\mathcal{B}$ and $P$.

\subsection{An analysis of the agents' propensity to interact} \label{propensity}

    The propensity to interact of each agent depends on two different factors: the graphon $\mathcal{B}$ indicates how often an agent with position $x \in \Omega$ interacts with the other agents; the interaction function $P$ tells us how much the same agent is influenced by each of these interactions. For any $x \in \Omega$, we can therefore define the propensity to interact $p = p(x)$ as 
    \begin{equation*} 
        p(x) = \int_\Omega \mathcal{B}(x,z) P(x,z) \dd z,  
    \end{equation*}
    which is well defined as soon as
    \begin{equation*}
        \mathcal{B}(x,\cdot) \in L^1(\Omega),
    \end{equation*}
    for any $x \in \Omega$. Note that this hypothesis is directly related with the integrability condition \eqref{eq:integrability condition}. In order to highlight the importance of the function $p: \Omega \to \R_+$, we analyze the following example (another one is reported in Appendix \ref{appA}, and both of them will be fundamental for the numerical simulations that we will carry out in Section \ref{sec:num}).

    Consider the symmetric (note that this assumption is nor unrealistic, nor necessary in order to define $p$) graphon
    \begin{equation} \label{eq:fat-tailed graphon}
        \mathcal{B}(x,y)=(xy)^{-\xi}, \quad \xi \in (0,1), 
    \end{equation}
    usually associated with the context of scale-free networks \cite{BorChaCohZha1, BorChaCohZha2}, i.e., simple graphs whose degree distribution possesses fat tails. Moreover, assume that \cite{DurFraWolZan} 
    \begin{equation} \label{eq:P}
        P(x,y) = \left(1 + \frac{d_\init(x)}{d_\init(y)}\right)^{- \chi}, \quad \chi>0,
    \end{equation}
    where the in-degree $d_\init(x)$ of $x \in \Omega$ is defined as 
    \begin{equation} 
        d_\init(x) = \int_\Omega \mathcal{B}(x,z) \dd z. 
    \end{equation}
    Note that, with this particular choice \eqref{eq:P} of $P$, the interactions depend on the connectivity of each agent: the more connected an agent is the less they are influenced by the others, while agents with a lower connectivity are more affected. In particular, 
    \begin{itemize}
        \item $d_\init(x) \gg d_\init(y)$ implies that, on average with respect to the random variables $\eta_J$, the agent with the highest degree keeps their opinion; \\[-2mm]
        \item $d_\init(x) \simeq d_\init(y)$ implies that agents with a similar number of incoming connections
        are the ones that can most influence each other; \\[-2mm]
        \item $d_\init(x) \ll d_\init(y)$ implies that the less influential agents tend to adopt the opinion of the more connected ones.
    \end{itemize}   
    Figure \ref{fig:main_graphon} shows the graphon \eqref{eq:fat-tailed graphon} and the interaction function \eqref{eq:P}. 

    \begin{figure}
    \centering
    \begin{subfigure}[t]{0.5\textwidth}
        \centering
        \includegraphics[width=\linewidth]{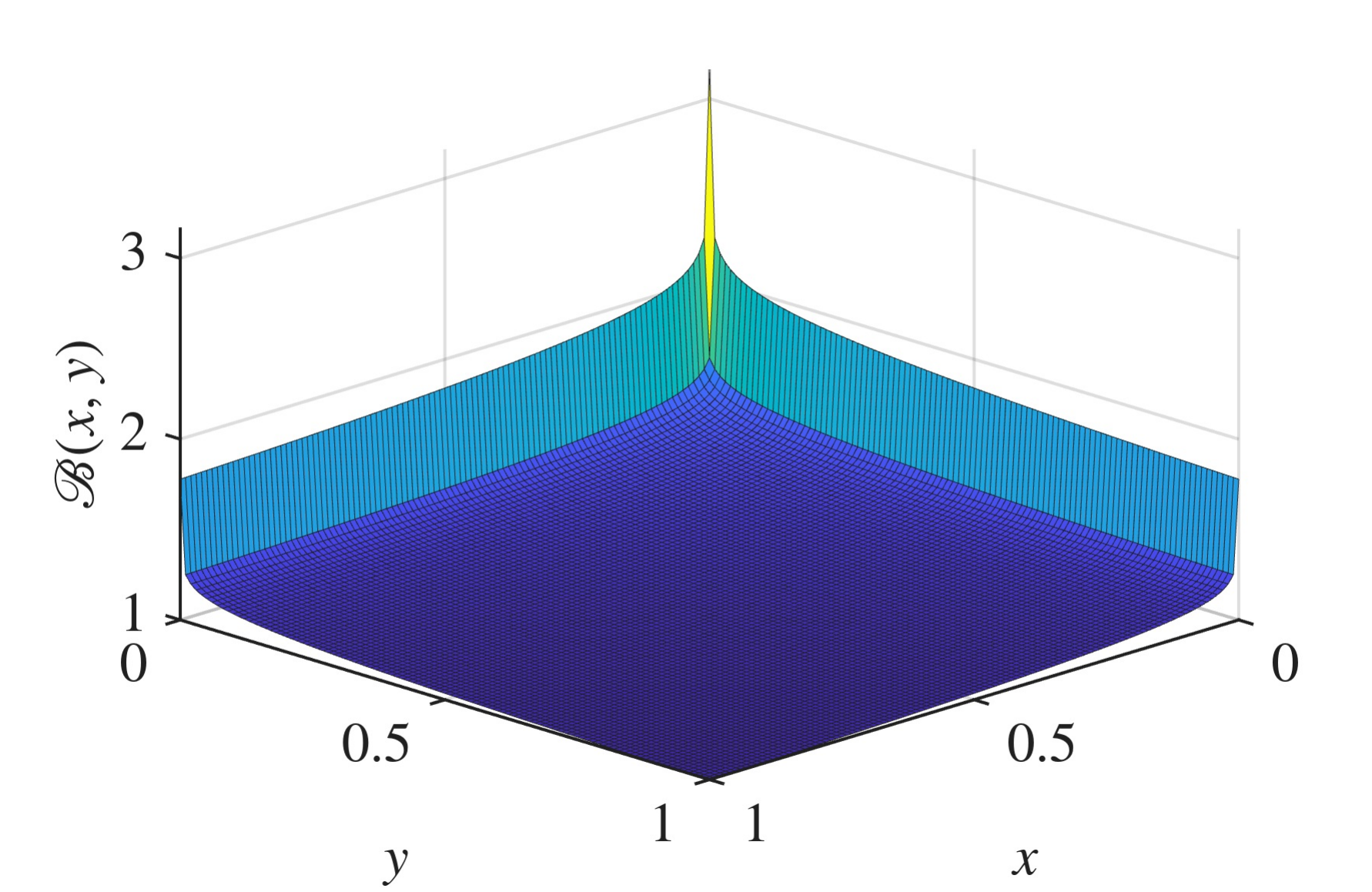}
        \caption*{(A)}
    \end{subfigure}\hfill
    \begin{subfigure}[t]{0.5\textwidth}
        \centering
        \includegraphics[width=\linewidth]{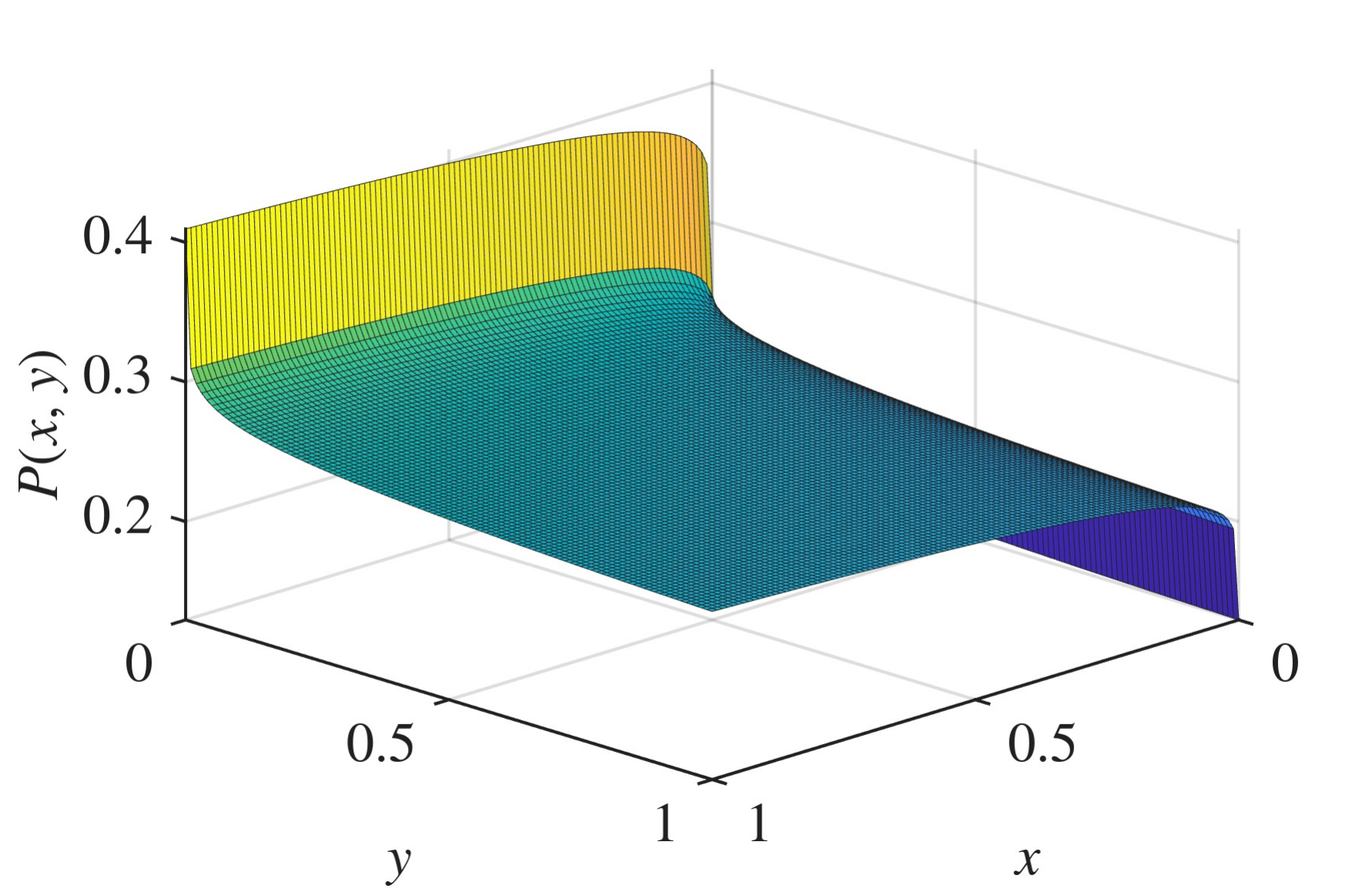}
        \caption*{(B)}
    \end{subfigure}
      \caption{(A) Plot of the graphon $\mathcal{B}$ given by \eqref{eq:fat-tailed graphon}. (B) Plot of the interaction function $P$ defined by \eqref{eq:P}. Values of the parameters: $\xi=0.05$ and $\chi=2$.}
    \label{fig:main_graphon}
    \end{figure}
    
    The propensity to interact $p$ explicitly reads 
    \begin{equation} \label{eq:propensity}
        p(x) = x^{\xi (\chi-1)} \int_\Omega z^{-\xi} (x^\xi + z^\xi)^{-\chi} \dd z, 
    \end{equation}
    hence two situations must be distinguished: 
    \begin{itemize}
        \item $\chi < 1$ implies that $p(x)$ decreases as $x$ increases. Moreover $p \to +\infty$ as $x \to 0$; \\[-2mm]
        \item $\chi > 1$ implies that $p(x)$ can show different behaviors as $x$ varies, depending on the relationship between the parameters $\xi$ and $\chi$. Note that in this case, if $\xi (1+\chi)<1$ then $p(0)=0$ (in particular, it is well-defined even if $\mathcal{B}(0, \cdot)$ is not).  
    \end{itemize}
    The chosen unbounded graphon \eqref{eq:fat-tailed graphon} can be used to model different phenomena. Agents in position $x \simeq 0$ could actually be interpreted as opinion leaders (e.g., the media or highly influential individuals) since they have an extremely high number of connections \cite{BonBor, BonTosZan}. In this case, it is reasonable to assume that their propensity to interact is almost equal to zero since they are willing to influence the others but they do not want to be influenced. Hence, a reasonable choice of $\chi>1$ and of $\xi<(1+\chi)^{-1}$ is such that $p$ attains its maximum value in the interior of $\Omega$. This means that the agents who are the most inclined to interact have an intermediate number of connections. Indeed, highly connected individuals (i.e., the opinion leaders, corresponding to $\mathcal{B} \to +\infty$) do not want to be influenced by the others ($P \simeq 0$), while poorly connected individuals do not engage in a relevant number of interactions ($\mathcal{B} \simeq 0$) even if they are strongly influenced by them ($P \simeq 1$). On the other hand, if no opinion leaders are considered, it is reasonable to assume $\chi < 1$. In this case, only normal agents are taken into account, and as their connectivity decreases, their propensity to interact decreases accordingly. In Figure \ref{fig:propensity to interact} we plot the function $p$ for the two discussed situations. 

    \begin{figure}[h!]
    \includegraphics[width=8cm]{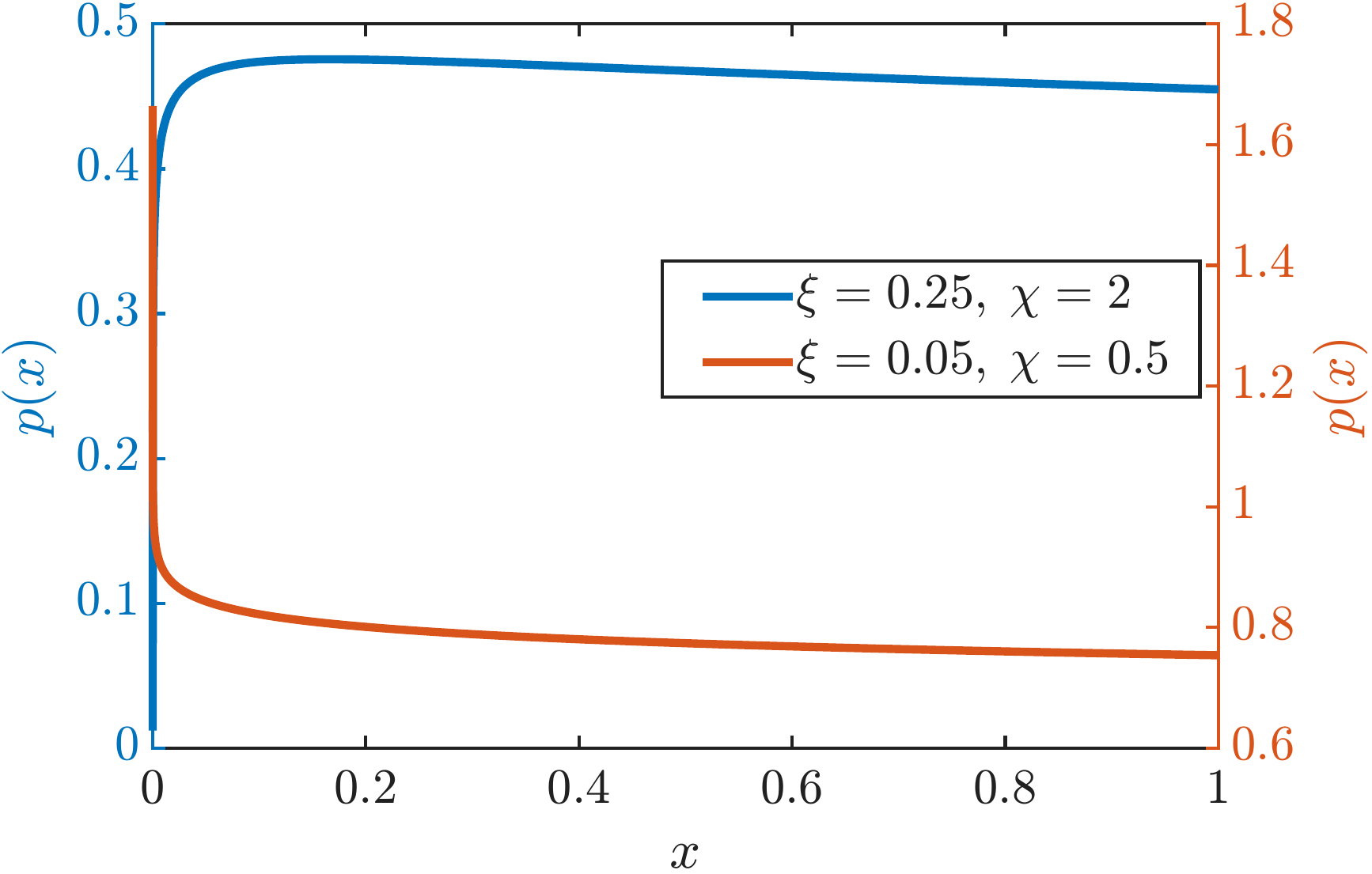}
    \caption{Graph of the propensity to interact $p$ defined by \eqref{eq:propensity}, for different values of $\xi$ and $\chi$. The choice of parameters $\xi=0.25$, $\chi=2$ (left scale) corresponds to the presence of opinion leaders ($p \to 0$ as $x \to 0$), while the choice $\xi=0.05$, $\chi=0.5$ (right scale) corresponds to their absence.} \label{fig:propensity to interact}
    \end{figure} 

    \begin{remark}
        The choice \eqref{eq:fat-tailed graphon} for the graphon $\mathcal{B}$ satisfies the integrability condition \eqref{eq:integrability condition} under a reasonable and realistic assumption. Indeed, provided that $f_J^\init \in L^1(\Omega\times\mathcal{I})$, $J \in \mathcal{C}$, in order to have $\int_{\Omega \times \mathcal{I}} y^{-\xi} f_J(t,y,w_*) \dd y \dd w_* = +\infty$ for some $J \in \mathcal{C}$ and $t>0$, one must have $f_J \sim y^{-\varsigma}$ at $y = 0$ with $\varsigma + \xi > 1$, meaning that the number of agents increases as the connectivity increases. This is clearly unrealistic since highly connected individuals should be a minority. Moreover, recall that (see Remark \ref{rmk1}) it suffices to require that $\int_{\Omega \times \mathcal{I}} y^{-\xi} f_J(0,y,w_*) \dd y \dd w_* < +\infty$ for any $J \in \mathcal{C}$. 
    \end{remark}

\subsection{Derivation of a simplified model} \label{sec:derivation_simplified}

    Instead of considering both the graphon $\mathcal{B}$ and the interaction function $P$, we can derive a simpler model depending solely on the propensity to interact $p$, since it incorporates information from both quantities. The idea is to remove the graphon $\mathcal{B}$ from \eqref{eq:weak form opinion} (and the subsequent calculations) and to replace the microscopic interactions \eqref{eq:microscopic interactions} with 
    \begin{align*}
    \begin{aligned}
        w' &= w + G(w,w_*)\tilde{P}(x)\tilde{P}(y)\lambda (w_*-w) + D(w) \eta_J, \\[2mm] 
        w_*' &= w_* + G(w_*,w)\tilde{P}(y)\tilde{P}(x) \lambda (w-w_*) + D(w_*) \eta_{J'},
    \end{aligned} 
    \end{align*}
    where we have defined $\tilde{P}(x)=g(p(x))$, based on an increasing function $g: \R_+ \to [0,1]$ such that $g(0)=0$. For example, one can consider
    \begin{equation} \label{eq:g_P}
        g(p(x)) = \frac{p(x)}{a+p(x)}, 
    \end{equation}
    for some $a > 0$. Note that this simplification is reasonable only if there are no opinion leaders, since the latter would be incapable of influencing other agents (indeed, according to Section \ref{propensity}, opinion leaders would satisfy $\tilde{P} \simeq 0$). 

    Following Sections \ref{kin_op} and \ref{sub3}, we can derive the following Fokker--Planck equation to describe the evolution of $f_J$, $J \in \mathcal{C}$,
    \begin{equation} \label{eq:opinion operator Q simplified}
    \begin{split}
        \partial_t f_J&=\frac{1}{\tau}\left(\lambda \partial_w \left(\tilde{P}(x)\sum_{J' \in \mathcal{C}} \tilde{\mathcal{K}}[f_{J'}] f_J\right)+\frac{\sigma_J^2}{2} \partial^2_w \left(D^2(w) f_J\right) \right) \\[2mm]
        & = \frac{1}{\tau}\tilde{Q}_J(\mathbf{f},\mathbf{f}), \quad t\in \R_+,\ x \in \Omega,\ w \in \mathcal{I},
    \end{split}
    \end{equation}
    where we have defined, for any $J \in \mathcal{C}$,
    \begin{equation}
        \tilde{\mathcal{K}}[f_J](t,x,w) = \int_{\Omega \times \mathcal{I}} \tilde{P}(y) G(w,w_*) (w-w_*) f_J(t,y,w_*) \dd y \dd w_*, \quad t \in \R_+,\; x \in \Omega,\; w \in \mathcal{I}. 
    \end{equation}
    Once again, we have to complete the Fokker--Planck equation \eqref{eq:opinion operator Q simplified} with no-flux boundary conditions of the form
    \begin{align*}
    \left\{
    \begin{aligned}
        & \left. D^2(w) f_J(t,x,w) \right|_{w=\pm 1} = 0, \\[2mm] 
        & \left. \lambda \tilde{P}(x) \displaystyle \sum_{J' \in \mathcal{C}} \tilde{\mathcal{K}}[f_{J'}](t,x,w)  f_J(t,x,w)+\frac{\sigma_J^2}{2} \partial_w \left(D^2(w) f_J(t,x,w)\right) \right|_{w=\pm 1}=0, 
    \end{aligned}
    \right.
    \end{align*}
    holding for any $t \in \R_+$ and $x \in \Omega$. Finally, notice that model \eqref{eq:opinion operator Q simplified} preserves the average opinion of the population over time, because obviously $\tilde{P}(x)\tilde{P}(y)=\tilde{P}(y)\tilde{P}(x)$. 
    
    By defining $\tilde{\mathbf{Q}}(\mathbf{f},\mathbf{f}) = (\tilde{Q}_J(\mathbf{f}, \mathbf{f}))_{J\in\mathcal{C}}$ componentwise with the reduced operators \eqref{eq:opinion operator Q simplified}, the simplified version of the full SIR kinetic system \eqref{eq:vectorial model} now reads 
    \begin{equation} \label{eq:vectorial model simplified}
        \partial_t \mathbf{f} = \mathbf{E}(\mathbf{f}, \mathbf{f}) + \frac{1}{\tau} \tilde{\mathbf{Q}}(\mathbf{f},\mathbf{f}), \quad t \in \R_+,\ x \in \Omega,\ w \in \mathcal{I}, 
    \end{equation}
    for which all analytical properties derived in Section \ref{section2} clearly still hold. 

\subsection{Equilibria, opinion polarization, and consensus formation} \label{polarization}

    Following and further expanding Section \ref{section_local_eq}, we can compute steady state of the Fokker--Planck equation \eqref{eq:opinion operator Q simplified} assuming $G \equiv 1$ and $D(w)=\sqrt{1-w^2}$. In order to do so, let us define, for $t \in \R_+$, the quantities 
    \begin{equation} \label{eq:weighted density and mean}
    \begin{split}
        \rho_{\tilde{P}}(t) & = \sum_{J \in \mathcal{C}} \int_{\Omega \times \mathcal{I}}  \tilde{P}(x) f_J(t,x,w) \dd x \dd w, \\[4mm]  
        m_{\tilde{P}}(t) & = \frac{1}{\rho_{\tilde{P}}(t)} \sum_{J \in \mathcal{C}} \int_{\Omega \times \mathcal{I}}  \tilde{P}(x) w f_J(t,x,w) \dd x \dd w,
     \end{split}
    \end{equation} 
    and point out that only $\rho_{\tilde{P}}$ is conserved over time. The local equilibrium states $B_J = B_J(t,x,w)$, $J \in \mathcal{C}$, of equations \eqref{eq:opinion operator Q simplified} must satisfy, for any $t \in \R_+$, $x \in \Omega$ and $w \in \mathcal{I}$, 
    \begin{equation}
        \lambda \tilde{P}(x) \rho_{\tilde{P}} \left(w-m_{\tilde{P}}(t)\right)B_J(t,x,w) + \frac{\sigma_J^2}{2} \partial_w \left( (1-w^2) B_J(t,x,w)  \right) = 0, 
    \end{equation}
    Therefore, after developing the derivative with respect to $w$, we find that the $B_J$, $J \in \mathcal{C}$ are given by the beta distributions 
    \begin{equation} \label{eq:equilibria opinion FP}
        B_J(t,x,w)= c_J^B(t,x) (1+w)^{-1+\frac{ \tilde{P}(x) \rho_{\tilde{P}}}{\nu_J} (1 + m_{\tilde{P}}(t))} (1-w)^{-1+\frac{\tilde{P}(x) \rho_{\tilde{P}}}{\nu_J} (1 - m_{\tilde{P}}(t))}, \quad t \in \R_+,\; x \in \Omega,\; w \in \mathcal{I},
    \end{equation}
    where each $c_J^B(t,x) > 0$ is a suitable normalization function such that $\int_{\Omega \times \mathcal{I}} B_J(t,x,w) \dd x \dd w = 1$ (assuming that $\tilde{Q}_J$ is initially applied to distributions $f_J^\init$, $J \in \mathcal{C}$, of unitary mass). This function is uniquely determined by the initial conditions, since the operators $\tilde{Q}_J$ are mass preserving and the agents' position on the graphon does not change over time. In particular, it is easy to check that
     \begin{equation*}
         c_J^B(t,x) = \frac{2^{1-2\frac{\tilde{P}(x)}{\nu_J} \rho_{\tilde{P}}}}{B\left(\frac{\tilde{P}(x) \rho_{\tilde{P}}}{\nu_J} (1 + m_{\tilde{P}}(t)), \frac{\tilde{P}(x) \rho_{\tilde{P}}}{\nu_J} (1 - m_{\tilde{P}}(t))\right)} \int_{\mathcal{I}} f_J^\init(x,w) \dd w, \quad t \in \R_+,\; x \in \Omega,
     \end{equation*}
     where $B$ denotes the standard Beta function. Note that the equilibrium \eqref{eq:equilibria opinion FP} is well defined as long as $|m_{\tilde{P}}(t)| < 1$, i.e., as long as not all agents have opinion $w=1$ or $w=-1$. This is a reasonable assumption since the latter cases would correspond to a trivial equilibrium given by a Dirac delta distribution. 

     \begin{remark} \label{rem:moment compatibility}
         We observe that the local equilibria \eqref{eq:equilibria opinion FP} are consistent with the graphon-weighted moments \eqref{eq:weighted density and mean}. Indeed, from the definition of $c_J^B(t,x)$ and the conservation of $\rho_{\tilde{P}}$ we infer the mass compatibility
         \begin{equation*}
              \sum_{J \in \mathcal{C}} \int_{\Omega \times \mathcal{I}} \tilde{P}(x) B_J(t,x,w) \dd x \dd w = \sum_{J \in \mathcal{C}} \int_{\Omega \times \mathcal{I}} \tilde{P}(x) f_J^\init(x,w) \dd x \dd w = \rho_{\tilde{P}}(0) = \rho_{\tilde{P}}(t).
         \end{equation*}
         Moreover, simple computations show that
         \begin{equation*}
             \int_{\mathcal{I}} w B_J(t,x,w) \dd w = m_{\tilde{P}}(t) \int_{\mathcal{I}} f_J^\init(x,w) \dd w,
         \end{equation*}
         which in turn allows us to prove the momentum compatibility
         \begin{equation*}
             \frac{1}{\rho_{\tilde{P}}} \sum_{J \in \mathcal{C}} \int_{\Omega \times \mathcal{I}} \tilde{P}(x) w B_J(t,x,w) \dd x \dd w = \frac{m_{\tilde{P}}(t)}{\rho_{\tilde{P}}} \sum_{J \in \mathcal{C}} \int_{\Omega \times \mathcal{I}} \tilde{P}(x) f_J^\init(x,w) \dd x \dd w = m_{\tilde{P}}(t),
         \end{equation*}
         thanks to the conservation of $\rho_{\tilde{P}}$.
     \end{remark}

    The equilibria \eqref{eq:equilibria opinion FP} can describe two different scenarios: 
    \begin{enumerate}
        \item None of the exponents of the distribution are negative, therefore $B_J(t,x,\pm 1)=0$ and, fixed $x$, the maximum of $B_J$ is obtained for $w \in (-1,1)$. This situation corresponds to a consensus formation \cite{DurFraWolZan} due to the compromise process of the binary interactions \eqref{eq:microscopic interactions}, where a full consensus would be associated with Dirac delta distributions. \\[-2mm]
        \item At least one of the exponents is negative, hence the distribution diverges. This situation corresponds to a polarization of opinions \cite{Zan} and is commonly related to a self-thinking process stronger than the compromise propensity. 
    \end{enumerate}
    Clearly, the formation of (partial) consensus is preferable over opinion polarization since one would like to avoid the formation of extreme opinions. For example, in \cite{Zan} it is shown that extreme opinions regarding a disease can trigger an increasing spread of infection in the society. Since the quantity $\tilde{P}(x)$ appears at the exponents of \eqref{eq:equilibria opinion FP}, at equilibrium, agents that are inclined to interact ($\tilde{P}(x) \simeq 1$) can be characterized by consensus formation, while at the same time, agents that are not inclined to interact ($\tilde{P}(x) \ll 1$) experience opinion polarization. This is related to the fact that compromise dynamics are always stronger among connected individuals. Moreover, since $|m_{\tilde{P}}| < \rho_{\tilde{P}} \leq 1$, one easily checks that
    \begin{equation*}
       \min \left\{\rho_{\tilde{P}} - m_{\tilde{P}}, \rho_{\tilde{P}} + m_{\tilde{P}} \right\}\in (0,1],
    \end{equation*}
    and since $\tilde{P}(x) \in (0,1]$, we need small values of $\nu_J$ in order to obtain consensus formation. However, as illustrated in Figure \ref{f1}, $\nu_J<1$ is not a sufficient condition to prevent opinion polarization. Recall that $\nu_J<1$ means that the compromise dynamics are stronger than the self-thinking ones, and we see again the importance of the compromise processes for consensus formation. Finally, since $\nu_J$ can vary with $J \in \mathcal{C}$, some agents with position $x$ can be characterized by consensus formation if they belong to a certain compartment, while by opinion polarization if they belong to another compartment. This interesting dynamics highlights the influence of the environment on self-thinking. Indeed, if we assume that infected agents gain a clearer idea of how the disease develops, one could take $\sigma_R^2 \leq \sigma_I^2 \leq \sigma_S^2$, implying that opinion polarization is more likely to occur among susceptible agents. 

    \begin{remark}
        These results align with the analysis performed in \cite{BonBor}, where the first and the second author studied a kinetic model for opinion dynamics driven by the social activity of individuals. There it was shown that socially inactive agents tend to be characterized by opinion polarization, while socially active ones are more prone to reach a consensus, highlighting the importance of social interactions to prevent the formation of extreme opinions. 
    \end{remark}

    \begin{figure}[h!]
    \centering
    \begin{subfigure}{0.48\textwidth}
        \centering
        \includegraphics[width=\textwidth]{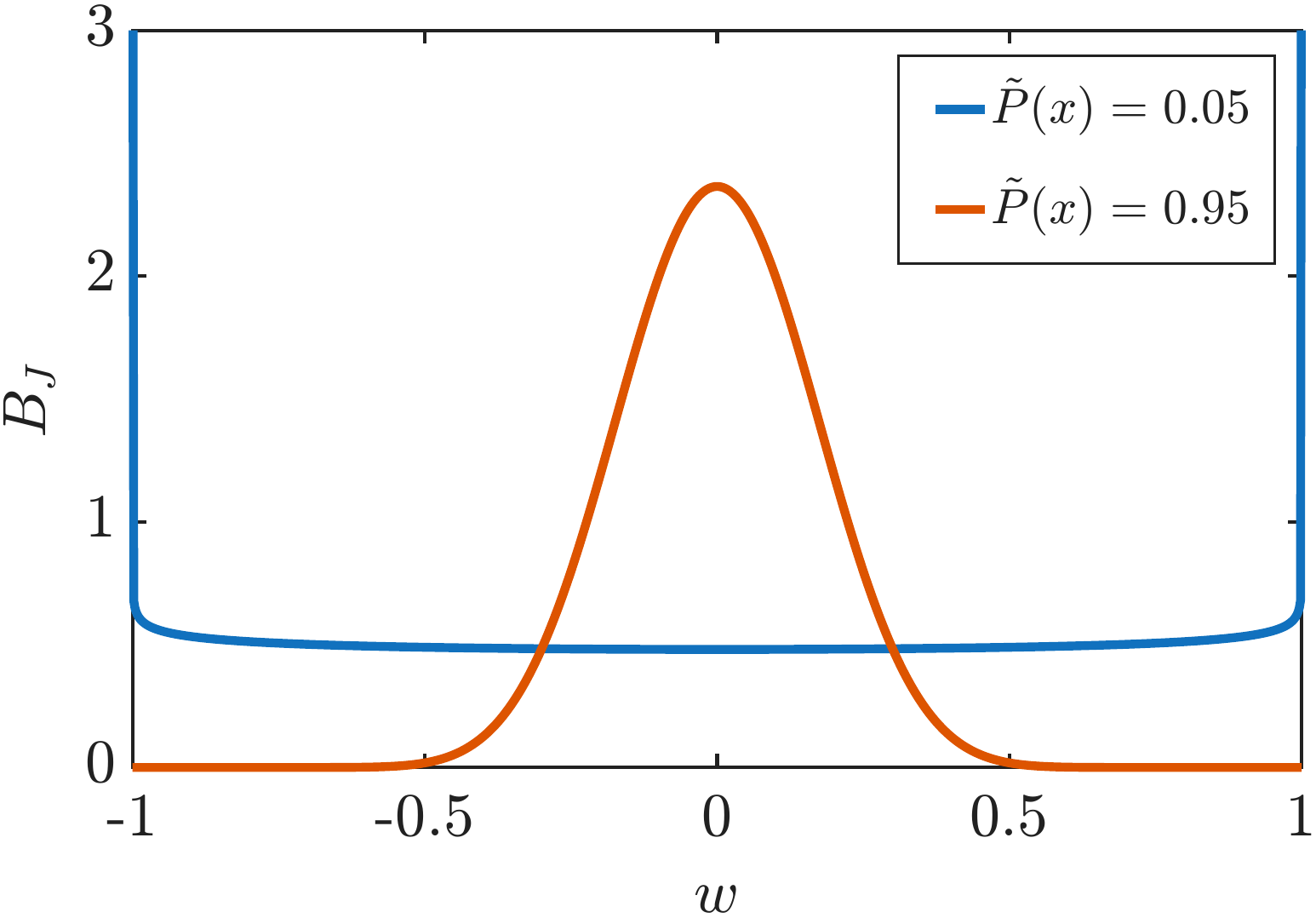}
        \caption{} \label{f1_a}
    \end{subfigure}
    \hspace{5mm}
    \begin{subfigure}{0.475\textwidth}
        \centering
        \includegraphics[width=\textwidth]{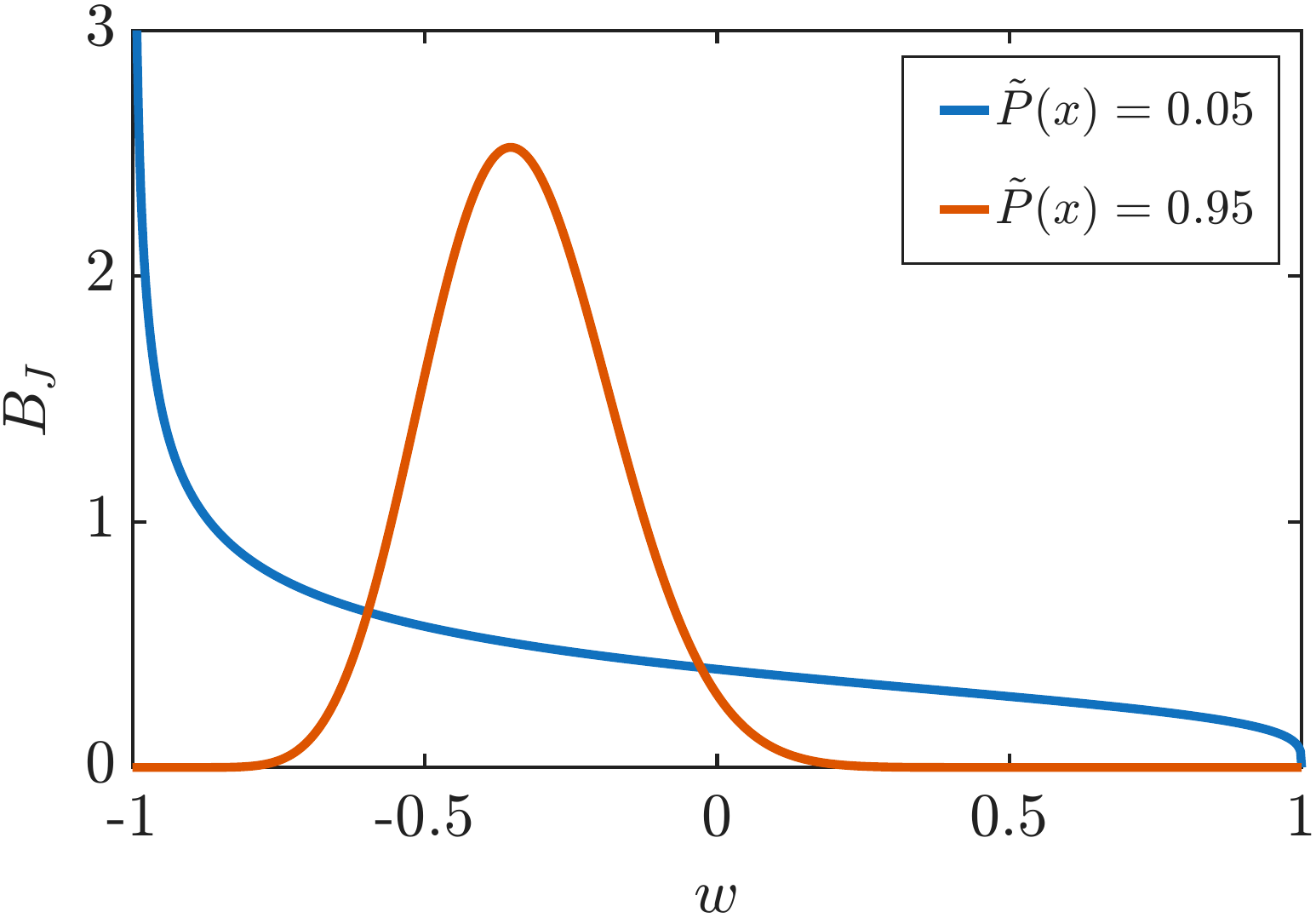}
        \caption{} \label{f1_b}
    \end{subfigure}
    \caption{Beta distributions \eqref{eq:equilibria opinion FP} for different values of $\tilde{P}(x)$. Agents that are not inclined to interact ($\tilde{P}(x)=0.05$) are characterized by opinion polarization, while agents that are inclined to do so ($\tilde{P}(x)=0.95$) experience consensus formation. We have chosen $m_{\tilde{P}}=0$ in (A) and $m_{\tilde{P}}=-0.2$ in (B), while the other parameters are given by $\rho_{\tilde{P}}=0.6$ and $\nu_J=0.032$. Moreover, the function $c_J(t,x)$ is such that the areas subtended by the graphs are equal to $1$.} \label{f1}
    \end{figure}
     
    Consider now the simplified SIR kinetic model \eqref{eq:vectorial model simplified}. We call \emph{local equilibrium} $f_J^\eq = f_J^\eq(t,x,w)$ any distribution that reached the steady state \eqref{eq:equilibria opinion FP} of the Fokker--Planck equation \eqref{eq:opinion operator Q simplified}. These local equilibrium distributions $f_J^\eq$, $J \in \mathcal{C}$, have the structure 
    \begin{equation} \label{eq:local equilibria opinion FP-SIR}
    \begin{split}
        f_J^\eq(t,x,w) &= c_J^\eq(t,x) (1+w)^{-1+\frac{\tilde{P}(x) \rho_{\tilde{P}}}{\nu_J} (1 + m_{\tilde{P}}(t))} (1-w)^{-1+\frac{\tilde{P}(x) \rho_{\tilde{P}}}{\nu_J} (1 - m_{\tilde{P}}(t))} \\[2mm]
        &= \frac{c_J^\eq(t,x)}{c_J^B(t,x)} B_J(t,x,w), \quad t \in \R_+,\; x \in \Omega,\; w \in \mathcal{I},
    \end{split} 
    \end{equation}
    where $c_J^\eq(t,x) \geq 0$ is a suitable normalization function such that $\int_{\Omega \times \mathcal{I}} f_J^\eq(t,x,w) \dd x \dd w = \rho_J(t)$ for any $t \in \R_+$. Notice that in general $c_J^\eq(t,x) \not\equiv \rho_J(t) c_J^B(t,x)$ since the local incidence rate \eqref{eq:function beta_T} depends on $w$. This implies that, for example, some positions on the graphon could increase the chances of contracting the disease because the evolution of the opinion depends on such positions. Hence, unlike $c_J^B$, the function $c_J^\eq$ cannot in general be derived from the initial conditions and its evolution over time can be complex. Going further, the \emph{global equilibrium} $f_J^\infty = f_J^\infty(x,w)$ of system \eqref{eq:vectorial model simplified} must satisfy $\rho_I^\infty =0$, and it is thus given by
    \begin{equation} \label{eq:global equilibria opinion FP-SIR}
    \begin{split}
        f_S^\infty(x,w) = \frac{c_S^\infty(x)}{c_S^{B,\infty}(x)} B_S^\infty(x,w), \quad f_I^\infty(x,w) = 0, \quad
        f_R^\infty(x,w) &= \frac{c_R^\infty(x)}{c_R^{B,\infty}(x)} B_R^\infty(x,w), \quad x \in \Omega,\; w \in \mathcal{I},
    \end{split}
    \end{equation} 
    where we have defined
    \begin{equation} \label{eq:Beta equilibria}
        B_J^\infty(x,w) = c_J^{B,\infty}(x) (1+w)^{-1+\frac{\tilde{P}(x) \rho_{\tilde{P}}}{\nu_J} (1 + m_{\tilde{P}}^\infty)} (1-w)^{-1+\frac{\tilde{P}(x) \rho_{\tilde{P}}}{\nu_J} (1 - m_{\tilde{P}}^\infty)}, \quad x \in \Omega,\; w \in \mathcal{I},
    \end{equation}
    each $c_J^\infty(x)$ and $c_J^{B,\infty}(x)$ respectively denote the quantities $c_J^\eq(t,x)$ and $c_J^B(t,x)$ computed at the corresponding final distributions of \eqref{eq:local equilibria opinion FP-SIR} and \eqref{eq:equilibria opinion FP}, and similarly $m_{\tilde{P}}^\infty$ corresponds to the quantity $m_{\tilde{P}}(t)$ computed at the limit distributions of \eqref{eq:equilibria opinion FP}. 

\subsection{Macroscopic dynamics at equilibrium} \label{sec:sec:macro}

    A particular situation occurs when we suppose that the opinion exchanges are much faster than the evolution of the epidemic, which happens when $\tau \ll 1$. In this case, one expects the distribution functions $f_J$ to quickly reach the corresponding steady state \eqref{eq:local equilibria opinion FP-SIR}. Therefore, in the limit $\tau \to 0$ we can formally replace $\mathbf{f}$ with $\mathbf{f}^\eq = (f_J^\eq)_{J \in \mathcal{C}}$ in the kinetic SIR system \eqref{eq:vectorial model simplified}, to recover the equation
    \begin{equation*}
        \partial_t \mathbf{f}^\eq = \mathbf{E}(\mathbf{f}^\eq, \mathbf{f}^\eq), \quad t \in \R_+,\; x \in \Omega,\; w \in \mathcal{I}.
    \end{equation*}
    Integrating the latter over $x \in \Omega$ and $w \in \mathcal{I}$ and recalling Remark \ref{rem:moment compatibility} on the compatibility of moments for $f_J^\eq$, we thus deduce that the $\rho_J$, $J \in \mathcal{C}$, solve the kinetic-controlled macroscopic SIR-like system
    \begin{equation} \label{eq:controlled SIR}
    \left\{
    \begin{aligned}
        & \frac{\dd}{\dd t} \rho_S = - \beta (1-m_{\tilde{P}})^2 \rho_S \rho_I,  \\[2mm] 
        & \frac{\dd}{\dd t} \rho_I = \beta (1-m_{\tilde{P}})^2 \rho_S \rho_I - \gamma \rho_I, \\[2mm] 
        & \frac{\dd}{\dd t} \rho_R = \gamma \rho_I,
    \end{aligned}
    \right.
    \end{equation}
    with effective reproduction number
    \begin{equation*}
        \mathcal{R}_{\textnormal{eff}}^{\tilde{P}}(t) = \frac{\beta}{\gamma} (1-m_{\tilde{P}}(t))^2 \rho_S(t). 
    \end{equation*}
    Therefore, at local equilibrium the dynamics of the mass densities are driven by the behavior of the graphon-weighted momentum $m_{\tilde{P}}$ (rather than the compartmental means $m_S$ and $m_I$), showing that the model is appropriate for setting up control strategies on the epidemic by suitably modifying the structure of the interaction function $\tilde{P}$ through the graphon $\mathcal{B}$ itself. 

    Notice however that system \eqref{eq:controlled SIR} is still not closed and one is lead to derive an infinite hierarchy of evolution equations for the higher order weighted moments
    \begin{equation*}
    \begin{split}
        \rho_{\tilde{P}^n}(t) & = \sum_{J \in \mathcal{C}} \int_{\Omega \times \mathcal{I}}  \tilde{P}^n(x) f_J(t,x,w) \dd x \dd w, \\[4mm]
        m_{\tilde{P}^n}(t) & = \frac{1}{\rho_{\tilde{P}}(t)} \sum_{J \in \mathcal{C}} \int_{\Omega \times \mathcal{I}}  \tilde{P}^n(x) w f_J(t,x,w) \dd x \dd w,
     \end{split}
    \end{equation*}
    defined for any $n \in \mathbb{N}^*$, where each $\rho_{\tilde{P}^n}$ is in fact conserved over time. Multiplying system \eqref{eq:vectorial model simplified} by $\tilde{P}^n(x) w$, summing over $J \in \mathcal{C}$ and integrating over $x \in \Omega$ and $w \in \mathcal{I}$, we can thus determine the hierarchy of equations
    \begin{equation} \label{eq:hierarchy of moments}
        \frac{\dd}{\dd t} m_{\tilde{P}^n} = -\frac{\lambda}{\tau} \rho_{\tilde{P}^{n+1}} (m_{\tilde{P}^{n+1}} - m_{\tilde{P}}),
    \end{equation}
    where we have used two successive integration by parts to get rid of the diffusion term in the Fokker--Planck operator, while the epidemiological operators naturally cancels out when summing over $J \in \mathcal{C}$ (recall also that we are still working with $G \equiv 1$ and $D(w) = \sqrt{1-w^2}$).
    
    System \eqref{eq:hierarchy of moments} can be written in simple matrix form as
    \begin{equation*}
        \frac{\dd}{\dd t} M_{\tilde{P}} = A_{\tilde{P}} M_{\tilde{P}},
    \end{equation*}
    with $M_{\tilde{P}} = (m_{\tilde{P}}, m_{\tilde{P}^2}, m_{\tilde{P}^3},\ldots)^\intercal$ and
    \begin{equation*}
        A_{\tilde{P}} = \frac{\lambda}{\tau} \left( \begin{array}{ccccc}
            \rho_{\tilde{P}^2} & -\rho_{\tilde{P}^2} & 0 & 0 & \ldots \\[1mm]
            \rho_{\tilde{P}^3} & 0 & -\rho_{\tilde{P}^3} & 0 & \ldots \\[1mm]
            \rho_{\tilde{P}^4} & 0 & 0 & -\rho_{\tilde{P}^4} & \ldots \\[1mm]
            \vdots & \vdots & \vdots & \ddots & \ddots
        \end{array} \right).
    \end{equation*}
    Assuming to consider a solution $\mathbf{f}$ of the kinetic SIR system \eqref{eq:vectorial model simplified} such that, for all $J \in \mathcal{C}$, $f_J \geq 0$ and $f_J \in L^\infty(\R_+, L^1(\Omega \times \mathcal{I}))$, it then follows by standard semigroup theory that the above evolution equation is well-posed in the space of bounded sequences $\ell^\infty$. Indeed, it is easy to check that $\|M_{\tilde{P}}(t)\|_{\ell^\infty} \leq 1$ for any $t \in \R_+$ and $\|A_{\tilde{P}}\|_{\ell^\infty} = 2 \frac{\lambda}{\tau} \sup_{n \geq 2} \rho_{\tilde{P}^n} = 2 \frac{\lambda}{\tau} \rho_{\tilde{P}^2}$ (the second equality comes from the fact that $\tilde{P}(x) \leq 1$), hence $A_{\tilde{P}}$ defines a linear bounded operator on $\ell^\infty$. Additionally, if we suppose that $\rho_{\tilde{P}^n} > 0$ for every $n \geq 1$, we have that that $\ker(A_{\tilde{P}}) = \Span\big((1, 1, 1,\ldots)^\intercal\big)$ and the dynamics of the hierarchy \eqref{eq:hierarchy of moments} identifies a unique nonzero macroscopic equilibrium state given by $M_{\tilde{P}}^\infty = m_{\tilde{P}}^\infty(1, 1, 1,\ldots)^\intercal$, where the value of $m_{\tilde{P}}^\infty \in [-1,1]$ is determined by the initial conditions $f_J^\init$, $J \in \mathcal{C}$, and defines the global kinetic equilibria \eqref{eq:Beta equilibria}. In particular, in the instantaneous relaxation limit $\tau \to 0$ each $m_{\tilde{P}^n}$ is driven toward $m_{\tilde{P}}$ and system \eqref{eq:hierarchy of moments} formally reaches the global consensus equilibrium state $M_{\tilde{P}}^\infty$.

\subsection{Trends to equilibrium} \label{sec:trends}

    We conclude this section with an analysis of the asymptotic behavior of solutions to the compartmental kinetic model \eqref{eq:vectorial model simplified} assuming $\alpha=0,1$ in \eqref{eq:function beta_T}, $G \equiv 1$, and $D(w) = \sqrt{1-w^2}$. Additionally, we shall only consider initial distributions $f_J^\init$, $J \in \mathcal{C}$, that are independent of $x$, since in this case the quantity $m_{\tilde{P}}$ \eqref{eq:weighted density and mean} is conserved over time (recall that $\rho_{\tilde{P}}$ is always conserved over time), as proven in the following lemma.

    \begin{lemma} \label{lem:weight}
        Suppose that $G \equiv 1$ and that the initial distributions $f_J^\init$, $J \in \mathcal{C}$, are independent of $x$, namely $f_J^\init(x,w) = f_J^\init(w)$ for all $J \in \mathcal{C}$. If the weighted average opinion $m_{\tilde{P}}$ defined by \eqref{eq:weighted density and mean} is analytical, then it is conserved over time. 
    \end{lemma}
    \begin{proof}
        If we show that for all $n \in \mathbb{N}^*$, 
        \begin{equation*}
            \frac{\dd^n}{\dd t^n}  \left. m_{\tilde{P}} \right|_{t=0} = 0, 
        \end{equation*}
        then the result will follow from the assumption that $m_{\tilde{P}}$ is analytical. For $n = 1$ we have 
        \begin{equation*}
        \begin{split}
            \frac{\dd}{\dd t} \left. m_{\tilde{P}} \right|_{t=0} &= - \frac{1}{\rho_{\tilde{P}}} \left. \sum_{J \in \mathcal{C}} \int_{\Omega \times \mathcal{I}}  \tilde{P}^2(x) \lambda \sum_{J' \in \mathcal{C}} \tilde{\mathcal{K}}[f_{J'}](t,x,w) f_J(t,x,w) \dd x \dd w \right|_{t=0} \\[2mm] 
            & = - \frac{1}{\rho_{\tilde{P}}} \left. \sum_{J \in \mathcal{C}} \sum_{J' \in \mathcal{C}} \lambda \int_{\Omega \times \mathcal{I}} \tilde{P}^2(x) \int_{\Omega \times \mathcal{I}} \tilde{P}(y) (w-w_*) f_{J'}(t,y,w_*) f_J(t,x,w) \dd y \dd w_*\dd x \dd w \right|_{t=0} \\[2mm] 
            &= - \frac{\lambda}{\rho_{\tilde{P}}} \int_{\Omega^2} \tilde{P}^2(x) \tilde{P}(y) \underbrace{\left( \sum_{J \in \mathcal{C}} \sum_{J' \in \mathcal{C}} \int_{\mathcal{I}^2} (w-w_*) f_{J'}^\init(w_*) f_J^\init(w) \dd w_* \dd w \right)}_{= 0} \dd x \dd y \\ 
            & = 0, 
        \end{split}
        \end{equation*}
        using the hypothesis on the initial distributions. Note that the epidemiological operator $\mathbf{E}$ in \eqref{eq:vectorial model simplified} can be neglected since its contributions cancel out by summing over $J \in \mathcal{C}$. The strategy to prove that also the higher order derivatives cancel out follows the same lines, and relies on exhibiting the sum of symmetric integrals. As an example, we show what happens in the case $n = 2$. Due to the Leibniz product rule, the second derivative is a sum of several terms. However, if for example we fix $J' \in \mathcal{C}$ and we differentiate $f_J$ for all $J \in \mathcal{C}$, then the resulting sum is equal to zero. Indeed, 
        \begin{equation*}
        \begin{split}
            -  \sum_{J \in \mathcal{C}} & \left. \lambda \int_{\Omega \times \mathcal{I}}  \tilde{P}^2(x) \int_{\Omega \times \mathcal{I}}  \tilde{P}(y) (w-w_*) f_{J'}(t,y,w_*) \partial_t f_J(t,x,w) \dd y \dd w_*\dd x \dd w \right|_{t=0} \\[2mm] 
            &= \left.\sum_{J \in \mathcal{C}} \lambda^2 \int_{\Omega \times \mathcal{I}} \tilde{P}^2(x) \int_{\Omega \times \mathcal{I}}  \tilde{P}(y) f_{J'}(t,y,w_*) \tilde{P}(x) \sum_{ J'' \in \mathcal{C}} \tilde{\mathcal{K}}[f_{J''}](t,x,w) f_J(t,x,w) \dd y \dd w_*\dd x \dd w  \right|_{t=0} 
            \\[2mm] 
            & = \lambda^2 \int_{\Omega^3 \times \mathcal{I}} \tilde{P}^3(x)  \tilde{P}(y) 
             \tilde{P}(\Bar{x}) \underbrace{\left( \sum_{J \in \mathcal{C}} \sum_{J'' \in \mathcal{C}} \int_{\mathcal{I}^2} (w-\Bar{w}) f_{J''}^\init(\Bar{w}) f_J^\init(w) \dd\Bar{w} \dd w \right)}_{= 0} f_{J'}^\init(w_*) \dd x \dd y \dd\Bar{x} \dd w_* \\ 
            & =0. 
        \end{split}
        \end{equation*}
        Clearly the computations are similar when fixing $J$ instead of $J'$. Therefore, the result is proved for $n = 2$, and by iterating these computations one gets the desired result for all $n \in \mathbb{N}^*$. Indeed, the key of the proof is that there are always two paired distributions such that their sums cancel out, while the remaining distributions (only $f_{J'}^\init$ in the previous case, but for larger $n$ the number will grow) can be separated from the others. 
    \end{proof}

    \begin{remark}
        If the initial distributions are not independent of the position $x \in \Omega$ on the graphon, then Lemma \ref{lem:weight} is not true since the function $\tilde{P}$ could give more weight to some opinions than to others. The additional hypothesis on the initial distributions ensures that no opinion is privileged at time $t = 0$, and the previous lemma shows that this property is in fact preserved over time. 
    \end{remark}
    
    Recall also that (see Section \ref{sec:macro}) when $\alpha = 0, 1$ we recover at the macroscopic level the classical SIR system \eqref{eq:SIR} and the generalized SIR system \eqref{eq:generalized SIR}, respectively. Taking into account these choices and the structure of $\mathbf{f}^\eq$ \eqref{eq:local equilibria opinion FP-SIR}, we are then lead to consider the particular local equilibria
    \begin{equation} \label{eq:f_eq}
        f_J^\eq(t,x,w)=\rho_J(t) B_J^\infty(x,w), \quad t \in \R_+,\; x \in \Omega,\; w \in \mathcal{I}.
    \end{equation}
    In particular, thanks to Lemma \ref{lem:weight} we have that $m_{\tilde{P}}^\infty = m_{\tilde{P}}(t) = m_{\tilde{P}}(0)$. Note that \eqref{eq:f_eq} correspond to the exact local equilibria \eqref{eq:local equilibria opinion FP-SIR} when $\alpha = 0$, since in this case the local incidence rate \eqref{eq:function beta_T} is independent of $w$ and thus $c_J^\eq(t,x) = \rho_J(t) c_J^{B,\infty}(x)$, but the situation instead changes when $\alpha = 1$, since the presence of the epidemiological operator does not ensure tensorization of the normalization function $c_J^\eq$. However, we have seen that the dynamics of system \eqref{eq:generalized SIR} converge to a unique macroscopic equilibrium $\bm{\rho}^\infty$ (as long as $|m_S| \leq 1$ and $m_I \leq 1$), which in turn uniquely identifies the kinetic global equilibria
    \begin{equation} \label{eq:f_inf}
    \begin{split}
        f_S^\infty(x,w) = \rho_S^\infty B_S^\infty(x,w), \quad f_I^\infty(x,w) = 0, \quad f_R^\infty(x,w) = \rho_R^\infty B_R^\infty(x,w), \quad x \in \Omega,\; w \in \mathcal{I}. 
    \end{split}
    \end{equation}   
    for system \eqref{eq:vectorial model simplified}. Heuristically, one expects that any solution $\mathbf{f}$ to \eqref{eq:vectorial model simplified} would approach over time a local equilibrium $\mathbf{f}^\eq$ given by \eqref{eq:local equilibria opinion FP-SIR}, which cancels out the kinetic operator $\tilde{\mathbf{Q}}(\mathbf{f},\mathbf{f})$ and whose moments $\bm{\rho}$ are solutions to the macroscopic SIR-type systems \eqref{eq:SIR} and \eqref{eq:generalized SIR}. These moments should then converge to the unique globally asymptotically stable equilibrium $\bm{\rho}^\infty$ of these SIR equations, meaning that we would eventually observe a long-time relaxation of the form $\mathbf{f} \underset{t \to +\infty}{\longrightarrow} \mathbf{f}^\infty$.
    
    In order to make this argument rigorous, we proceed as in \cite{AurTosZan, BonBor} by studying the temporal evolution of the relative entropy functional between two distributions $g = g(t,x,w)$ and $h = h(t,x,w)$, defined as
    \begin{equation} \label{eq:relative entropy}
        H(g|h)(t) = \int_{\Omega \times \mathcal{I}} g(t,x,w) \log \frac{g(t,x,w)}{h(t,x,w)} \dd x\dd w, \quad t \in \R_+,
    \end{equation}
    proving that it exhibits a decay to zero as time approaches infinity, and concluding with the decay of the $L^1(\Omega \times \mathcal{I})$ norm via Csiszár--Kullback--Pinsker inequality. The idea will be to investigate the relaxation of $\mathbf{f}$ toward the auxiliary local equilibria \eqref{eq:f_eq}, rather than the correct one \eqref{eq:local equilibria opinion FP-SIR}, to avoid dealing with the normalization function $c_J^\eq$ whose temporal evolution is not known. Eventually we will show that these auxiliary local equilibria still converge toward the correct global equilibria \eqref{eq:f_inf}, thanks to the uniqueness of the global macroscopic equilibrium $\bm{\rho}^\infty$. We shall develop these considerations in three successive steps, starting to analyze the relaxation of $f_S$ toward $f_S^\infty$, proceeding with the study of the long-time behavior of $f_I$ toward zero, and finally deducing the convergence of $f_R$ toward $f_R^\infty$ thanks to mass conservation. For this last step, we will need to assume that for any $J \in \mathcal{C}$ it holds $\sigma_J^2 = \sigma^2$ for some $\sigma^2 > 0$.

    Throughout the analysis, we consider a solution $\mathbf{f} \in L^\infty(\R_+, L^1(\Omega \times \mathcal{I})$ to system \eqref{eq:vectorial model simplified} starting from a nonnegative initial datum $f_J^\init = f_J^\init(w)$, $J \in \mathcal{C}$, independent of the position on the graphon (so that Lemma \ref{lem:weight} holds).

    \medskip
    \noindent \textbf{Step 1 -- Relaxation of $f_S$.} Let us fix $g = f_S$ and $h = f_S^\eq$ in \eqref{eq:relative entropy}. We distinguish between the two cases $\alpha = 0,1$ to make the presentation easier to read.

    \smallskip
    \noindent \textit{Case $\alpha = 0$}. The transition rate \eqref{eq:function beta_T} is simply a constant $\beta_T(w,w_*) = \beta$. We look a the time derivative of the relative entropy $H(f_S|f_S^\eq)$n which reads 
    \begin{equation*}
        \frac{\dd}{\dd t} H(f_S|f_S^\eq) = \int_{\Omega \times \mathcal{I}} \left( 1 + \log \frac{f_S}{f_S^\eq} \right) \partial_t f_S \dd x\dd w - \int_{\Omega \times \mathcal{I}} \frac{f_S}{f_S^\eq} \partial_t f_S^\eq \dd x \dd w.
    \end{equation*}
    Now, using the explicit form of $f_S^\eq$ given by \eqref{eq:f_eq}, we see that the second integral reduces to
    \begin{equation*}
        -\int_{\Omega \times \mathcal{I}} \frac{f_S(t,x,w)}{f_S^\eq(t,x,w)} \partial_t f_S^\eq(t,x,w) \dd x \dd w = -\frac{\partial_t \rho_S}{\rho_S} \int_{\Omega \times \mathcal{I}} f_S \dd x\dd w = -\partial_t \rho_S,
    \end{equation*}
    thanks to the definition of $\rho_S$. Using the first equation of the macroscopic SIR system \eqref{eq:SIR}, we thus conclude that
    \begin{equation*}
        -\int_{\Omega \times \mathcal{I}} \frac{f_S(t,x,w)}{f_S^\eq(t,x,w)} \partial_t f_S^\eq(t,x,w) \dd x\dd w = \beta \rho_S \rho_I.
    \end{equation*}
    Going back now to the first integral, since $f_S$ solves the SIR kinetic system \eqref{eq:vectorial model simplified}, we can rewrite it as
    \begin{equation*}
        \begin{split}
            \int_{\Omega \times \mathcal{I}} & \left( 1 + \log \frac{f_S(t,x,w)}{f_S^\eq(t,x,w)} \right) \partial_t f_S(t,x,w) \dd x\dd w \\[4mm]
            & = \int_{\Omega \times \mathcal{I}} \left(1 + \log \frac{f_S}{f_S^\eq} \right) E_S(\mathbf{f},\mathbf{f}) \dd x \dd w + \int_{\Omega \times \mathcal{I}} \left(1 + \log \frac{f_S}{f_S^\eq} \right) \tilde{Q}_S(\mathbf{f},\mathbf{f}) \dd x \dd w \\[4mm]
            & = - \beta \rho_S \rho_I - \beta \rho_I \int_{\Omega \times \mathcal{I}} f_S \log \frac{f_S}{f_S^\eq} \dd x \dd w - I_H(f_S|f_S^\eq), \\[6mm]
            & = - \beta \rho_S \rho_I - \beta \rho_I H(f_S|f_S^\eq) - I_H(f_S|f_S^\eq),
        \end{split}
    \end{equation*}
    where $I_H(f_S|f_S^\eq) \geq 0$ denotes the entropy production functional
    \begin{equation} \label{eq:I_H}
        I_H(f_S|f_S^\eq)(t) = 4 \int_{\Omega \times \mathcal{I}} (1 - w^2) f_S^\eq(t,x,w) \left( \partial_w \sqrt{\frac{f_S(t,x,w)}{f_S^\eq(t,x,w)}} \right)^2 \dd x \dd w, \quad t \in \R_+.
    \end{equation}
    Combining the computations of the two integral terms, we infer the initial estimate
    \begin{equation} \label{eq:estimate of H}
        \frac{\dd}{\dd t} H(f_S|f_S^\eq) \leq - I_H(f_S|f_S^\eq),
    \end{equation}
    since the term depending on $H(f_S|f_S^\eq) \geq 0$ can be neglected.
    
    From the computations performed in \cite{FurPulTerTos1} one can then prove the inequality
    \begin{equation} \label{eq:relation between H and D}
        D_H^2(f_S|f_S^\eq) \leq \frac{1}{2} I_H(f_S|f_S^\eq),
    \end{equation}
    linking the entropy production $I_H$ with the Hellinger distance between the two distributions $f_S$ and $f_S^\eq$, given by
    \begin{equation*}
        D_H(f_S|f_S^\eq)(t)  = \left( \int_{\Omega \times \mathcal{I}} \left( \sqrt{f_S} - \sqrt{f_S^\eq} \right)^2 \dd x \dd w \right)^{\frac{1}{2}}, \quad t \in \R_+.
    \end{equation*}
    In particular, its temporal evolution reads
    \begin{equation*}
        \frac{\dd}{\dd t} D_H^2(f_S|f_S^\eq) = \int_{\Omega \times \mathcal{I}} \left( 1 - \sqrt{\frac{f_S^\eq}{f_S}} \right) \partial_t f_S \dd x \dd w + \int_{\Omega \times \mathcal{I}} \left( 1 - \sqrt{\frac{f_S}{f_S^\eq}} \right) \partial_t f_S^\eq \dd x\dd w.
    \end{equation*}
    Proceeding in a similar way as with the study of $H(f_S|f_S^\eq)$, we can see the second integral is nonpositive since
    \begin{equation*}
        \begin{split}
            \int_{\Omega \times \mathcal{I}} \left( 1 - \sqrt{\frac{f_S(t,x,w)}{f_S^\eq(t,x,w)}} \right) \partial_t f_S^\eq(t,x,w) \dd x \dd w &= \int_{\Omega \times \mathcal{I}} \left( f_S^\eq - \sqrt{f_S f_S^\eq} \right) \frac{\partial_t f_S^\eq}{f_S^\eq} \dd x \dd w \\[2mm]
            & = - \beta \rho_S \rho_I  + \beta \rho_I \int_{\Omega \times \mathcal{I}} \sqrt{f_S f_S^\eq} \dd x \dd w \\[2mm]
            & \leq - \beta \rho_S \rho_I + \beta \rho_I \underbrace{\left( \int_{\Omega \times \mathcal{I}} f_S \dd x \dd w \right)^{\frac{1}{2}}}_{\sqrt{\rho_S}} \ \underbrace{\left( \int_{\Omega \times \mathcal{I}} f_S^\eq \dd x\dd w \right)^{\frac{1}{2}}}_{\sqrt{\rho_S}}  \\ 
            & = 0
        \end{split}
    \end{equation*}
    where we have replaced $\partial_t \rho_S = - \beta \rho_S \rho_I$ and used Cauchy--Schwarz inequality to bound the integral of $\sqrt{f_S f_S^\eq}$. Using the first equation of the simplified SIR kinetic system \eqref{eq:vectorial model simplified} to replace $\partial_t f_S$ inside the first integral and adapting the computations performed in \cite{FurPulTerTos1}, which can be proved to hold (in particular the weighted Chernoff inequality from \cite[Theorem 3.3]{FurPulTerTos1}) for two general distributions having the same density, we can then infer the estimate
    \begin{equation} \label{eq:estimate of D}
        \begin{split}
            \frac{\dd}{\dd t}  D_H^2(f_S|f_S^\eq)  \leq &  \int_{\Omega \times \mathcal{I}} \left( 1 - \sqrt{\frac{f_S^\eq}{f_S}} \right) \partial_t f_S \dd x\dd w \\[2mm]
             = & \int_{\Omega \times \mathcal{I}} \left( 1 - \sqrt{\frac{f_S^\eq}{f_S}} \right) E_S(\mathbf{f},\mathbf{f}) \dd x \dd w + \int_{\Omega \times \mathcal{I}} \left( 1 - \sqrt{\frac{f_S^\eq}{f_S}} \right) \tilde{Q}_S(\mathbf{f},\mathbf{f}) \dd x \dd w \\[2mm]
             = & - \beta \rho_S \rho_I + \beta \rho_I \int_{\Omega \times \mathcal{I}} \sqrt{f_S f_S^\eq} \dd x \dd w + \int_{\Omega \times \mathcal{I}} \left( 1 - \sqrt{\frac{f_S^\eq}{f_S}}  \right)  \tilde{Q}_S(\mathbf{f},\mathbf{f}) \dd x \dd w \\[4mm]
             \leq & - I_D(f_S|f_S^\eq),
        \end{split}
    \end{equation}
    where we have used once again Cauchy--Schwarz inequality as in the previous estimate to prove that the first two contributions cancel out, and we have introduced the functional $I_D(f_S|f_S^\eq) \geq 0$ defining the entropy production of the squared Hellinger distance, which is given by
    \begin{equation} \label{eq:I_D}
        I_D(f_S|f_S^\eq)(t) = 8 \int_{\Omega \times \mathcal{I}} (1 - w^2) f_S^\eq(t,x,w) \left( \partial_w \sqrt[4]{\frac{f_S(t,x,w)}{f_S^\eq(t,x,w)}} \right)^2 \dd x \dd w, \quad t \in \R_+.
    \end{equation}
    
    Now, integrating equation \eqref{eq:estimate of H} over $t \in \R_+$, we infer that $I_H(f_S|f_S^\eq)$ belongs to $L^1(\R_+)$ since $H(f_S|f_S^\eq) \geq 0$ and thus
    \begin{equation*}
        \int_{\R_+} I_H(f_S|f_S^\eq)(t) \dd t \leq H(f_S|f_S^\eq)(0) < +\infty,
    \end{equation*}
    Therefore, from the inequality \eqref{eq:relation between H and D} relating the squared Hellinger distance with the entropy production $I_H$ and from the entropy-entropy production estimate \eqref{eq:estimate of D} on $D_H^2$, we respectively deduce that the Hellinger distance is integrable in time and non-increasing. This allows us to conclude that $D_H^2(f_S|f_S^\eq) = \smallO(1/t)$ when $t \to +\infty$, recovering the following $L^1(\Omega \times \mathcal{I})$ estimate for the convergence of $f_S$ toward the local equilibrium $f_S^\eq$: 
    \begin{equation*}
        \norm{f_S - f_S^\eq}_{L^1(\Omega \times \mathcal{I})} \leq 2 \sqrt{\rho_S^\init} D(f_S|f_S^\eq) = \smallO(1/\sqrt{t}), \quad t \to +\infty,
    \end{equation*}
    which is obtained from the equality $|a - b| = |\sqrt{a} - \sqrt{b}|(\sqrt{a} + \sqrt{b})$ and from two successive applications of Cauchy--Schwarz inequality, also recalling that $\rho_S \leq \rho_S^\init$ for any $t \in \R_+$.
    
    At last, we can study the relaxation of the local equilibrium $f_S^\eq$ toward the global one $f_S^\infty$ via simple computations of their relative $L^1(\Omega \times \mathcal{I})$ distance. We obtain
    \begin{equation*}
   	 \norm{f_S^\eq - f_S^\infty}_{L^1({\Omega \times \mathcal{I}})} = \big| \rho_S - \rho_S^\infty \big| \underbrace{\int_{\Omega \times \mathcal{I}} B_S^\infty(x,w) \dd x \dd w}_{= 1} \underset{t \to +\infty}{\longrightarrow} 0,
    \end{equation*}
    using the fact that $\rho_S$ converges toward $\rho_S^\infty$ in \eqref{eq:SIR}. Therefore, one finally deduces the long-time relaxation to equilibrium $\norm{f_S - f_S^\infty}_{L^1({\Omega \times \mathcal{I}})} \underset{t \to +\infty}{\longrightarrow} 0$.

    \smallskip
    \noindent \textit{Case $\alpha = 1$}. The transition rate \eqref{eq:function beta_T} now writes $\beta_T(w,w_*) = \beta (1-w)(1-w_*)$. We start by noticing that because $f_S, f_I \geq 0$ and $f_S, f_I \in L^\infty(\R_+, L^1(\Omega \times \mathcal{I}))$, all their moments remain bounded and in particular $\rho_S, \rho_I \in [0,1]$ and $m_S, m_I \in [-1,1]$. Therefore, the solutions to the generalized SIR system \eqref{eq:generalized SIR} satisfy the properties detailed at the end of Section \ref{sec:macro}. In particular, it holds that $\rho_I \in L^1(\R_+)$ and there exists a unique asymptotically stable global equilibrium $\bm{\rho}^\infty$.
    
    Recalling the time derivative of the relative entropy $H(f_S|f_S^\eq)$
    \begin{equation*}
        \frac{\dd}{\dd t} H(f_S|f_S^\eq) = \int_{\Omega \times \mathcal{I}} \left( 1 + \log \frac{f_S}{f_S^\eq} \right) \partial_t f_S \dd x\dd w - \int_{\Omega \times \mathcal{I}} \frac{f_S}{f_S^\eq} \partial_t f_S^\eq \dd x \dd w,
    \end{equation*}
    from the expression of $f_S^\eq$ given by \eqref{eq:f_eq} we can compute the second integral as
    \begin{equation*}
        -\int_{\Omega \times \mathcal{I}} \frac{f_S(t,x,w)}{f_S^\eq(t,x,w)} \partial_t f_S^\eq(t,x,w) \dd x \dd w = -\partial_t \rho_S = \beta (1-m_S) (1-m_I) \rho_S \rho_I,
    \end{equation*}
    where we have used the first equation of the macroscopic system \eqref{eq:SIR}. Then, similarly to the previous case $\alpha = 0$, the first integral rewrites
    \begin{equation*}
        \begin{split}
            \int_{\Omega \times \mathcal{I}} & \left( 1 + \log \frac{f_S(t,x,w)}{f_S^\eq(t,x,w)} \right) \partial_t f_S(t,x,w) \dd x\dd w \\[4mm]
            & = \int_{\Omega \times \mathcal{I}} \left(1 + \log \frac{f_S}{f_S^\eq} \right) E_S(\mathbf{f},\mathbf{f}) \dd x \dd w + \int_{\Omega \times \mathcal{I}} \left(1 + \log \frac{f_S}{f_S^\eq} \right) \tilde{Q}_S(\mathbf{f},\mathbf{f}) \dd x \dd w \\[4mm]
            & = -\beta (1-m_S)(1-m_I) \rho_S \rho_I - \beta \rho_I(1-m_I) \int_{\Omega \times \mathcal{I}} (1-w) f_S \log \frac{f_S}{f_S^\eq} \dd x \dd w - I_H(f_S|f_S^\eq),
        \end{split}
    \end{equation*}
    where $I_H(f_S|f_S^\eq) \geq 0$ denotes the same entropy production functional \eqref{eq:I_H}. Now, adding and subtracting $f_S - f_S^\eq$ inside second integral term, we can successively
    \begin{equation*}
        \begin{split}
            \int_{\Omega \times \mathcal{I}} (1-w) f_S(t,x,w) \log \frac{f_S(t,x,w)}{f_S^\eq(t,x,w)} \dd x \dd w &= \int_{\Omega \times \mathcal{I}} (1-w) f_S^\eq \left( \frac{f_S}{f_S^\eq} \log \frac{f_S}{f_S^\eq} - \frac{f_S}{f_S^\eq} + 1 \right) \dd x \dd w \\[2mm]
            & \qquad\qquad\qquad\qquad\qquad\qquad+ \int_{\Omega \times \mathcal{I}} (1-w) (f_S  - f_S^\eq) \dd x \dd w \\[4mm]
            &\geq - \rho_S (m_S - m_{\tilde{P}}^\infty),
        \end{split}
    \end{equation*}
    using that $y \log y - y + 1 \geq 0$ for any $y \in \R_+$, leading us to infer the estimate
    \begin{equation} \label{eq:estimate of H 2}
        \frac{\dd}{\dd t} H(f_S|f_S^\eq) \leq - I_H(f_S|f_S^\eq) + \beta (1-m_I) (m_S - m_{\tilde{P}}^\infty) \rho_S \rho_I,
    \end{equation}
    where we have discarded the other terms because they are nonpositive. Integrating the latter over $t \in \R_+$, we find once again that $I_H(f_S|f_S^\eq)$ belongs to $L^1(\R_+)$ since $H(f_S|f_S^\eq) \geq 0$ and $\rho_I \in L^1(\R_+)$:
    \begin{equation*}
        \int_{\R_+} I_H(f_S|f_S^\eq)(t) \dd t \leq H(f_S|f_S^\eq)(0) + 4\beta \int_{\R_+} \rho_I(t) \dd t < +\infty,
    \end{equation*}
    where we have used that $\rho_S \in [0,1]$ and $m_S, m_{\tilde{P}}^\infty \in [-1,1]$.
    
    From inequality \eqref{eq:relation between H and D} we thus deduce that the Hellinger distance $D^2_H(f_S|f_S^\eq)$ also belongs to $L^1(\R_+)$. The remaining piece consists in studying its temporal evolution, which we recall to be given by
    \begin{equation*}
        \frac{\dd}{\dd t} D_H^2(f_S|f_S^\eq) = \int_{\Omega \times \mathcal{I}} \left( 1 - \sqrt{\frac{f_S^\eq}{f_S}} \right) \partial_t f_S \dd x \dd w + \int_{\Omega \times \mathcal{I}} \left( 1 - \sqrt{\frac{f_S}{f_S^\eq}} \right) \partial_t f_S^\eq \dd x\dd w.
    \end{equation*}
    It is then easy to check in the same way as for $\alpha = 0$ that the second integral is nonpositive
    \begin{equation*}
        \begin{split}
            \int_{\Omega \times \mathcal{I}} & \left( 1 - \sqrt{\frac{f_S(t,x,w)}{f_S^\eq(t,x,w)}} \right) \partial_t f_S^\eq(t,x,w) \dd x \dd w \\[2mm]
            & = - \beta (1-m_S)(1-m_I) \rho_S \rho_I  + \beta (1-m_S)(1-m_I) \rho_I \int_{\Omega \times \mathcal{I}} \sqrt{f_S f_S^\eq} \dd x \dd w \\[2mm]
            & \leq - \beta (1-m_S)(1-m_I) \rho_S \rho_I  + \beta (1-m_S)(1-m_I) \rho_I \underbrace{\left( \int_{\Omega \times \mathcal{I}} f_S \dd x \dd w \right)^{\frac{1}{2}}}_{\sqrt{\rho_S}} \ \underbrace{\left( \int_{\Omega \times \mathcal{I}} f_S^\eq \dd x\dd w \right)^{\frac{1}{2}}}_{\sqrt{\rho_S}}  \\ 
            & = 0.
        \end{split}
    \end{equation*}
    Explicit computations on the first integral instead give \cite{FurPulTerTos1}
    \begin{equation*}
        \begin{split}
            \int_{\Omega \times \mathcal{I}} & \left( 1 - \sqrt{\frac{f_S^\eq(t,x,w)}{f_S(t,x,w)}} \right) \partial_t f_S(t,x,w) \dd x \dd w \\[2mm]
             = & \int_{\Omega \times \mathcal{I}} \left( 1 - \sqrt{\frac{f_S^\eq}{f_S}} \right) E_S(\mathbf{f},\mathbf{f}) \dd x \dd w + \int_{\Omega \times \mathcal{I}} \left( 1 - \sqrt{\frac{f_S^\eq}{f_S}} \right) \tilde{Q}_S(\mathbf{f},\mathbf{f}) \dd x \dd w \\[2mm]
             \leq & -\beta (1-m_S)(1-m_I) \rho_S \rho_I - \beta \rho_I(1-m_I) \int_{\Omega \times \mathcal{I}} (1-w) \sqrt{f_S f_S^\eq} \dd x \dd w - I_D(f_S|f_S^\eq),
        \end{split}
    \end{equation*}
    where the entropy production of the squared Hellinger distance is still defined by \eqref{eq:I_D}. We can then apply Cauchy--Schwarz inequality to bound
    \begin{equation*}
        \int_{\Omega \times \mathcal{I}} (1-w) \sqrt{f_S f_S^\eq} \dd x \dd w \leq \underbrace{\left( \int_{\Omega \times \mathcal{I}} (1-w) f_S \dd x \dd w \right)^{\frac{1}{2}}}_{\sqrt{\rho_S(1-m_S)}} \ \underbrace{\left( \int_{\Omega \times \mathcal{I}} (1-w) f_S^\eq \dd x\dd w \right)^{\frac{1}{2}}}_{\sqrt{\rho_S(1-m_{\tilde{P}}^\infty)}},
    \end{equation*}
    and thus combine the first two terms into
    \begin{equation*}
        -\beta (1-m_S)(1-m_I) \rho_S \rho_I - \beta \rho_I(1-m_I) \int_{\Omega \times \mathcal{I}} (1-w) \sqrt{f_S f_S^\eq} \dd x \dd w \leq \beta (1-m_I) \sqrt{|m_S - m_{\tilde{P}}^\infty|} \rho_S \rho_I,
    \end{equation*}
    using the simple inequality $\sqrt{a} - \sqrt{b} \leq \sqrt{|a-b|}$ applied to $a = 1-m_S$ and $b = 1-m_{\tilde{P}}^\infty$.

    We can finally estimate the time derivative of $D^2(f_S|f_S^\eq)$ as
    \begin{equation*}
        \frac{\dd}{\dd t} D_H^2(f_S|f_S^\eq) \leq -I_D(f_S|f_S^\eq) + \beta (1-m_I) \sqrt{|m_S - m_{\tilde{P}}^\infty|} \rho_S \rho_I < +\infty,
    \end{equation*}
    thanks to the boundedness of the moments and the nonnegativity of $I_D(f_S|f_S^\eq)$.
    
    The Hellinger distance is thus integrable in time and has bounded derivative, so it must vanish asymptotically when $t \to +\infty$. In particular $D_H^2(f_S|f_S^\eq) = \smallO(1/t)$ when $t \to +\infty$, hence $f_S$ converges in $L^1(\Omega \times \mathcal{I})$ toward the auxiliary local equilibrium \eqref{eq:f_eq} as
    $\norm{f_S - f_S^\eq}_{L^1(\Omega \times \mathcal{I})} = \smallO(1/\sqrt{t})$. Because $\rho_S \underset{t \to +\infty}{\longrightarrow} \rho_S^\infty$ in \eqref{eq:generalized SIR}, the auxiliary local equilibrium  relaxes toward the global one \eqref{eq:f_inf} at the same rate, and one concludes once again with the long-time relaxation to equilibrium $\norm{f_S - f_S^\infty}_{L^1({\Omega \times \mathcal{I}})} \underset{t \to +\infty}{\longrightarrow} 0$.

    \medskip  
    \noindent \textbf{Step 2 -- Relaxation of $f_I$.} The study of the long-time behavior of $f_I$ is very simple, since $\rho_I^\infty = 0$ and we would thus want to prove that $f_I$ relaxes to zero in the $L^1({\Omega \times \mathcal{I}})$ norm as $t \to +\infty$. But since the SIR kinetic system \eqref{eq:vectorial model simplified} preserves the mass and the positivity of solutions, one has that $\norm{f_I}_{L^1({\Omega \times \mathcal{I}})} = \rho_I$ which converges to the unique disease-free equilibrium of the corresponding macroscopic SIR system \eqref{eq:SIR}. We thus easily conclude that $\norm{f_I}_{L^1({\Omega \times \mathcal{I}})} \underset{t \to +\infty}{\longrightarrow} 0$.

    \medskip  
    
    \noindent \textbf{Step 3 -- Relaxation of $f_R$.} It remains to show that also the distribution $f_R$ of recovered individuals converges toward its corresponding global equilibrium state $f_R^\infty$, which we shall prove by exploiting the fact that the SIR kinetic model preserves the total mass of the population. Specifically, we introduce the distribution $f(t,x,w)= \sum_{J \in \mathcal{C}} f_J (t,x,w)$ of the whole population. Summing the equations of system \eqref{eq:vectorial model simplified}, since $\sum_{J \in \mathcal{C}} E_J(\mathbf{f},\mathbf{f}) \equiv 0$ and thanks to the particular structure of the operators $\tilde{Q}_J$ with parameters independent of $J \in \mathcal{C}$ (recall that we assumed $\sigma_J^2 = \sigma^2$, $J \in \mathcal{C}$), one finds that the total distribution $f$ solves the Fokker--Planck equation
    \begin{equation*}
        \partial_t f = \frac{1}{\tau} \left( \lambda \rho_{\tilde{P}} \partial_w \left( (w - m_{\tilde{P}}^\infty) f \right) + \frac{\sigma^2}{2} \partial_w^2 \left( (1 - w^2) f \right) \right), \quad t \in \R_+,\; x \in \Omega,\; w \in \mathcal{I}.
    \end{equation*}
    In particular, thanks to mass conservation ensuring that $\sum_{J \in \mathcal{C}} \rho_J = \sum_{J \in \mathcal{C}} \rho_J^\infty = 1$ for all $t \in \R_+$, one expects $f$ to converge toward the global equilibrium state $f^\infty(x,w) = B^\infty(x,w)$ defined by \eqref{eq:Beta equilibria} (recall once again that $\sigma_J^2 = \sigma^2$ for any $J \in \mathcal{C}$, implying that $\nu_J = \nu$ for some $\nu>0$; similarly, for any $x \in \Omega$ we have that $c_J^{B,\infty}(x) = c^{B,\infty}(x)$ for some normalization function $c(x)>0$). Following the exact same computations performed in Step 1 to prove the convergence of $f_S$ toward $f_S^\eq$ (except that now we do not have to deal with the epidemiological operators), one can show that $\norm{f - f^\infty}_{L^1({\Omega \times \mathcal{I}})} \underset{t \to +\infty}{\longrightarrow} 0$.
    
    Therefore, using the definition of $f$ and the fact that $f^\infty(x,w) = \sum_{J \in \mathcal{C}} \rho_J^\infty B_J^\infty(x,w)$ for any $x \in \Omega$ and $w \in \mathcal{I}$, we easily deduce the long-time relaxation of $f_R$ toward $f_R^\infty$ as
    \begin{equation*}
    \begin{split}
    \norm{f_R - f_R^\infty}_{L^1({\Omega \times \mathcal{I}})} & = \norm{f - f_S - f_I - (f^\infty - f_S^\infty)}_{L^1({\Omega \times \mathcal{I}})} \\[4mm]
    & \leq \underbrace{\norm{f_S - f_S^\infty}_{L^1({\Omega \times \mathcal{I}})}}_{\underset{\textrm{from Step 1}}{\underset{t \longrightarrow +\infty}{\longrightarrow} 0}} + \underbrace{\norm{f_I}_{L^1({\Omega \times \mathcal{I}})}}_{\underset{\textrm{from Step 2}}{\underset{t \longrightarrow +\infty}{\longrightarrow} 0}} + \underbrace{\norm{f - f^\infty}_{L^1({\Omega \times \mathcal{I}})}}_{\underset{\textrm{from Step 3}}{\underset{t \longrightarrow +\infty}{\longrightarrow} 0}} \underset{t \longrightarrow +\infty}{\longrightarrow} 0.
    \end{split}
    \end{equation*}
    
    We are therefore ready to prove the following theorem. 

    \begin{theorem} \label{teo:SIR}
        Let $\mathbf{f}$ be a solution to the simplified SIR kinetic model \eqref{eq:vectorial model simplified}, starting from a nonnegative initial datum that is independent of the position on the graphon, namely $f_J^\init = f_J^\init(w)$ for all $J \in \mathcal{C}$. Assume that $G \equiv 1$, $D(w) = \sqrt{1-w^2}$, $\alpha = 0,1$ in \eqref{eq:function beta_T} and, if $\alpha=1$, that $|m_J(t)| \leq 1$, $J = S, I$, for any $t \in \R_+$. Moreover, suppose that $\sigma_J^2 = \sigma^2$, $J \in \mathcal{C}$, for some $\sigma^2 > 0$. Then, for any $J \in \mathcal{C}$, 
        \begin{equation*}
            \norm{f_J - f_J^\infty}_{L^1(\Omega \times \mathcal{I})} = \smallO(1/\sqrt{t}), \quad t \to + \infty,
        \end{equation*}
        where the global equilibria $f_J^\infty$, $J \in \mathcal{C}$, are given by \eqref{eq:f_inf}. 
    \end{theorem}
    \begin{proof}
        First of all, thanks to Proposition \ref{prop:nonnegativity} we know that $f_J \geq 0$, $J \in \mathcal{C}$, and thus $\mathbf{f} \in L^\infty(\R_+,L^1(\Omega \times \mathcal{I})$ as a consequence of mass conservation. In particular, both macroscopic SIR systems \eqref{eq:SIR} and \eqref{eq:generalized SIR} uniquely identify the global equilibrium state $\bm{\rho}^\infty$. Moreover, since the initial datum is independent of $x$, Lemma \ref{lem:weight} ensures that the global kinetic equilibria \eqref{eq:f_inf} are well-defined and unique. Therefore, all our previous computations are rigorously justified (note that the hypothesis that $|m_J(t)| \leq 1$ for $J = S, I$ remains crucial in the case when $\alpha = 1$).
        
        Now, we have already proved that $\norm{f_J - f_J^\infty}_{L^1(\Omega \times \mathcal
        I)} \underset{t \longrightarrow +\infty}{\longrightarrow} 0$ for any $J \in \mathcal{C}$. It remains to demonstrate the rates of convergence. The rate for $f_S$ was already deduced in Step 1. Concerning $f_I$, since it is integrable, then $\norm{f_I - f_I^\infty}_{L^1(\Omega \times \mathcal
        I)} = \norm{f_I}_{L^1(\Omega \times \mathcal
        I)} = \smallO(1/t)$ when $t \to + \infty$. As shown in \cite{FurPulTerTos1}, $\norm{f - f^\infty}_{L^1(\Omega \times \mathcal
        I)} = \smallO(1/\sqrt{t})$ when $t \to + \infty$, hence also $\norm{f_R - f_R^\infty}_{L^1(\Omega \times \mathcal
        I)} = \smallO(1/\sqrt{t})$ when $t \to + \infty$, and the thesis follows. 
    \end{proof} 

\section{Perception of the disease's dangerousness} \label{section4}

    \noindent Until now, we have analyzed how the agents' opinion about the willingness to adopt or not a protective behavior influences the evolution of an epidemic. In this section we want to investigate how seriously the disease is perceived by the population. Even if this dynamic does not directly affect the evolution of the epidemic, it would be interesting to determine the relationship between the global equilibria derived in Section \ref{polarization} and the population's opinion about the disease itself. Our goal is to fix a value $\hat{w} \in [-1,1]$, representing a hypothesized danger level of the disease, and understand how any products related to this hypothesis spread inside the population. Products may be any piece of information reaching the agents: news, videos, advertisements, and social posts, that individuals possibly share depending on how much their current opinion is aligned with them. Agents with an opinion $w \geq \hat{w}$ will contribute to the spread of these products, since they consider the disease at least as dangerous as $\hat{w}$. On the other hand, agents with opinion $w < \hat{w}$ consider the disease less dangerous than $\hat{w}$, so they will not contribute to spread these products. For example, by fixing $\hat{w} = 0$ we could measure the popularity of the ideas that judge the disease to be dangerous. We will assume that the evolution of these products depends on the propensity to interact $p$ of each agent, even if we do not necessarily have to consider the simplified version of our model (recall that $p$ can in fact be defined even for the original model). 

\subsection{Kinetic modeling of popularity spread}

    We quantify the popularity of a product by means of a variable $v \in \R_+$, whose evolution depends on the interaction with the opinion $w$ and the propensity to interact $p$ of the agents that the product reaches. Inspired by \cite{TosTosZan}, we consider the following microscopic update rule for $v$: 
    \begin{equation} \label{eq:microscopic interactions popularity}
        v' = v - \mu v + \theta p(x) \mathbb{I}_{[\hat{w},1]}(w) + D_p(v) \omega, 
    \end{equation}
    where $\mu \in (0,1)$ represents the natural decay rate of the product's popularity when it is not shared by an agent, $\theta > 0$ measures the incidence of $p$ on the popularity, $\mathbb{I}_A$ denotes the indicator function of the set $A \subseteq \mathcal{I}$, and $\omega$ is a random variable with zero mean and finite variance $\zeta^2>0$, modeling a stochastic fluctuation of the popularity that is modulated by a popularity-dependent strength $D_p(v) \geq 0$. According to \eqref{eq:microscopic interactions popularity}, the increase in popularity depends on whether an individual decides to share the product, which happens when their opinion about the disease's dangerousness is equal or larger than the hypothesized danger level $\hat{w}$. In such a case, the increase in popularity is proportional to the individual's propensity to interact. In order to guarantee that $v' \geq 0$ it suffices to impose that $D_p(v) \omega \geq (\mu - 1) v$, which is satisfied as soon as $\omega \geq \mu-1$ and $D_p(v) \leq v$. Note that this latter inequality implies $D_p(0) = 0$. We may assume 
    \begin{equation}
        D_p(v) = v, \quad v \in \R_+,
    \end{equation}
    since, as we shall see, with this particular choice the equilibria are coherent with empirical evidence. Notice however that other choices are possible \cite{TosTosZan}. Obviously we also have to assume that $p$ is bounded, although this is not a restrictive assumption since, if needed, it suffices to consider a modified version $p_\Lambda$ of $p$ such that for any $x \in \Omega$, $p_\Lambda(x) = \min \{p(x), \Lambda\}$ for some $\Lambda \gg 1$.

    Given a distribution function $h = h(t,v)$, depending on time $t \in \R_+$ and popularity $v \in \R_+$, and such that $h(t,v) \dd v$ represents the number of products with popularity in $[v, v+\dd v]$ at time $t$, the evolution of $h$ should depend on the opinion distribution inside the population, namely we have to consider the coupling 
    \begin{align} \label{eq:FP-SIR and popularity}
    \left\{
    \begin{aligned}
         & \partial_t \mathbf{f} = \mathbf{E}(\mathbf{f},\mathbf{f}) + \frac{1}{\tau} \mathbf{Q}(\mathbf{f},\mathbf{f}), \quad t\in \R_+,\; x \in \Omega,\; w \in \mathcal{I}, \\[2mm] 
         & \partial_t h  = \frac{1}{\tau_p} Q_p (h, \mathbf{f}), \qquad \qquad \ t \in \R_+,\; v \in \R_+, 
    \end{aligned}
    \right.
    \end{align}
    where $\tau_p > 0$ represents a possible additional timescale at which the popularity of products evolve. Note that in \eqref{eq:FP-SIR and popularity} the collisional operator $\mathbf{Q}$ can be replaced by its simplified version $\tilde{\mathbf{Q}}$ introduced in Section \ref{connectivity}. On the other hand, the collisional operator $Q_p$ can be derived as follows. Proceeding as in Section \ref{kin_op}, the second equation from system \eqref{eq:FP-SIR and popularity} reads in weak form
    \begin{equation} \label{eq:weak form popularity}
        \frac{\dd}{\dd t} \int_{\R_+} \psi(v) h(t,v) \dd v = \frac{1}{\tau_p}\sum_{J \in \mathcal{C}} \int_{\Omega \times \mathcal{I} \times \R_+} \left<\psi(v')-\psi(v)\right> h(t,v) f_J(t,x,w) \dd x \dd w \dd v, 
    \end{equation}
    where $\psi=\psi(v)$ is a smooth test function. Note that the choice $\psi(v)=1$ implies the conservation of mass for the distribution $h$, hence we can assume that $\int_{\R_+} h(t,v) \dd v = 1$. Introducing the scaling
    \begin{equation} \label{scaling_h}
        \mu \mapsto \epsilon\mu, \qquad \theta \mapsto \epsilon\theta, \qquad \zeta^2 \mapsto \epsilon\zeta^2, \qquad \tau_p \mapsto \epsilon\tau_p,
    \end{equation}
    which depends on a small parameter $0 < \epsilon \ll 1$, we can follow the computations of Section \ref{sub3} to formally derive, in the limit $\epsilon \to 0$, a Fokker--Planck-type equation describing the evolution of the distribution $h$, namely 
    \begin{equation} \label{eq:popularity FP}
    \begin{split}
        \partial_t h &= \frac{1}{\tau_p} \left( \partial_v \big( \left(\mu v - \theta \mathcal{F}[f] \right)h \big) + \frac{\zeta^2}{2} \partial^2_v \left(D_p^2(v) h\right)\right) \\[2mm] 
        & = \frac{1}{\tau_p} Q_p(h, f), \quad t \in \R_+,\; v \in \R_+, 
    \end{split}
    \end{equation}
    where we have defined the nonlocal quantity
    \begin{equation*} 
        \mathcal{F}[f](t) = \int_{\Omega \times \mathcal{I}} p(x) \mathbb{I}_{[\hat{w},1]}(w) f(t,x,w) \dd x \dd w, \quad t \in \R_+,
    \end{equation*}
    with $f(t,x,w) = \sum_{J \in \mathcal{C}} f_J (t,x,w)$. Model \eqref{eq:popularity FP} is then completed with the following no-flux boundary conditions, holding for any $t \in \R_+$: 
    \begin{align} \label{eq:BC popularity FP}
    \left\{
    \begin{aligned}
        & \left.D_p^2(v) h(t,v)\right|_{v=0} = 0, \\[2mm] 
        & \left.(\mu v - \theta \mathcal{F}[f](t)) h(t,v) + \frac{\zeta^2}{2} \partial_v \left(D_p^2(v) h(t,v)\right) \right|_{v=0} = 0, \\[2mm]
        & \displaystyle\lim_{v \to +\infty} v^{\kappa -1} D_p^2(v) h(t,v) = 0, \\[2mm]
        & \displaystyle\lim_{v \to +\infty} v^\kappa \left( (\mu v - \theta \mathcal{F}[f](t)) h(t,v) + \frac{\zeta^2}{2} \partial_v \left(D_p^2(v) h(t,v)\right) \right) = 0,  
    \end{aligned}
    \right.
    \end{align}
    where $\kappa \geq 1$ must be sufficiently large so that the distribution $h$ has finite moments of interest. Note that the first boundary condition is related to $D_p(0) = 0$. It is important to notice that the evolution of $h$ depends only on the opinion operator $\mathbf{Q}$. Moreover, one can bound $\mathcal{F}[f]$ as
    \begin{equation} \label{eq:estimate of F}
        0 \leq \mathcal{F}[f](t) \leq \int_{\Omega \times \mathcal{I}} p(x) f(t,x,w) \dd x \dd w,
    \end{equation}
    for any $t \in \R_+$, where the last integral is constant since the position on the graphon $x$ does not vary over time (similarly to the total weighted density $\rho_{\tilde{P}}$ defined by \eqref{eq:weighted density and mean}). 

    Regarding the macroscopic moments, we introduce the average popularity 
    \begin{equation}
        m_p(t) = \int_{\R_+} v h(t,v) \dd v, \quad t \in \R_+,
    \end{equation}
    whose evolution is obtained by taking $\psi(v)=v$ in \eqref{eq:weak form popularity}, to get
    \begin{equation} \label{eq:evolution of m_p}
        \frac{\dd}{\dd t}m_p = - \frac{\mu}{\tau_p} m_p + \frac{\theta}{\tau_p} \mathcal{F}[f], \quad t \in \R_+. 
    \end{equation}
    Similarly, if we set $D_p(v)=v$ and we assume that $\kappa \geq 2$, from \eqref{eq:popularity FP} it follows that the energy of the distribution $h$ 
    \begin{equation}
        e_p(t) = \int_{\R_+} v^2 h(t,v) \dd v, \quad t \in \R_+,
    \end{equation}
    evolves according to
    \begin{equation} \label{eq:evolution of e_p}
        \frac{\dd}{\dd t}e_p = \frac{\zeta^2 -2 \mu}{\tau_p} e_p + \frac{2 \theta}{\tau_p}  \mathcal{F}[f] m_p, \quad t \in \R_+.
    \end{equation}
    These considerations will be useful later on. 

\subsection{Local and global equilibria} \label{sub_popul_eq}

    Let us set $D_p(v)=v$ and assume that $\mathcal{F}[f](t) > 0$ for any $t \in \R_+$. Then, the \emph{local equilibrium} $h^\eq$ of the Fokker--Planck equation \eqref{eq:popularity FP} is the (unique) solution of unitary mass to 
    \begin{equation}
        \left(\mu v - \theta \mathcal{F}[f](t) \right)h(t,v) + \frac{\zeta^2}{2} \partial_v \left(D_p^2(v) h(t,v)\right) = 0,
    \end{equation}
    for any $t \in \R_+$ and $v \in \R_+$, which is given by the inverse gamma distribution
    \begin{equation} \label{eq:local equilibria popularity FP}
        h^\eq (t,v) = \frac{\left(2 \theta \mathcal{F}[f](t)\right)^{1 + \frac{2 \mu}{\zeta^2}}}{\zeta^{2 + \frac{4 \mu}{\zeta^2}} \Gamma\left(1 + \frac{2 \mu}{\zeta^2}\right)} \; v^{-2 \left(1+\frac{\mu}{\zeta^2}\right)} \exp \left(-\frac{2 \theta \mathcal{F}[f](t)}{\zeta^2}\frac{1}{v}\right), \quad t \in \R_+,\; v \in \R_+,
    \end{equation}
    where $\Gamma$ denotes the gamma function. 
    Note that $\int_{\R_+} h^\eq(t,v) \dd v = 1$.
    Assuming now that $f$ reached a global equilibrium $f^\infty$ and defining 
    \begin{equation*} 
        \mathcal{F}^\infty = \int_{\Omega \times \mathcal{I}} p(x) \mathbb{I}_{[\hat{w},1]}(w) f^\infty(x,w) \dd x \dd w, 
    \end{equation*}
    the \emph{global equilibrium} of equation \eqref{eq:popularity FP} reads 
    \begin{equation} \label{eq:global equilibria popularity FP}
        h^\infty(v) = \frac{\left(2 \theta \mathcal{F}^\infty\right)^{1 + \frac{2 \mu}{\zeta^2}}}{\zeta^{2 + \frac{4 \mu}{\zeta^2}} \Gamma\left(1 + \frac{2 \mu}{\zeta^2}\right)} \; v^{-2 \left(1+\frac{\mu}{\zeta^2}\right)} \exp \left(-\frac{2 \theta \mathcal{F}^\infty}{\zeta^2}\frac{1}{v}\right), \quad v \in \R_+,
    \end{equation}
    which has also unit mass. We point out that $h^\infty$ is admissible as long as $\mathcal{F}^\infty > 0$, i.e., as long as $\hat{w}<1$, otherwise it would not have a finite mass. Notice also that such global equilibrium is a fat-tailed inverse gamma distribution, since $h^\infty(v) \sim v^{-2(1+\mu/\zeta^2)}$ when $v \to +\infty$. More precisely, this distribution exhibits a Pareto tail \cite{GuaTos} which indicates that products reaching very high popularity levels may be rare in general but not that improbable. The mean of the distribution $h^\infty$ is $m_p^\infty = \theta \mathcal{F}^\infty / \mu$, in accordance with equation \eqref{eq:evolution of m_p} (indeed, $h^\infty$ satisfies the boundary conditions \eqref{eq:BC popularity FP} for $\kappa=1$, which corresponds to having a finite mean). Concerning the variance of $h^\infty$, due to its behavior when $v \to +\infty$, it is finite if and only if $\zeta^2 < 2 \mu$, which implies that the global equilibrium \eqref{eq:global equilibria popularity FP} satisfies the boundary conditions \eqref{eq:BC popularity FP} for $\kappa=2$. Note that this is coherent with equation \eqref{eq:evolution of e_p}, which in turn tells us that in this case the energy of $h^\infty$ is given by $e_p^\infty = 2 \theta^2 {\mathcal{F}^\infty}^2/(\mu(2\mu-\zeta^2))$. The relevance of the distribution $h^\infty$ is supported by empirical evidence on the diffusion of products (in non-epidemiological contexts). For instance, in \cite{Red} the statistical distribution of the popularity of scientific articles, measured in terms of their citations, features a power law-type tail with Pareto exponent close to 3. Moreover, in \cite{SinRag} it is shown that the popularity of movies in the United States, measured in terms of their box office gross income, exhibits completely analogous statistical properties. Figure \ref{f3} shows different graphs of the distribution $h^\infty$ defined by \eqref{eq:global equilibria popularity FP}, for fixed values of the parameters $\mu$, $\zeta$, and $\theta$ (corresponding to a finite variance) and increasing values of the quantity $\mathcal{F}^\infty$, which corresponds to the amount of popularity pumped into the system by social network users sharing the products. In particular, from the asymptotic analysis conducted in Section \ref{connectivity} we infer that $\mathcal{F}^\infty$ is well defined and can be computed explicitly at least in the case when $\mathbf{f}$ evolves according to the simplified SIR kinetic model \eqref{eq:vectorial model simplified} and the transition rate $\beta_T$ is constant, since in this case we proved the rigorous convergence of $\mathbf{f}$ toward $\mathbf{f}^\infty$ (for which we provided the explicit expression \eqref{eq:f_inf}) and thus also that of $f = \sum_{J \in \mathcal{C}} f_J$ toward $f^\infty = \sum_{J \in \mathcal{C}} f_J^\infty$. The lowest values of $\mathcal{F}^\infty$ (and hence of the average popularity) correspond to the occurrence of opinion polarization only at $w=-1$, like shown in Figure \ref{f1}(B). In fact, in this case the agents tend to have opinions $w$ lower than some $\hat{w} \in (-1,1)$ and the function $p$ tends to be small, since we showed that high values of $p$ facilitate consensus formation (see Section \ref{polarization}). These two factors contribute to reduce both values of $\mathcal{F}[f](t)$ and $\mathcal{F}^\infty$. On the other hand, the hightest values of $\mathcal{F}^\infty$ do not necessarily correspond to the occurrence of opinion polarization at $w=1$. Indeed, even if in the latter situation there are many agents willing to share the products, their influence on other individuals could be very low since, once again, opinion polarization is facilitated by small values of the function $p$. Therefore, these two effects might compensate each other and prevent $\mathcal{F}^\infty$ to reach its supremum. Such considerations suggest that it is generally better to avoid the formation of extreme opinions. In particular, the popularity of products considering the disease as dangerous is favored by a population of agents that willing to adopt a protective behavior, but are also inclined to interact with each other.
    
    \begin{figure}[h!]
    \centering
        \includegraphics[width=0.5\textwidth]{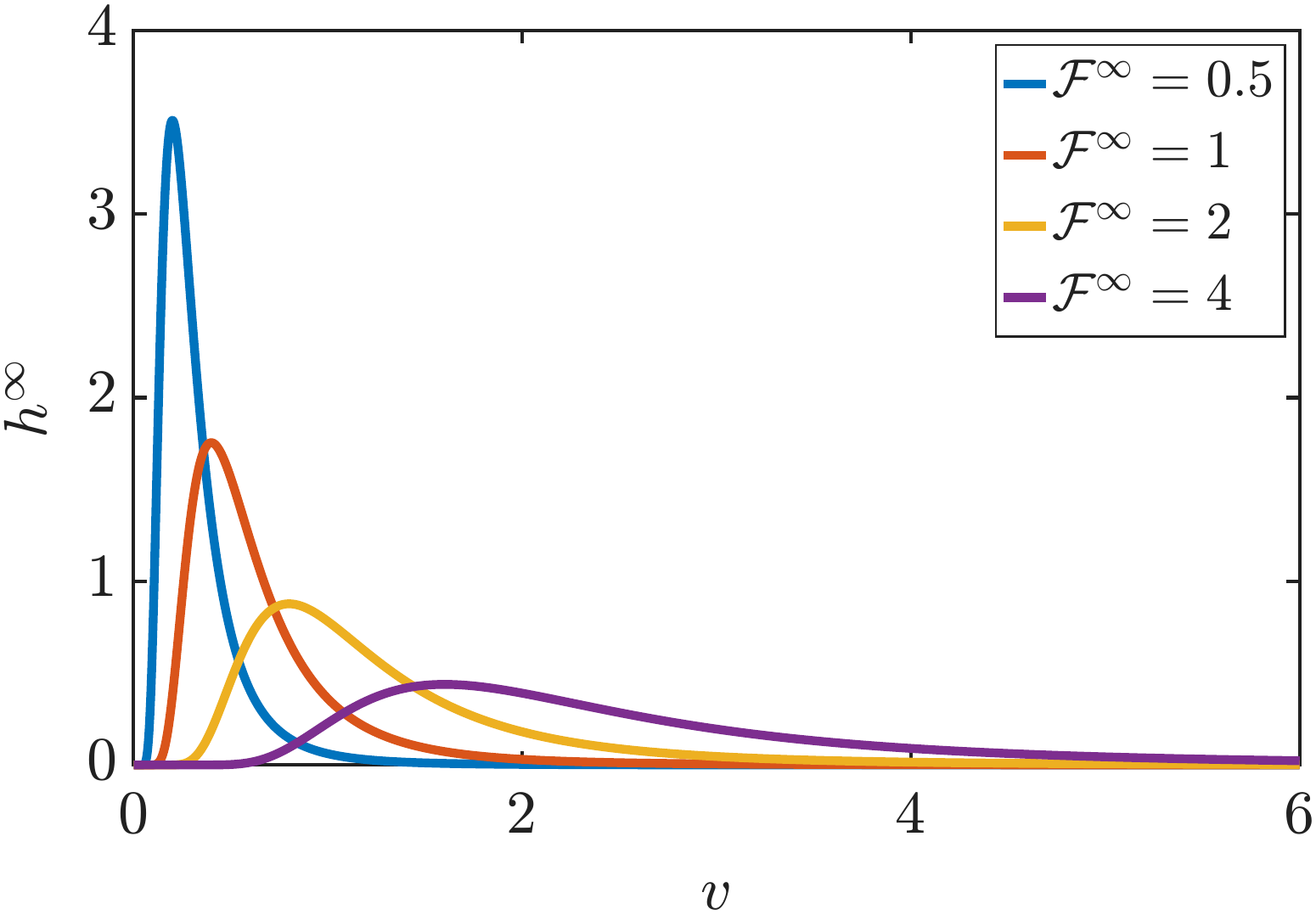}
    \caption{Inverse gamma distributions \eqref{eq:global equilibria popularity FP} for different values of $\mathcal{F}^\infty$. We have set $\mu=1.5$, $\zeta=1$, and $\theta=1$, corresponding to an equilibrium 
    $h^\infty$ with finite variance.} \label{f3}
    \end{figure}

    Finally, we remark that it is possible to replace the function $\mathbb{I}_{[\hat{w},1]}(w)$ in \eqref{eq:microscopic interactions popularity} with its opposite $\mathbb{I}_{[-1,\hat{w}]}(w)$. In this case, one could measure how not seriously the disease is perceived by the population and we could determine analogous results to the ones obtained in this section. 

\subsection{Analytical properties} 

    Following the strategies presented in Section \ref{section2}, it is easy to prove that the Fokker--Planck equation \eqref{eq:popularity FP} preserves the nonnegativity of the initial datum and, since it is mass preserving, also its $L^1(\R_+)$ regularity. Moreover, given the function $\mathcal{F}[f]$, it is possible to demonstrate the uniqueness of its solution. Concerning the $L^q(\R_+)$ regularity, $q \geq 2$, the following proposition holds. 

    \begin{proposition} [Regularity] \label{prop:regularity popularity}
        Let $q \in [2,+\infty)$ and let $h$ be a solution to the Fokker--Planck equation \eqref{eq:popularity FP}. Assume that $\partial_v^2 D_p^2 \in L^\infty(\R_+)$ and that the following additional no-flux boundary conditions
        \begin{align} \label{eq:additional BC for regularity popularity FP}
        \left\{
        \begin{aligned}
            & \left. D_p^2(v) \partial_v h(t,v) h^{q-1}(t,v) \right|_{v=0} = 0, \\[2mm] 
            & \left. h(t,v)  \right|_{v=0}=0,  \\[2mm]
            & \displaystyle \lim_{v \to +\infty}  D_p^2(v)  \partial_v h(t,v) h^{q-1}(t,v), \\[2mm]
            & \displaystyle \lim_{v \to +\infty}  h^q(t,v) (v + \partial_v D_p^2(v)) = 0,
        \end{aligned}
        \right.
        \end{align}
        hold for any $t \in \R_+$ and $v \in \R_+$. If $h^\init \in L^q(\R_+)$, then $h(t, \cdot) \in L^q(\R_+)$ for any $t \in \R_+$. More precisely, 
        \begin{equation*}
            \norm{h(t, \cdot)}_{L^q(\R_+)} \leq \norm{h^\init}_{L^q(\R_+)} \exp(\eta_q t), \quad t \in \R_+, 
        \end{equation*}
        where $\eta_q = \frac{q-1}{q}\left(\mu + \frac{\zeta^2}{2} \norm{\partial_v^2 D_p^2}_{L^\infty(\R_+)}\right)$. 
    \end{proposition}

    Proposition \ref{prop:regularity popularity} can be proved in the exact same way as Proposition \ref{prop:regularity}. Notice that the choice $D_p(v)=v$ satisfies the request $\partial_v^2 D_p^2 \in L^\infty(\R_+)$. Recall that, if $D_p(v)=v$ and $\mathcal{F}[f] \underset{t \to +\infty}{\longrightarrow} \mathcal{F}^\infty > 0$, the stationary solution $h^\infty$ to \eqref{eq:popularity FP} is given by \eqref{eq:global equilibria popularity FP}. Due to the presence of the exponential function, such equilibrium always satisfies the two no-flux boundary conditions \eqref{eq:additional BC for regularity popularity FP} at $v=0$. Moreover, a simple computation shows that, since $h^\infty (v) \sim v^{-2(1+\mu/\zeta^2)}$ when $v \to + \infty$, it also always satisfies the other two no-flux boundary conditions at $v \to + \infty$. Therefore, there are no restrictions on the parameters governing the popularity spread in order to derive these regularity properties for the global equilibrium.

    The constants $\eta_q$ are equi-bounded by $\eta_\infty=\mu + \frac{\zeta^2}{2} \norm{\partial_v^2 D_p^2}_{L^\infty(\R_+)}$, for any $q \in [2,+\infty)$. Therefore, in the limit $q \to +\infty$ we infer the regularity $h(t,\cdot) \in L^\infty(\R_+)$ for any $t \in \R_+$, specifically
    \begin{equation*}
        \norm{h(t, \cdot)}_{L^\infty(\R_+)} \leq \norm{h^\init}_{L^\infty(\R_+)} \exp(\eta_\infty t), \quad t \in \R_+,
    \end{equation*}
    provided that $h^\init \in L^\infty(\R_+)$. 

    We conclude with a comment about the existence of solutions for equation \eqref{eq:popularity FP} in the relevant case $D_p(v) = v$. In such specific setting, well-posedness of \eqref{eq:popularity FP} can be inferred within the framework of renormalized solutions studied by LeBris and Lions \cite{BriLio}. For this, we extend $h$ to $\R$ by defining $h(t,v) = 0$ for $v < 0$, and we observe that the evolution of \eqref{eq:popularity FP} preserves this property by maximum principle considerations (using similar arguments to those presented in Remark \ref{rem:positivity}, one shows that $v = 0$ acts as a barrier keeping $h(t,\cdot) \equiv 0$ when $v \leq 0$). Since the dynamics remains confined to $\R_+$, one can consider smooth enough extensions to $\R$ of the drift and diffusion coefficients in order to fall into the hypotheses of \cite[Proposition 2]{BriLio}. By defining the augmented drift term as $\mathbf{b}(t,v) = (\mu + \zeta^2) v - \theta \mathcal{F}[f](t)$ and the diffusion term as $\sigma(v) = \zeta v$, we can extend $\mathbf{b}$ from $v \in \R_+$ to $v \in \R$ by simply keeping its same expression, while we set $\sigma$ to be zero when $v < 0$. Then, $\mathbf{b}$ and $\sigma$ satisfy the regularity conditions (notice that $\mathcal{F}[f] \in L^1([0,T])$ for any $T>0$ thanks to the bound \eqref{eq:estimate of F}) required by \cite[Proposition 2]{BriLio}, hence it follows that given an initial datum $h^\init \in L^1(\R_+) \cap L^\infty(\R_+)$, for any time $T>0$ Eq. \eqref{eq:popularity FP} possesses a (unique) solution $h \in L^\infty ([0,T], L^1(\R_+) \cap L^\infty(\R_+))$ with $v\partial_v h \in L^2([0,T], L^2(\R_+))$. 

\subsection{Asymptotic analysis}

    Assuming that $D_p(v)=v$, we shall rigorously analyze the trend to equilibrium of the solution $h$ to \eqref{eq:popularity FP} in homogeneous Sobolev spaces $\dot{H}^{-s}(\R_+)$, with $s \in \big(\frac{1}{2}, 1\big)$. These are defined through the corresponding norm 
    \begin{equation}
        \norm{g}_{\dot{H}^{-s}(\R_+)} = \left( \int_{\R} |\xi|^{-2s} |\hat{g}(\xi)|^2 \dd \xi \right)^{\frac{1}{2}}, 
    \end{equation}
    where $\hat{g}: \R \to \mathbb{C}$ denotes the Fourier transform of some function $g: \R_+ \to \R_+$. Clearly, no probability distribution $g$ belongs to such space because $\hat{g}(0)=1$, $\hat{g}$ is continuous, and $\int_{(-\eps, \eps)} |\xi|^{-2s} \dd \xi = + \infty$ for any $s \in \big(\frac{1}{2}, 1\big)$ and any $\eps > 0$. For this reason, instead of investigating the convergence of $h$ toward $h^\infty$, we will study the convergence of the difference $h-h^\infty$ toward $0$. We start with the following useful result.

    \begin{lemma} \label{lemma:appB}
        Consider a random variable $X$ distributed as $g$, with a finite mean $\mathbb{E}(X)$. Then, the Fourier transform $\hat{g}$ of $g$ can be written as $\hat{g}(\xi) = 1 - i \xi \mathbb{E}(X) + R(\xi)$ for any $\xi \in \R$, where: 
        \begin{itemize}
            \item if $\mathbb{E}(X^2) < +\infty$, then $|R(\xi)| = \mathcal{O}(|\xi|^2)$; \\
            \item if $g(v) = \frac{\beta^\alpha}{\Gamma(\alpha)} v^{-1-\alpha} \exp(-\beta/v)$ for some constants $\alpha > 1$ and $\beta>0$, then $|R(\xi)| = \mathcal{O}(|\xi|^\alpha)$. 
        \end{itemize}
    \end{lemma}
    \begin{proof}
        Let $X$ be a random variable with probability density function $g$ over $\R_+$, then 
        \begin{equation}
            \hat{g}(\xi) = \mathbb{E}(\exp(-i\xi X)) = \int_{\R_+} g(v) \exp(-i \xi v) \dd v. 
        \end{equation}
        The following Taylor expansion with integral error term for $\exp(-i\xi X)$ holds: 
        \begin{equation}
            \exp(-i\xi X) = 1 - i \xi X - \xi^2 X^2 \int_0^1 (1-\varsigma) \exp(-i \xi X \varsigma) \dd \varsigma. 
        \end{equation}
        Therefore 
        \begin{equation}
            \hat{g}(\xi) = \mathbb{E}(\exp(-i\xi X)) = \mathbb{E}(1-i \xi X) + R(\xi) = 1 - i \xi \mathbb{E}(X) + R(\xi), 
        \end{equation}
        where 
        \begin{equation}
            R(\xi) = \mathbb{E}(\exp(-i\xi X) - 1 + i \xi X) = -\xi^2 \int_0^1 (1-\varsigma) \mathbb{E}(X^2 \exp(-i \xi X \varsigma)) \dd \varsigma
        \end{equation}
        denotes the remainder term. If $\mathbb{E}(X^2)<+\infty$ then 
        \begin{equation}
            |R(\xi)| \leq |\xi|^2 \int_0^1 (1-\varsigma) \mathbb{E}(X^2) \dd \varsigma = \frac{\mathbb{E}(X^2)}{2} |\xi|^2. 
        \end{equation}
        On the other hand, if $g(v) = \frac{\beta^\alpha}{\Gamma(\alpha)} v^{-1-\alpha} \exp(-\beta/v)$, with $\alpha > 1$ and $\beta>0$, then for $\xi > 0$
        \begin{equation}
        \begin{split}
            |R(\xi)| & \ \ \leq \int_0^{+ \infty} |\exp(-i\xi v) - 1 + i \xi v| \frac{\beta^\alpha}{\Gamma(\alpha)} v^{-1-\alpha} \dd v \\[2mm] 
            & \overset{\varsigma=\xi v}{=} \xi^\alpha \int_0^{+\infty} |\exp(-i \varsigma) - 1 + i \varsigma| \frac{\beta^\alpha}{\Gamma(\alpha)} \varsigma^{-1-\alpha} \dd \varsigma \\[4mm] 
            & \ \ \leq C(\alpha, \beta) |\xi|^{\alpha}, 
        \end{split}
        \end{equation} 
        where $C(\alpha,\beta)>0$ because $\exp(-i \varsigma) - 1 + i \varsigma = -\frac{\varsigma^2}{2} + \smallO(\varsigma^3)$ for $\varsigma \to 0$. The calculations for $\xi<0$ are similar and the thesis follows. 
    \end{proof}

    \begin{remark}
        The estimates of Lemma \ref{lemma:appB} are interesting in the limit $|\xi| \to 0$. Indeed, we trivially have         \begin{equation}
        \begin{split}
            |R(\xi)| \leq \mathbb{E}(2+|\xi| X) = 2 + |\xi| \mathbb{E}(X), 
        \end{split}
        \end{equation}
        which is stronger than the previous estimates for sufficiently large $|\xi|$. 
    \end{remark}

    Assume that the initial distribution $h^\init(v) = h(0,v)$ has finite second moment. Lemma \ref{lemma:appB} implies that, in the limit $|\xi| \to 0$, 
    \begin{equation} \label{eq:initial datum}
        \hat{h}(0,\xi) - \hat{h}^\eq (0,\xi) = \mathcal{O}(|\xi|).  
    \end{equation}
    Note that this equality holds even if the local equilibrium does not have a finite energy. Therefore, as $|\xi| \to 0$, 
    \begin{equation}
        |\xi|^{-2s} |\hat{h}(0,\xi) - \hat{h}^\infty (0,\xi)|^2 \leq C |\xi|^{-2s} |\xi|^{2}
    \end{equation}
    for some $C>0$, hence $\norm{h(0,\xi) - h^\infty(0,\xi)}_{\dot{H}^{-s}(\R_+)} < + \infty$ for all $s \in \big(\frac{1}{2}, 1\big)$. Obviously, the same is true if one of the two distributions is replaced with the local equilibrium $h^\eq$ defined by \eqref{eq:local equilibria popularity FP}. We are now ready to prove a convergence result in the spaces $\dot{H}^{-s}(\R_+)$, $s \in \big(\frac{1}{2}, 1\big)$. 

    \begin{theorem} \label{teo_appB}
        Let $D_p(v)=v$ and $h$ be a solution to the Fokker--Planck equation \eqref{eq:popularity FP}, starting from a nonnegative initial datum $h^\init$ with finite second moment. Assume that $\mathcal{F}[f](t) > 0$ for any $t \in \R_+$, $\mathcal{F}[f] \underset{t \to +\infty}{\longrightarrow} \mathcal{F}^\infty>0$ and $\frac{\dd}{\dd t} \mathcal{F}[f] \underset{t \to +\infty}{\longrightarrow} 0$. Then, for any $s \in \big(\frac{1}{2}, 1\big)$, the difference $h - h^\infty$ converges to $0$ in the homogeneous Sobolev space $\dot{H}^{-s}(\R_+)$. 
    \end{theorem}
    \begin{proof}
        Since 
        \begin{equation*}
            \norm{h - h^\infty}_{\dot{H}^{-s}(\R_+)} \leq \norm{h-h^\eq}_{\dot{H}^{-s}(\R_+)} + \norm{h^\eq - h^\infty}_{\dot{H}^{-s}(\R_+)}, 
        \end{equation*}
        we divide the proof into two parts, similarly to what we have done in Section \ref{sec:trends}. Firstly, we will demonstrate that $\norm{h-h^\eq}_{\dot{H}^{-s}(\R_+)} \underset{t \to +\infty}{\longrightarrow} 0$, and then we will prove that $\norm{h^\eq - h^\infty}_{\dot{H}^{-s}(\R_+)} \underset{t \to +\infty}{\longrightarrow} 0$.  

        \medskip   

        \noindent \textbf{Step 1 -- Convergence to the local equilibrium ${h}^\eq$.} We follow the strategy adopted in \cite{MarTosZan}. Since the right hand side of the Fokker--Planck equation \eqref{eq:popularity FP} vanishes when evaluated at the local equilibrium $h^\eq$, we have 
        \begin{equation} \label{eq:evolution of h-h_eq}
            \tau_p \partial_t (h - h^\eq) = - \tau_p \partial_t h^\eq + \partial_v \left( \left(\mu v - \theta \mathcal{F}[f] \right)(h -h^\eq\right) + \frac{\zeta^2}{2} \partial^2_v \left(v^2 (h-h^\eq) \right). 
        \end{equation}
        The Fourier transformed version of \eqref{eq:evolution of h-h_eq} is \cite{TorTos2}
        \begin{equation*}
        \begin{split}
            \tau_p \partial_t (\hat{h} - \hat{h}^\eq) = \frac{\zeta^2}{2} \xi^2 \partial_{\xi}^2 (\hat{h} - \hat{h}^\eq) - \mu \xi \partial_{\xi} (\hat{h} - \hat{h}^\eq) - i \theta \mathcal{F}[f] \xi (\hat{h} - \hat{h}^\eq) - \tau_p \partial_t  \hat{h}^\eq, 
        \end{split}
        \end{equation*}
        Let $\hat{h}(t,\xi) - \hat{h}^\eq(t,\xi) = a(t,\xi) + i b(t,\xi)$ and $\hat{h}^\eq(t,\xi) = \tilde{a}(t,\xi) + i \tilde{b}(t,\xi)$, then 
        \begin{equation*}
        \begin{split}
            \tau_p \partial_t a =& \frac{\zeta^2}{2} \xi^2 \partial_{\xi}^2 a - \mu \xi \partial_{\xi} a + \theta \mathcal{F}[f] \xi b - \tau_p \partial_t  \tilde{a}, \\ 
            \tau_p \partial_t b =& \frac{\zeta^2}{2} \xi^2 \partial_{\xi}^2 b - \mu \xi \partial_{\xi} b - \theta \mathcal{F}[f] \xi a - \tau_p \partial_t  \tilde{b}.
        \end{split}
        \end{equation*}
        If we multiply the first equation by $2a(t,\xi)$ and the second one by $2b(t,\xi)$, summing up we get 
        \begin{equation} \label{eq:proof4}
        \begin{split}
            \tau_p \partial_t \left| \hat{h}(t,\xi) - \hat{h}^\eq(t,\xi) \right|^2 &= \zeta^2 \xi^2 (a \partial^2_\xi a + b \partial^2_\xi b) - \mu \xi \partial_{\xi} \left| \hat{h} - \hat{h}^\eq \right|^2 -2 a \tau_p \partial_t  \tilde{a} -2 b \tau_p \partial_t  \tilde{b} \\[4mm] 
            & \leq \zeta^2 \xi^2 (a \partial^2_\xi a + b \partial^2_\xi b) - \mu \xi \partial_{\xi} \left| \hat{h} - \hat{h}^\eq \right|^2 + 2 \tau_p \left| \hat{h} - \hat{h}^\eq \right| \left|\partial_t \hat{h}^\eq\right|. 
        \end{split}
        \end{equation}
        Multiplying by $|\xi|^{-2s}$, $\frac{1}{2} < s <1$, and integrating over $\R$ with respect to $\xi$, we obtain the evolution equation for $\|h-h^\eq\|_{\dot{H}^{-s}(\R_+)}^2$, which reads 
        \begin{equation} \label{eq:initial estimate of h-h_eq}
        \begin{split}
            \tau_p \frac{\dd}{\dd t} \int_\R |\xi|^{-2s} |\hat{h}(t,\xi) - \hat{h}^\eq(t,\xi)|^2 \dd \xi & \leq \underbrace{\zeta^2 \int_\R |\xi|^{2-2s} \left(a \partial^2_\xi a + b \partial^2_\xi b \right) \dd \xi - \mu \int_\R \xi |\xi|^{-2s} \partial_{\xi} \left| \hat{h} - \hat{h}^\eq \right|^2 \dd \xi}_{\mathcal{T}_1} \\[2mm] 
            & \hspace{5cm} +\underbrace{2\tau_p \int_\R |\xi|^{-2s} \left| \hat{h} - \hat{h}^\eq \right| \left|\partial_t \hat{h}^\eq\right| \dd \xi}_{\mathcal{T}_2}. 
        \end{split}
        \end{equation}
        Let us start from $\mathcal{T}_1$. Suitable repeated integrations by parts (and the behavior of the Fourier transform for large $|\xi|$) imply that \cite{MarTosZan, TorTos1, TorTos2}
        \begin{equation} \label{eq:T_1}
            \mathcal{T}_1 = (1-2s) (\zeta^2(1-s)+\mu) \int_\R |\xi|^{-2s} \left| \hat{h} - \hat{h}^\eq \right|^2 \dd \xi - \zeta^2 \int_\R |\xi|^{2-2s} \left(\left|\partial_\xi a\right|^2 + \left|\partial_\xi b\right|^2\right) \dd \xi. 
        \end{equation}
        Then, for any $R>0$, 
        \begin{equation}
            0 \leq \int_\R |\xi|^{-2s} \big(\xi \partial_\xi a(t,\xi) - R a(t,\xi)\big)^2 \dd \xi = \int_\R |\xi|^{2-2s} |\partial_\xi a|^2 \dd \xi + \left(R^2 + R(1-2s)\right) \int_\R |\xi|^{-2s} a^2 \dd \xi, 
        \end{equation}
        where we have integrated by parts to deal with the mixed term of the square. In particular, we find
        \begin{equation}
            \int_\R |\xi|^{2-2s} |\partial_\xi a(t,\xi)|^2 \dd \xi \geq (R(2s-1) - R^2) \int_\R |\xi|^{-2s} a^2 \dd \xi,
        \end{equation}
        and an analogous inequality holds for $b(t,\xi)$. By choosing $R$ in the optimal way, namely $R=\frac{2s-1}{2}$, we thus obtain 
        \begin{equation} \label{eq:estimate of a and b}
            \int_\R |\xi|^{2-2s} \left(\left|\partial_\xi a(t,\xi) \right|^2 + \left|\partial_\xi b(t,\xi) \right|^2\right) \dd \xi \geq \frac{(2s-1)^2}{4} \int_\R |\xi|^{-2s} \left| \hat{h} - \hat{h}^\eq \right|^2 \dd \xi. 
        \end{equation}
        Combining \eqref{eq:T_1} and \eqref{eq:estimate of a and b}, we infer that $\mathcal{T}_1$ is controlled by  
        \begin{equation} \label{eq:estimate of T_1}
            \mathcal{T}_1 \leq - \underbrace{(2s-1) \left(\zeta^2 \frac{3-2s}{4} + \mu\right)}_{C_0>0} \int_\R |\xi|^{-2s} \left| \hat{h} - \hat{h}^\eq \right|^2 \dd \xi . 
        \end{equation} 
        
        Concerning $\mathcal{T}_2$, by the definition of the Fourier transform and the properties of the inverse Gamma distributions \cite[Section 5(a)]{MarTosZan} (see also \cite{GraRyz}) we have 
        \begin{equation} \label{eq:estimate of h_eq}
            |\partial_t \hat{h}^\eq(t,\xi)| = \left| \int_{\R_+} \partial_t h^\eq(t,v) \exp(-i \xi v) \dd v\right| \leq \int_{\R_+} |\partial_t h^\eq(t,v)| \dd v \leq C_1 \left|\frac{\dd}{\dd t} \mathcal{F}[f](t)\right|, 
        \end{equation}
        for some suitable constant $C_1>0$. Since both $h$ and $h^\eq$ are probability densities, then $\hat{h}(t,0)=\hat{h}^\eq(t,0)=1$ and there exists a constant $C_2 > 0$ such that the following Taylor expansion with Lagrange remainder can be bounded as 
        \begin{equation} \label{eq:estimate with Taylor}
            |\hat{h}(t,\xi) - \hat{h}^\eq(t,\xi)| = |\partial_\xi \hat{h}(t,\xi_1) - \partial_\xi \hat{h}^\eq(t,\xi_2)| |\xi| \leq (m_p(t) + m_p^\eq(t)) |\xi| \leq C_2 |\xi|, 
        \end{equation}
        where $\xi_1, \xi_2 \in (\min\{0, \xi\}, \max\{0, \xi\})$ and we have defined $\displaystyle m_p^\eq(t) = \int_{\R_+} v h^\eq(t,v) \dd v$. Indeed, 
        \begin{equation}
            |\partial_\xi \hat{h}(t,\xi)| = \left| -i\int_{\R_+} v h(t,v) \exp(-i \xi v) \dd v \right| \leq \int_{\R_+} v h(t,v) \dd v = m_p(t),
        \end{equation}
        so that \eqref{eq:evolution of m_p} and \eqref{eq:estimate of F} hold (clearly, a similar inequality holds also for $\partial_\xi \hat{h}^\eq(t,\xi)$). Combining \eqref{eq:estimate of h_eq} and \eqref{eq:estimate with Taylor} we then get 
        \begin{equation*}
            |\xi|^{-2s} \left| \hat{h} - \hat{h}^\eq \right| \left|\partial_t \hat{h}^\eq\right| \leq C_1 C_2  \left|\frac{\dd}{\dd t} \mathcal{F}[f]\right| |\xi|^{1-2s}, 
        \end{equation*}
        which in turn implies that for any $R>0$ 
        \begin{equation*}
            \int_{|\xi|\leq R} |\xi|^{-2s} \left| \hat{h}(t,\xi) - \hat{h}^\eq(t,\xi) \right| \left|\partial_t \hat{h}^\eq(t,\xi) \right| \dd \xi \leq C_1 C_2  \left|\frac{\dd}{\dd t} \mathcal{F}[f]\right| \frac{R^{2-2s}}{1-q}, 
        \end{equation*}
        hence Cauchy--Schwartz inequality allows us to infer that 
        \begin{equation*}
        \begin{split}
            \int_{|\xi|> R} |\xi|^{-2s} \left| \hat{h}(t,\xi) - \hat{h}^\eq(t,\xi) \right| \left|\partial_t \hat{h}^\eq(t,\xi) \right| \dd \xi &\leq \sqrt{\int_{|\xi|> R} |\xi|^{-2s} \left| \hat{h} - \hat{h}^\eq \right|^2 \dd \xi} \; \sqrt{\int_{|\xi|> R} |\xi|^{-2s} \left|\partial_t \hat{h}^\eq \right|^2  \dd \xi} \\[2mm]
            &\leq C_1 \left|\frac{\dd}{\dd t} \mathcal{F}[f]\right| \sqrt{\int_\R |\xi|^{-2s} \left| \hat{h} - \hat{h}^\eq \right|^2 \dd \xi} \; \sqrt{\int_{|\xi|> R} |\xi|^{-2s} \dd \xi} \\[2mm] 
            &= C_1 \left|\frac{\dd}{\dd t} \mathcal{F}[f]\right| \frac{1}{R^{q-\frac{1}{2}}} \sqrt{\frac{2}{2s-1} \int_\R |\xi|^{-2s} \left| \hat{h} - \hat{h}^\eq \right|^2 \dd \xi}. 
        \end{split}
        \end{equation*}
        Combining these two inequalities and choosing $R$ in the optimal way, the term $\mathcal{T}_2$ can be finally bounded as 
        \begin{equation} \label{eq:estimate of T_2}
            \mathcal{T}_2 \leq \tau_p C_3 \left|\frac{\dd}{\dd t} \mathcal{F}[f]\right| \left(\int_\R |\xi|^{-2s} \left| \hat{h} - \hat{h}^\eq \right|^2 \dd \xi\right)^{\frac{2-2s}{3-2s}}, 
        \end{equation}
        where $C_3=C_3(s, C_1, C_2)>0$. At last, combining \eqref{eq:initial estimate of h-h_eq}, \eqref{eq:estimate of T_1}, and \eqref{eq:estimate of T_2} we get 
        \begin{equation}
        \begin{split}
            \frac{\dd}{\dd t} \int_\R |\xi|^{-2s} \left| \hat{h}(t,\xi) - \hat{h}^\eq(t,\xi) \right|^2 \dd \xi \leq -\frac{C_0}{\tau_p} \int_\R |\xi|^{-2s} \left| \hat{h} - \hat{h}^\eq \right|^2 \dd \xi + C_3 \left|\frac{\dd}{\dd t} \mathcal{F}[f]\right| \left(\int_\R |\xi|^{-2s} \left| \hat{h} - \hat{h}^\eq \right|^2 \dd \xi\right)^{\frac{2-2s}{3-2s}}.
        \end{split}
        \end{equation}
        Since $\frac{\dd}{\dd t} \mathcal{F}[f] \underset{t \to +\infty}{\longrightarrow} 0$, it is easy to demonstrate that $\displaystyle \int_\R |\xi|^{-2s} \left| \hat{h}(t,\xi) - \hat{h}^\eq(t,\xi) \right|^2 \dd \xi \underset{t \to +\infty}{\longrightarrow} 0$, implying the convergence of $h$ toward $h^\eq$ with respect to the $\dot{H}^{-s}(\R_+)$ norm, for $s \in \big(\frac{1}{2},1\big)$. 

        \medskip    

        \noindent \textbf{Step 2 -- Convergence to the global equilibrium $h^\infty$.} For this step, notice that the dominate convergence theorem implies that  $\norm{h^\eq - h^\infty}_{L^1(\R_+)} \underset{t \to +\infty}{\longrightarrow} 0$. Indeed, for a sufficiently large $t$, 
        \begin{align}
            h^\eq(t,v) \leq g(v) =
            \begin{cases}
                2 \frac{\left(2 \theta \mathcal{F}^\infty\right)^{1 + \frac{2 \mu}{\zeta^2}}}{\zeta^{2 - \frac{4 \mu}{\zeta^2}} \Gamma\left(1 + \frac{2 \mu}{\zeta^2}\right)} v^{-2 \left(1+\frac{\mu}{\zeta^2}\right)} \exp \left(-\frac{\theta \mathcal{F}^\infty}{\zeta^2}\frac{1}{v}\right),   \quad &  0 < v < 1, \\
                \\
                2 \frac{\left(2 \theta \mathcal{F}^\infty\right)^{1 + \frac{2 \mu}{\zeta^2}}}{\zeta^{2 - \frac{4 \mu}{\zeta^2}} \Gamma\left(1 + \frac{2 \mu}{\zeta^2}\right)} v^{-2 \left(1+\frac{\mu}{\zeta^2}\right)}, & \quad v \geq 1, 
            \end{cases}
        \end{align} 
        with $g \in L^1(\R_+)$, and obviously $h^\eq(t,v) \underset{t \to +\infty}{\longrightarrow} h^\infty(v)$ pointwise in $v \in \R_+$. 

        For a fixed $0 < \eps \ll 1$, we have that
        \begin{equation}
            \int_{\R_+} |\xi|^{-2s} \left| \hat{h}^\eq(t,\xi) - \hat{h}^\infty(\xi) \right|^2 \dd \xi = \underbrace{\int_{|\xi|>\eps} |\xi|^{-2s} \left| \hat{h}^\eq - \hat{h}^\infty \right|^2 \dd \xi }_{\mathcal{T}_3} + \underbrace{\int_{|\xi| \leq \eps} |\xi|^{-2s} \left| \hat{h}^\eq - \hat{h}^\infty \right|^2 \dd \xi}_{\mathcal{T}_4}, 
        \end{equation}
        and we can study the two integrals separately. Let us start with $\mathcal{T}_3$. For all $\xi \in \R$ we have 
        \begin{equation}
            \left| \hat{h}^\eq(t,\xi) - \hat{h}^\infty(\xi) \right| = \left| \int_{\R_+} \big( h^\eq(t,v) - h^\infty(v) \big) \exp(-i \xi v) \dd v\right| \leq \norm{h^\eq - h^\infty}_{L^1(\R_+)}. 
        \end{equation}
        Therefore 
        \begin{equation}
            \sup_{\xi \in \R} \left| \hat{h}^\eq(t,\xi) - \hat{h}^\infty(\xi) \right|^2 \leq \norm{h^\eq - h^\infty}_{L^1(\R_+)}^2, 
        \end{equation}
        which implies that 
        \begin{equation}
            \mathcal{T}_3 \leq \norm{h^\eq - h^\infty}_{L^1(\R_+)}^2 \int_{|\xi|> \eps} |\xi|^{-2s} \dd \xi = \frac{2}{2s-1} \norm{h^\eq - h^\infty}_{L^1(\R_+)}^2 \eps^{1-2s}. 
        \end{equation}
        Consider now the integral $\mathcal{T}_4$. Substituting $\hat{h}$ with $\hat{h}^\textnormal{eq}$ in \eqref{eq:initial datum} leads directly to  
        \begin{equation}
            \mathcal{T}_4 \leq \frac{2C}{3-2s} \eps^{3-2s},
        \end{equation}
        for some $C>0$ which does not depend on $t$ since the first moment is bounded (recall \eqref{eq:evolution of m_p} and \eqref{eq:estimate of F}) and, in the case of an infinite second moment, with the notation of Lemma \ref{lemma:appB}, $\alpha$ is constant and $\beta$ is bounded because $\mathcal{F}[f]  \underset{t \to +\infty}{\longrightarrow} \mathcal{F}^\infty$ (hence the estimate for $|R(\xi)|$ is uniform in time). 

        In conclusion, 
        \begin{equation}
            \int_{\R_+} |\xi|^{-2s} \left| \hat{h}^\eq(t,\xi) - \hat{h}^\infty(\xi) \right|^2 \dd \xi \leq \frac{2}{2s-1} \norm{h^\eq - h^\infty}_{L^1(\R_+)}^2 \eps^{1-2s} + \frac{2C}{3-2s} \eps^{3-2s}. 
        \end{equation}
        Setting, for example, $\eps=\eps(t)=\norm{h^\eq - h^\infty}_{L^1(\R_+)}$ we obtain that $\norm{h^\eq - h^\infty}_{\dot{H}^{-s}(\R_+)} \underset{t \to +\infty}{\longrightarrow} 0$. 
    \end{proof} 

    \begin{remark}
        The calculations of Theorem \ref{teo_appB} could be repeated to prove the convergence of $h$ to equilibrium, even when initially $h^\init$ does not have a finite energy. However, in this case one would also have to repeat the calculations of Lemma \ref{lemma:appB} to demonstrate that $h(0,v)-h^\eq(0,v) \in \dot{H}^{-s}(\R_+)$, which held true in the proof of Theorem \ref{teo_appB} thanks precisely to the hypothesis that the energy of the initial distribution was finite. 
    \end{remark}

    \begin{remark} \label{rmk:L1_conv}
        Assuming the exponential relaxation of $\mathcal{F}[f]$ toward $\mathcal{F}^\infty > 0$, it is possible to recover the convergence of $h$ to the global equilibrium $h^\infty$ in $L^1(\R_+)$. Indeed, it would be enough to adapt the computations from the proof of \cite[Theorem 2]{BonBor}. The only additional difficulty would arise from the fact that here our (local and global) equilibria are given by inverse gamma distributions, while in the cited paper the equilibria were described by beta distributions. These lead to the appearance in the estimates of integral terms involving the $\log$ moment of the solution to the Fokker--Planck equation, which could be controlled by applying H\"older's inequality and by exploiting the regularity of the solution itself. In the present setting, the inverse gamma distributions would instead produce more singular moments of the form $\displaystyle \int_{\R_+} \frac{1}{v} h(t,v) \dd v$, and the regularity of $h$ could not be used to bound them since obviously $\frac{1}{v}$ does not belong to any $L^q(\R_+)$, $q \geq 1$. The idea in this case would thus be to reduce the singularity via an integration by parts, recovering integrals of the form $\displaystyle \int_{\R_+} \log v \partial_v h(t,v) \dd v$ that one could estimate once again thanks to H\"older's inequality, provided that $\partial_v h$ is regular enough. Such regularity would follow from an adaptation of Proposition \ref{prop:regularity popularity} to study the evolution of the norms $\|\partial_v h \|_{L^q(\R_+)}$, $q \geq 2$, under suitable additional boundary conditions on the solution $h$.
    \end{remark}

    We spend some final words on the particular case of $\mathcal{F}^\infty = 0$. Since \eqref{eq:evolution of m_p} implies that $m_p \underset{t \to +\infty}{\longrightarrow} 0$, then $h(t,v)$ converges (in the sense of distributions) toward the Dirac delta $\delta(v)$ because the support of $h$ is always contained in $\R_+$ (even if $D_p(v) \ne v$). Indeed, for any $\eps > 0$ 
    \begin{equation}
        \int_\eps^{+\infty} h(t,v) \dd v \leq \frac{1}{\eps} \int_\eps^{+\infty} v h(t,v) \dd v \leq \frac{m_p}{\eps} \underset{t \to +\infty}{\longrightarrow} 0. 
    \end{equation}
    This result is not surprising. Indeed, $\mathcal{F}^\infty = 0$ is equivalent to $\hat{w}=1$, which implies that only a negligible part of the population will contribute to the spread of the products, hence their popularity can only decrease. Interestingly, the variance of $h$ converges to zero (i.e., to the variance of $\delta(v)$) if and only if $\zeta^2 < 2 \mu$ (see \eqref{eq:evolution of e_p}). Indeed, if $\zeta^2 \geq 2 \mu$ the diffusion process is very strong and a relevant fraction of the mass of $h$ remains scattered over $\R_+$ for any time $t \in \R_+$. This represents a nice example of the fact that the convergence in the sense of distributions does not imply the convergence of the moments. 

\section{Numerical simulations} \label{sec:num}

    \noindent We conclude the analysis of our model with several numerical simulations to illustrate the complete behaviors of system \eqref{eq:FP-SIR and popularity}. We start by presenting the numerical algorithm used in the simulations in Sections \ref{code} and \ref{code_popularity}, while we will carry out a series of tests in the following subsections. We will mainly describe the structure-preserving numerical scheme for the kinetic SIR-type system \eqref{eq:vectorial model} since it represents a generalization of the numerical scheme employed in the simulations of the Fokker--Planck model \eqref{eq:popularity FP}. However, we will explain how the two schemes interact with each other. Concerning the numerical tests, we will start by analyzing the trends to equilibrium of the simplified model \eqref{eq:vectorial model simplified}, before moving to the analysis of epidemic dynamics considering the original model \eqref{eq:vectorial model}. Throughout this last section we will always assume $D(w)=\sqrt{1-w^2}$ and $D_p(v)=v$, for which we derived explicit expressions for the equilibria of the models.     

\subsection{Numerical algorithm for the kinetic SIR system} \label{code}  

The approximation of the kinetic SIR system \eqref{eq:vectorial model} is obtained through a splitting strategy to deal separately with the Fokker--Planck opinion formation process and with the epidemiological operators. Specifically, the opinion step is advanced using a semi-implicit, structure-preserving scheme \cite{ParZan}, whereas the epidemiological step is integrated using a classical fourth-order Runge--Kutta method. 

In the following, we will focus on the Chang--Cooper-type scheme \cite{ChaCoo} adopted for the Fokker--Planck step. This method is structure-preserving, meaning that it preserves at the numerical level the structural properties of the model, by conserving mass, maintaining the solution nonnegative, and capturing its correct long-time behavior.

\medskip 

\noindent \textbf{Domain discretization.} Given $N_x \in \mathbb{N}^*$ points in the graphon domain $\Omega = [0,1]$, the (uniform) spatial grid is given by
\begin{equation*}
    x_i = i \Delta x, \qquad i \in \{0, \ldots, N_x-1\}, \qquad \Delta x = \frac{1}{N_x-1},
\end{equation*}
with $x_{N_x-1} = 1$, while given $N_w \in \mathbb{N}^*$ points in the opinion domain $\mathcal{I} = [-1,1]$, the (uniform) opinion grid reads
\begin{equation*}
    w_j = -1 + j\Delta w, \qquad j \in \{0, \ldots, N_w-1\}, \qquad \Delta w = \frac{2}{N_w-1},
\end{equation*}
with $w_{N_w-1} = 1$. Finally, the time variable is discretized as 
\begin{equation*}
    t^n = n \Delta t, \qquad n \in \{0, \ldots, N_t\}, 
\end{equation*}
where $\Delta t>0$ is an adaptive time step chosen to satisfy the stability requirement \eqref{eq:CFL kinetic SIR} of the scheme (which will be presented later on) and we denote with $T > 0$ the final simulation time, which may vary between the different tests.

At each time $t^n$, the discrete approximation of each distribution function $f_J$, $J \in \mathcal{C}$, is represented by the values $(f_{J})_{i,j}^n \simeq f_J(t^n, x_i, w_j)$, forming a two-dimensional array for each compartment. The indices $i$ and $j$ run over all spatial and opinion grid points, respectively. 

\medskip 
\noindent \textbf{Chang--Cooper reconstruction of the $w$-flux.} We formally rewrite the Fokker--Planck operator $\mathbf{Q}$ defined by \eqref{eq:opinion operator Q} in gradient flow form as
\begin{equation*}
    Q_J(\mathbf{f}, \mathbf{f})(t, x, w) = \partial_w \mathcal{J}_J(t, x, w), \quad t \in \R_+,\; x \in \Omega,\; w \in \mathcal{I},
\end{equation*}
for any $J \in \mathcal{C}$, where the $J$-th microscopic flux in the opinion variable is given by 
\begin{align*}
    \mathcal{J}_J(t, x, w) &= \left( \lambda \sum_{J' \in \mathcal{C}} \mathcal{K}[f_{J'}](t, x, w)  - \sigma_J^2 w \sum_{J'\in \mathcal{C}} \mathcal{H}[f_{J'}](t, x)\right) f_J(t, x, w) \\
    & \hspace{3.5cm} +\frac{\sigma_J^2}{2} (1-w^2) \sum_{J'\in \mathcal{C}} \mathcal{H}[f_{J'}](t, x) \partial_w f_J(t, x, w), \quad t \in \R_+,\; x \in \Omega,\; w \in \mathcal{I}.
\end{align*}
This structure allows for a Chang--Cooper-type discretization in $w$. For this, let us define for any $t \in \R_+$, $x \in \Omega$, and $w \in \mathcal{I}$, the drift and diffusion terms
\begin{align*}
    \mathcal{T}_J^\drift(t, x, w) &= \lambda \sum_{J' \in \mathcal{C}} \mathcal{K}[f_{J'}](t, x, w) - \sigma_J^2w \sum_{J'\in \mathcal{C}} \mathcal{H}[f_{J'}](t, x), \\[2mm]
    \mathcal{T}_J^\diff(t, x, w) &= \frac{\sigma_J^2}{2} (1-w^2) \sum_{J' \in \mathcal{C}} \mathcal{H}[f_{J'}](t, x).
\end{align*}
At each interface $w_{j+\frac{1}{2}} = w_j + \frac{\Delta w}{2}$, the diffusion
coefficient is evaluated pointwise,
\begin{equation*}
    (\mathcal{T}_J^\diff)_{i, j+\frac{1}{2}}^n
    = \mathcal{T}_J^\diff\big(t^n, x_i, w_{j+\frac{1}{2}}\big),
\end{equation*}
while the interface drift coefficient is tied to the
cell integral of the ratio between drift and diffusion, namely
\begin{equation*}
    (\mathcal{T}_J^\drift)_{i, j+\frac{1}{2}}^n
    = \frac{(\mathcal{T}_J^\diff)_{i, j+\frac{1}{2}}^n}{\Delta w}\,
      \lambda_{i,j+\frac{1}{2}}^n,
    \qquad
    \lambda_{i,j+\frac{1}{2}}^n
    = \int_{w_j}^{w_{j+1}}
      \frac{\mathcal{T}_J^\drift(t^n, x_i, w)}{\mathcal{T}_J^\diff(t^n, x_i, w)}
      \, \dd w.
\end{equation*}
Then, at each fixed position $x_i$ on the graphon and time $t^n$, the Chang--Cooper
numerical flux in the $w$-direction is written as
\begin{equation*}
    (\mathcal{J}_J)_{i,j+\frac{1}{2}}^n
    = (\mathcal{T}_J^\drift)_{i, j+\frac{1}{2}}^n
      (\bar{f}_{J})_{i,j+\frac{1}{2}}^n
    + (\mathcal{T}_J^\diff)_{i, j+\frac{1}{2}}^n
      \frac{(f_{J})_{i,j+1}^n - (f_{J})_{i,j}^n}{\Delta w},
\end{equation*}
where the central weighted average $(\bar{f}_{J})_{i,j+\frac{1}{2}}^n$ is
defined by
\begin{equation*}
  (\bar{f}_{J})_{i,j+\frac{1}{2}}^n
  = \big(1-\delta_{i,j+\frac{1}{2}}^n\big)\, (f_{J})_{i,j+1}^n
  + \delta_{i,j+\frac{1}{2}}^n\, (f_{J})_{i,j}^n,
\end{equation*}
with weights
\begin{equation*}
    \delta_{i,j+\frac{1}{2}}^n
    = \frac{1}{\lambda_{i,j+\frac{1}{2}}^n}
    + \frac{1}{1 - \exp\big(\lambda_{i,j+\frac{1}{2}}^n\big)} \,\in (0,1).
\end{equation*}
The cell integrals $\lambda_{i,j+\frac{1}{2}}^n$ are approximated by an eighth--order accurate Gauss--Legendre quadrature
formula. However, in the two cells adjacent to
$w = \pm 1$, where the diffusion vanishes and the integrand is singular, the
integral is evaluated with the second--order accurate midpoint rule instead. Mass conservation is ensured by the no-flux boundary conditions
imposed at the boundaries of the opinion domain \cite{ParZan}.

\medskip 
\noindent \textbf{Discretization in time.} The drift--diffusion operator described above is integrated with a semi-implicit backward Euler step. For each fixed position $x_i$ on the graphon, the update in the opinion variable reads
\begin{equation*}
  \frac{(f_{J})_{i,j}^{n+1} - (f_{J})_{i,j}^n}{\Delta t} - \frac{1}{\Delta w} \left( (\mathcal{J}_J)_{i,j+\frac{1}{2}}^{n+1} - (\mathcal{J}_J)_{i,j-\frac{1}{2}}^{n+1} \right) = 0. 
\end{equation*}
This semi-implicit discretization leads, at each time step and for each position $x_i$, to a tridiagonal linear system in the $w$-index for $(f_{J})_{i,j}^{n+1}$, which can be written in matrix form as
\begin{equation*}
A[(\mathbf{f}_{J})_i^n] (\mathbf{f}_{J})_i^{n+1} = (\mathbf{f}_{J})_i^n,
\end{equation*}
where $(\mathbf{f}_{J})_i^n = ((f_{J})_{i,0}^n, \ldots, (f_{J})_{i,N_w-1}^n)$ is the vector of unknowns at time $t^n$ and $A[(\mathbf{f}_{J})_i^n]$ is the tridiagonal matrix whose entries depend on the drift--diffusion terms and on the Chang--Cooper weights introduced above.

A key property of this semi-implicit scheme is that, under a suitable restriction on the time step, nonnegativity of the solution is preserved, namely from $(\mathbf{f}_J)_i^n \geq 0$ it follows that $(\mathbf{f}_{J})_i^{n+1} \geq 0$. More precisely, nonnegativity is ensured as soon as the CFL (Courant--Friedrichs--Lewy) condition 
\begin{equation} \label{eq:CFL kinetic SIR}
    \Delta t < \frac{\Delta w}{2 C^n} \qquad \textnormal{with} \qquad C^n = \max_{\underset{j \in \{0, \dots, N_w-2\}}{\underset{i \in \{0, \dots, N_x-1\}}{J \in \mathcal{C}} }}   \left| (\mathcal{T}_{J}^\drift)_{i, j+\frac{1}{2}}^n \right|
\end{equation}
is satisfied \cite{ParZan}.

\medskip 
\noindent \textbf{Operator splitting for the full kinetic SIR system.} To efficiently advance system \eqref{eq:vectorial model} in time, we employ an operator splitting technique during which, for any fixed position $x_i$ on the graphon, opinion and epidemic dynamics are solved separately over each time interval $[t^n, t^{n+1}]$. We adopt a simple first-order splitting scheme composed of the following two steps.

Firstly, in the opinion formation step, for each compartment $J \in \mathcal{C}$ we solve
\begin{equation*}
    \partial_t f_J^* = \frac{1}{\tau} Q_J(\mathbf{f}^*, \mathbf{f}^*)
\end{equation*}
using the semi-implicit structure-preserving Chang--Cooper scheme described beforehand. Note that here $f_J^*$ is a two-dimensional array whose components correspond to a pair $(x_i, w_j)$. Similarly $\mathbf{f}^*$ denotes a three-dimensional array incorporating all values of the two-dimensional arrays $f_J^*$, $J \in \mathcal{C}$. The result of this step is then
\begin{equation*}
    f_J^{n,*}= \mathcal{Q}_J^{\Delta t}(\mathbf{f}^n, \mathbf{f}^n), 
\end{equation*}
for any $J \in \mathcal{C}$. Secondly, in the epidemiological step, we consider the updated opinion distributions $\mathbf{f}^*$ as initial datum to solve the SIR-type system 
\begin{align*}
    \left\{
    \begin{aligned}
    \partial_t f_S^{**} &= - f_S^{**} \int_{\Omega \times \mathcal{I}} \beta_T(w,w_*) f_I^{**}(y,w_*) \dd y \dd w_*, \\[2mm]
    \partial_t f_I^{**} &= f_S^{**} \int_{\Omega \times \mathcal{I}} \beta_T(w,w_*) f_I^{**}(y,w_*) \dd y \dd w_* - \gamma f_I^{**}, \\[4mm]
    \partial_t f_R^{**} &= \gamma f_I^{**},
    \end{aligned}
    \right.
\end{align*}
using a classical fourth-order Runge--Kutta method. The outcome of this step is
\begin{equation}
    f_J^{n,**} = \mathcal{E}_J^{\Delta t}(\mathbf{f}^{n,*}, \mathbf{f}^{n,*}),
\end{equation}
for any $J \in \mathcal{C}$. In conclusion, the overall update of the unknowns over one time step is given by 
\begin{equation}
    f_J^{n+1} = \mathcal{E}_J^{\Delta t} \left( \mathcal{Q}_J^{\Delta t}(\mathbf{f}^n, \mathbf{f}^n), \mathcal{Q}_J^{\Delta t}(\mathbf{f}^n, \mathbf{f}^n) \right),
\end{equation}
for any $J \in \mathcal{C}$. In all our simulations, the opinion domain $\mathcal{I}$ is discretized using a uniform grid of $N_w = 101$ points, while $\Omega$ is discretized with $N_x = 21$ evenly spaced nodes.

\subsection{Numerical algorithm for the popularity model} \label{code_popularity} 

As the popularity of a product is quantified by the variable $v \in \R_+$ and since the steady state of the kinetic equation \eqref{eq:popularity FP} is represented by an inverse gamma distribution with fat tails, in the numerical approximation the domain of $v$ must be truncated in an appropriate manner. In more detail, we discretize the popularity interval $[0,L]$ using a uniform grid
\begin{equation*}
    v_k = k \Delta v, \qquad k \in \{0,\dots,N_v-1\}, \qquad \Delta v = \frac{L}{N_v-1},
\end{equation*}
where the right boundary $L$ and the number of nodes $N_v$ are selected adaptively by balancing the retention of the inverse gamma tail-mass and a sufficient resolution of its peak, namely we impose a tail criterion based on the incomplete gamma function, choosing $L$ so that the residual density beyond $L$ is below a prescribed tolerance $\varepsilon_{\mathrm{tail}}$. In addition, we prescribe a target mesh size $\Delta v_{\mathrm{target}}$ to resolve the peak of the equilibria with a fixed number of grid points. The final pair $(N_v,L)$ is obtained by enforcing
\begin{equation*}
    \Delta v \leq \Delta v_{\mathrm{target}}, \qquad N_v \in [N_{\min},N_{\max}], \qquad L \geq L_{\min},
\end{equation*}
where $N_{\min}$ and $N_{\max}$ are user-prescribed bounds on the admissible grid size and $L_{\min}$ is a minimum admissible truncation ensuring that the peak region is included.

The time variable is then discretized as 
\begin{equation} 
    t^m = m \Delta t_p, \qquad m \in \{0, \ldots, N_{t_p}\}, 
\end{equation}
where $\Delta t_p>0$ is an adaptive time step chosen to satisfy the CFL condition \eqref{eq:CFL popularity} detailed later on (which ensures the stability of the scheme for the popularity) and in general $\Delta t_p \neq \Delta t$. We highlight however that the final simulation time $T$ does instead coincide with that used to run the structure-preserving scheme for the kinetic SIR system. This is natural, since the evolution of the popularity depends on how the nonlocal quantity $\mathcal{F}[f](t)$ evolves through $f = \sum_{J \in \mathcal{C}} f_J$.

At each time $t^m$, the discrete approximation of the distribution function $h$ is given by the values $h_k^m \simeq h(t^m, v_k)$, forming now a one-dimensional array with the index $k$ running over all popularity grid points. Proceeding as before, we formally rewrite the Fokker--Planck operator $Q_p$ defined by \eqref{eq:popularity FP} in gradient flow form
\begin{equation*}
    Q_p(h, f)(t, v) = \partial_v \mathcal{J}_p(t, v), \quad t \in \R_+,\; v \in [0,L],
\end{equation*}
where the flux is given by
\begin{equation*}
    \mathcal{J}_p(t, v) = \Big( (\mu + \zeta^2) v - \theta \mathcal{F}[f](t) \Big) h(t,v) + \frac{\zeta^2}{2} v^2 \partial_v h(t,v), \quad t \in \R_+,\; v \in [0,L].
\end{equation*}
Let us then define for any $t \in \R_+$ and $v \in [0,L]$ the drift and diffusion terms
\begin{align*}
    \mathcal{T}_p^\drift(t, v) = (\mu + \zeta^2) v - \theta \mathcal{F}[f](t), \qquad \mathcal{T}_p^\diff(t, v) = \frac{\zeta^2}{2} v^2.
\end{align*}
At each interface $v_{k+\frac{1}{2}} = v_k + \frac{\Delta v}{2}$, these terms are evaluated as 
\begin{equation*}
    \left(\mathcal{T}_p^\diff\right)_{k+\frac{1}{2}}^m
    = \mathcal{T}_p^\diff\big(t^m, v_{k+\frac{1}{2}}\big),
    \qquad
    \left(\mathcal{T}_p^\drift\right)_{k+\frac{1}{2}}^m
    = \frac{\left(\mathcal{T}_p^\diff\right)_{k+\frac{1}{2}}^m}{\Delta v}\,
      \lambda_{k+\frac{1}{2}}^m,
    \qquad
    \lambda_{k+\frac{1}{2}}^m
    = \int_{v_k}^{v_{k+1}}
      \frac{\mathcal{T}_p^\drift(t^m, v)}{\mathcal{T}_p^\diff(t^m, v)}\, \dd v,
\end{equation*}
and the Chang--Cooper flux formulation in $v$ reads
\begin{equation*}
    (\mathcal{J}_p)_{k+\frac{1}{2}}^m
    = (\mathcal{T}_p^\drift)_{k+\frac{1}{2}}^m
      \bar{h}_{k+\frac{1}{2}}^m
    + (\mathcal{T}_p^\diff)_{k+\frac{1}{2}}^m
      \frac{h_{k+1}^m - h_k^m}{\Delta v},
\end{equation*}
where the central weighted average $\bar{h}_{k+\frac{1}{2}}^m$ is defined by
\begin{equation*}
  \bar{h}_{k+\frac{1}{2}}^m
  = \big(1-\delta_{k+\frac{1}{2}}^m\big)\, h_{k+1}^m
  + \delta_{k+\frac{1}{2}}^m\, h_k^m,
\end{equation*}
with weights
\begin{equation*}
    \delta_{k+\frac{1}{2}}^m
    = \frac{1}{\lambda_{k+\frac{1}{2}}^m}
    + \frac{1}{1 - \exp\big(\lambda_{k+\frac{1}{2}}^m \big)} \,\in (0,1).
\end{equation*}
We approximate all the cell integrals $\lambda_{k+\frac{1}{2}}^m$ with an eighth--order accurate Gauss--Legendre quadrature formula.

At last, the drift--diffusion operator described above is integrated with a semi-implicit backward Euler step, providing the following temporal update in the popularity variable: 
\begin{equation*}
  \frac{h_k^{m+1} - h_k^m}{\Delta t_p} - \frac{1}{\Delta v} \left( (\mathcal{J}_p)_{k+\frac{1}{2}}^{m+1} - (\mathcal{J}_p)_{k-\frac{1}{2}}^{m+1} \right) = 0. 
\end{equation*}
In particular, time integration of system \eqref{eq:FP-SIR and popularity} involves two distinct CFL constraints. The SIR Fokker--Planck solver is run on a time grid determined by the time step $\Delta t$ computed from the CFL stability condition \eqref{eq:CFL kinetic SIR}, which provides values of the functional $\mathcal{F}[f](t)$ only at the corresponding output instants. The popularity equation is then advanced with an independent time step $\Delta t_p$ dictated by the stability condition (guaranteeing nonnegativity of the solution) of the structure-preserving scheme in the variable $v$, which reads
\begin{equation} \label{eq:CFL popularity}
    \Delta t_p < \frac{\Delta v}{2 C^m} \qquad \textnormal{with} \qquad C^m = \max_{k \in \{0, \dots, N_v-2\}}   \left| \left(\mathcal{T}_p^\drift\right)_{k+\frac{1}{2}}^m \right|.
\end{equation}
Since typically the popularity time step $\Delta t_p$ is much smaller than the kinetic SIR time step $\Delta t$, we point out that the quantity $\mathcal{F}[f](t)$ must be evaluated at intermediate times to determine the evolution of the popularity. We therefore reconstruct $\mathcal{F}[f](t)$ by piecewise-constant interpolation on the kinetic SIR output grid, namely we take $\mathcal{F}[f](t)=\mathcal{F}[f](t^n)$ for $t \in [t^n,t^{n+1})$.

\subsection{Test 1 -- Trends to equilibrium for the SIR system} 

We begin our numerical investigations by illustrating the trends to equilibrium proved in Theorem \ref{teo:SIR} for the simplified model \eqref{eq:vectorial model simplified}. For this, we fix $\alpha=0$ in the function \eqref{eq:function beta_T}, $G \equiv 1$, $\sigma_J^2 = \sigma^2>0$ for all $J \in \mathcal{C}$, and we consider initial distributions that are independent of the position on the graphon.

We focus our attention on two different graphons $\mathcal{B}$. The first \eqref{eq:fat-tailed graphon} has been introduced in Section \ref{connectivity}, with a propensity function $p$ given by \eqref{eq:propensity}, while the second \eqref{eq:graphon_app} is described in Appendix \ref{appA}, with propensity function $p$ given by \eqref{eq:propensity_app}. In particular, we make use of a cutoff $10^{-10}$ to deal with the singularities of the graphon \eqref{eq:fat-tailed graphon} when computing $p(x)$, namely we consider the transformation $(x,y) \mapsto (x+10^{-10}, y+10^{-10})$ (see the published code \cite{MATLAB} for more details).

Figure \ref{IC5 P2 POLARIZZ} presents the results of the analysis related to the first graphon. Note that the choice of the parameters $\chi$ and $\xi$ corresponds to the absence of opinion leaders (recall Sections \ref{propensity} and \ref{sec:derivation_simplified}). The distributions $f_S$, $f_I$, and $f_R$ are initialized as
\begin{gather*}
    f_S^\init(x,w)=\rho_S^\init \left(\frac{2}{3}\,\mathbb{I}_{\Omega \times [-1,0]}(x,w)+\frac{1}{3}\mathbb{I}_{\Omega \times [0,1]}(x,w)\right), \\[4mm] f_I^\init(x,w)=\frac{\rho_I^\init}{2}\mathbb{I}_{\Omega \times \mathcal{I}}(x,w), \qquad f_R^\init(x,w)=\frac{\rho_R^\init}{2} \mathbb{I}_{\Omega \times \mathcal{I}}(x,w), 
\end{gather*}
with $\rho_I^\init = \rho_R^\init = 10^{-3}$ and $\rho_S^\init = 1-\rho_I^\init-\rho_R^\init$. Their evolution over time is plotted for different fixed values of $x \in \Omega$. Even if the initial distributions are independent of the position on the graphon, its influence leads to opinion polarization for certain $x$ and to consensus formation for others. We highlight that, up to our knowledge, even if the numerical results do not perfectly match the theoretical equilibria, this level of precision in reproducing opinion polarization has never been achieved before with these numerical techniques (indeed, previous studies reported only graphs showing consensus formation). Obviously, in this case it is not possible to perfectly capture the analytical solutions because they diverge at the boundaries $w = \pm 1$. As a consequence, the numerical errors at the boundaries probably propagate in the interior of $\mathcal{I}$. Notice that $\rho_S$ is decreasing, $\rho_R$ is increasing, and $\rho_I$ initially increases and then decreases, in line with the dynamics of the classical SIR system \eqref{eq:SIR} with $\mathcal{R}_0 > 1$. Finally, as predicted by Theorem \ref{teo:SIR}, note that the $L^1(\Omega \times \mathcal{I})$ norms decrease as $t \to + \infty$, reaching values of order $10^{-2}$ or lower (the points $w=\pm 1$ were excluded when computing these norms).

Figure \ref{IC4 P6 CONSENSO} regards the second graphon. The distributions $f_S$, $f_I$, and $f_R$ are initialized as
\begin{gather*}
    f_S^\init(x,w) = 3 \rho_S^\init \mathbb{I}_{\Omega \times \left[\frac{1}{3},\frac{2}{3}\right]}(x,w), \\[4mm]
    f_I^\init(x,w)=\frac{\rho_I^\init}{2}\mathbb{I}_{\Omega \times \mathcal{I}}(x,w), \qquad f_R^\init(x,w)=\frac{\rho_R^\init}{2} \mathbb{I}_{\Omega \times \mathcal{I}}(x,w),
\end{gather*}
with $\rho_I^\init = \rho_R^\init = 10^{-3}$ and $\rho_S^\init = 1-\rho_I^\init-\rho_R^\init$.
We plot again their evolution over time for different fixed values of $x \in \Omega$, but we now observe consensus formation among all individuals. Like before, the results align with the dynamics of the classical SIR system \eqref{eq:SIR} and with Theorem \ref{teo:SIR}. It is clear that the $L^1(\Omega \times \mathcal{I})$ distance between the solutions to their corresponding equilibria decreases exponentially fast, an effect that was not evident in the previous test due to the numerical errors (recall that we were only able to prove a convergence of the order $\smallO(1/\sqrt{t})$ for $t \to + \infty$). The higher precision obtained in this case follows from the absence of opinion polarization, which prevents numerical errors at the boundaries $w = \pm 1$. In particular, the quantities $\| f_J - f_J^\infty \|_{L^1(\Omega \times \mathcal{I})}$ quickly reach values of order $10^{-6}$ and would continue to decrease as $t$ increases.

\begin{figure}[h!] 
    \centering
    
    \begin{subfigure}[t]{0.49\textwidth}
        \centering
        \includegraphics[width=\linewidth]{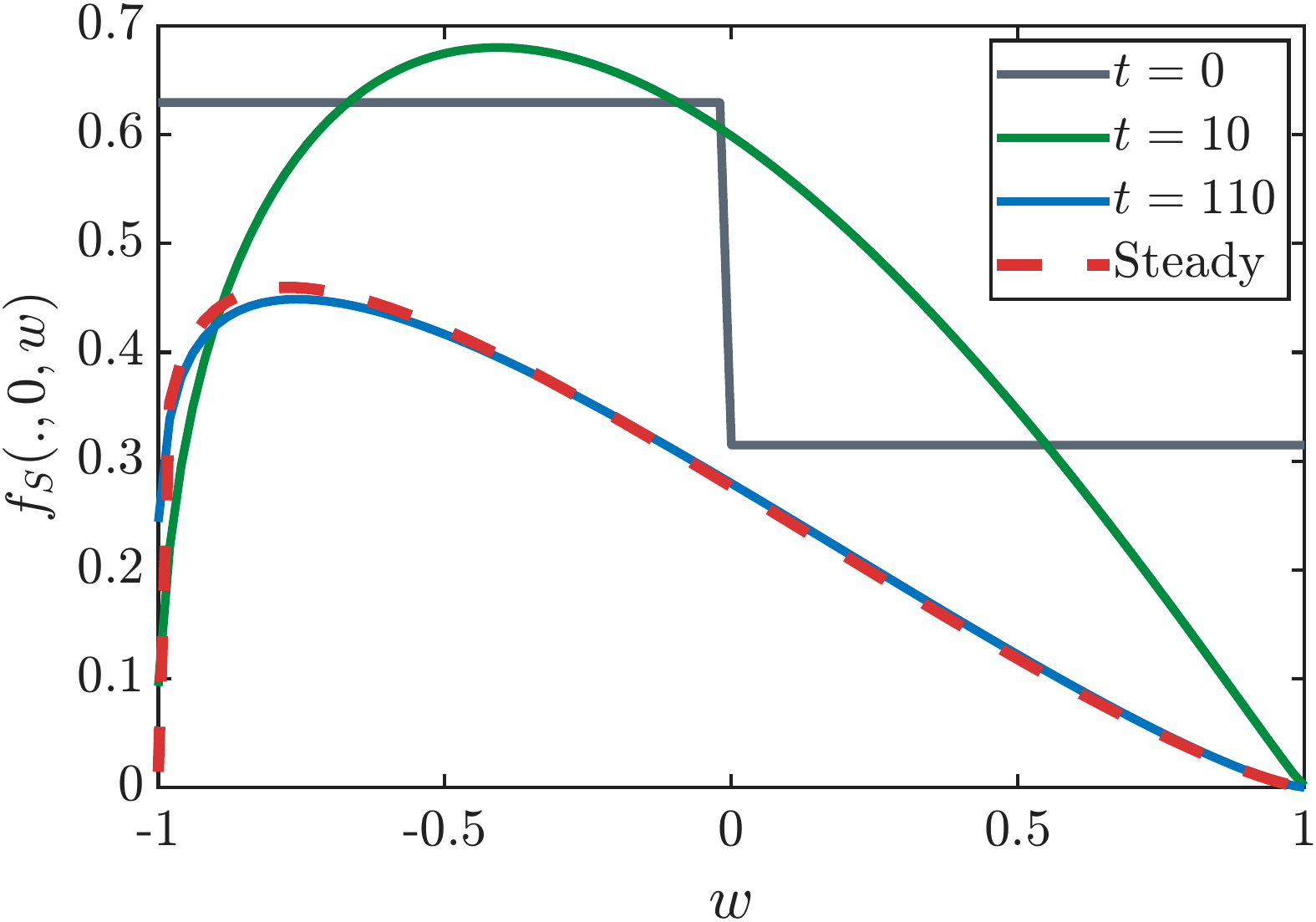}
        \caption*{(A)}
    \end{subfigure}\hfill
    \begin{subfigure}[t]{0.49\textwidth}
        \centering
        \includegraphics[width=\linewidth]{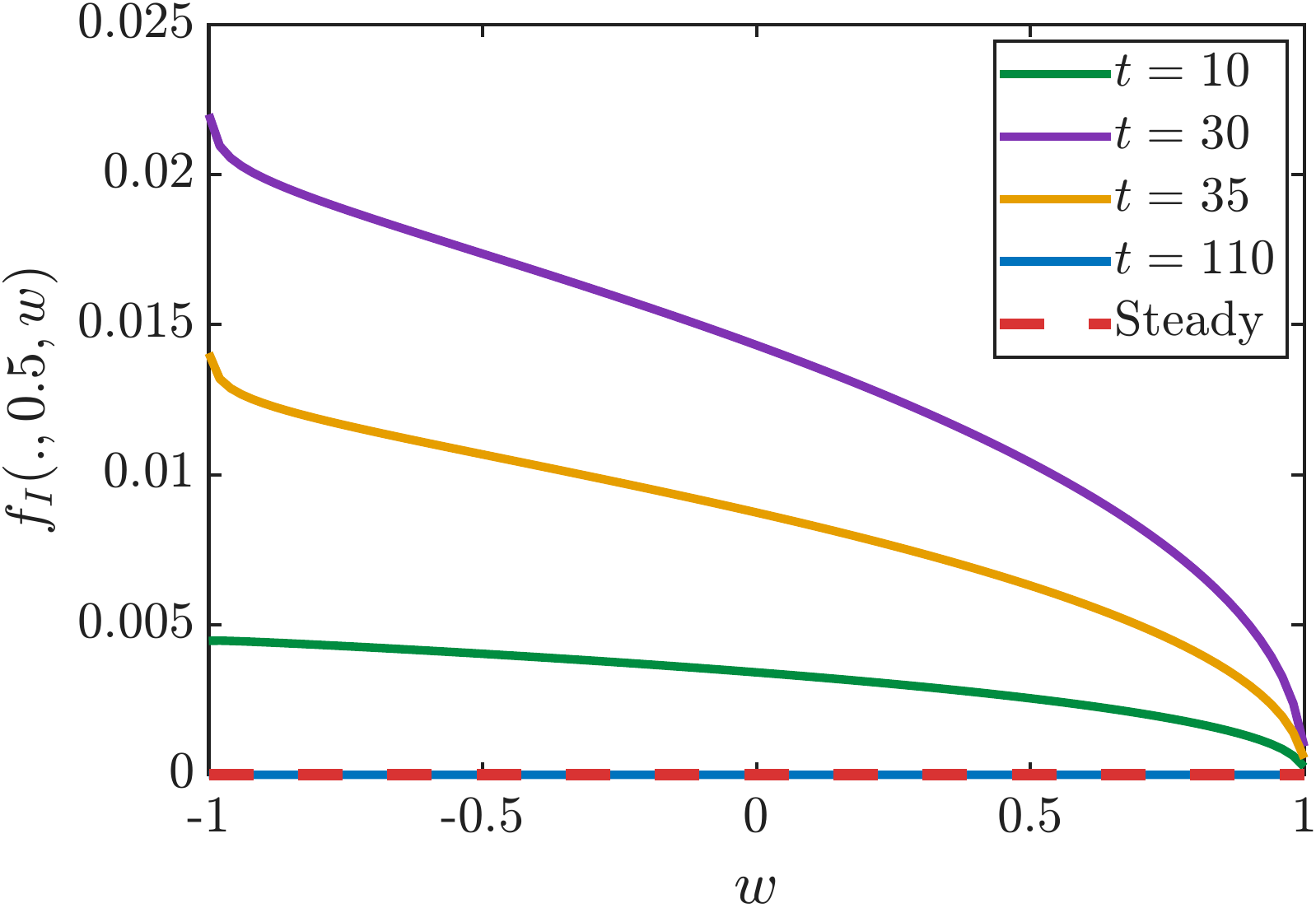}
        \caption*{(B)}
    \end{subfigure}

    \vspace{0.8em}

    \begin{subfigure}[t]{0.49\textwidth}
        \centering
        \includegraphics[width=\linewidth]{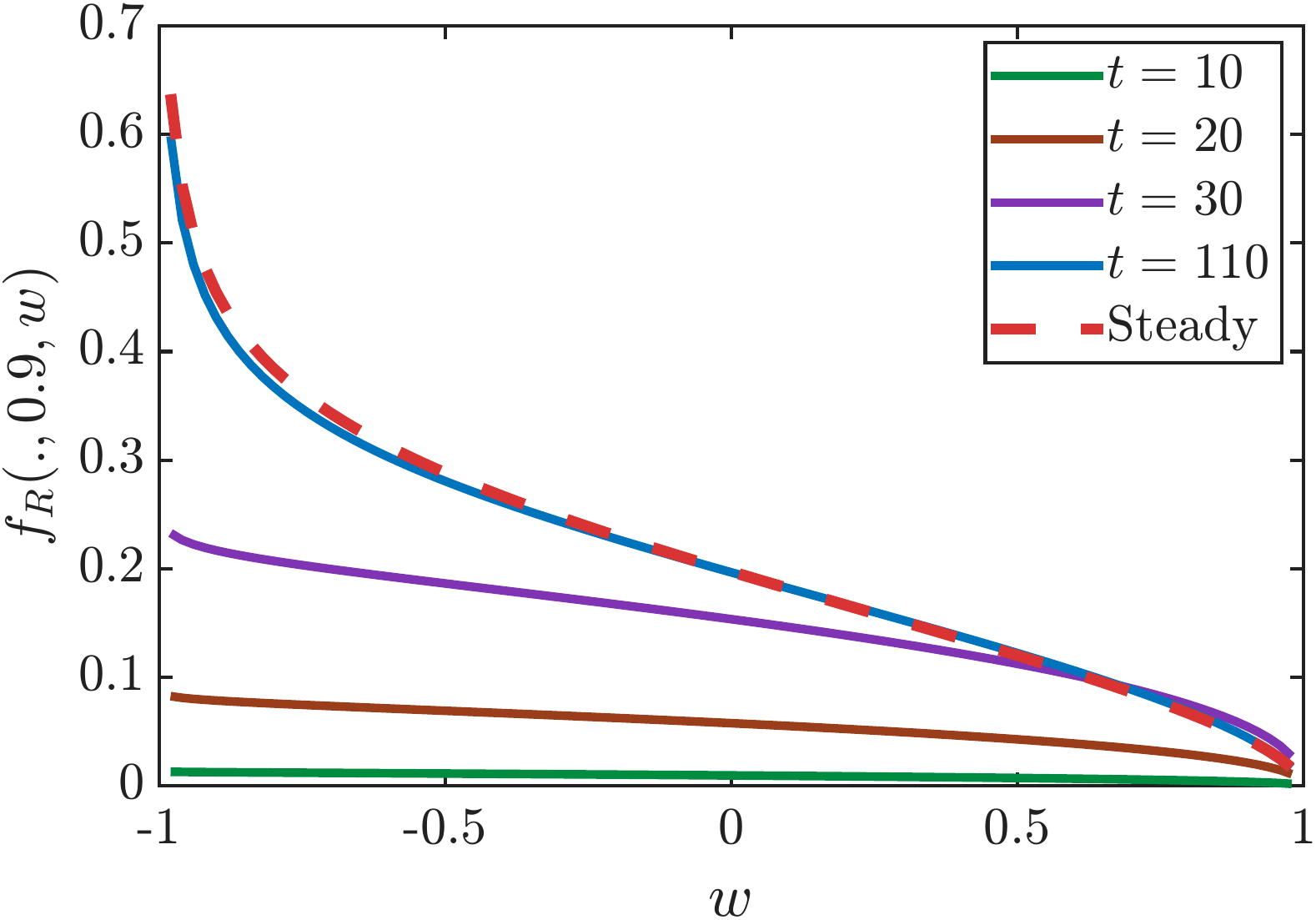}
        \caption*{(C)}
    \end{subfigure}\hfill
    \begin{subfigure}[t]{0.49\textwidth}
        \centering
        \includegraphics[width=\linewidth]{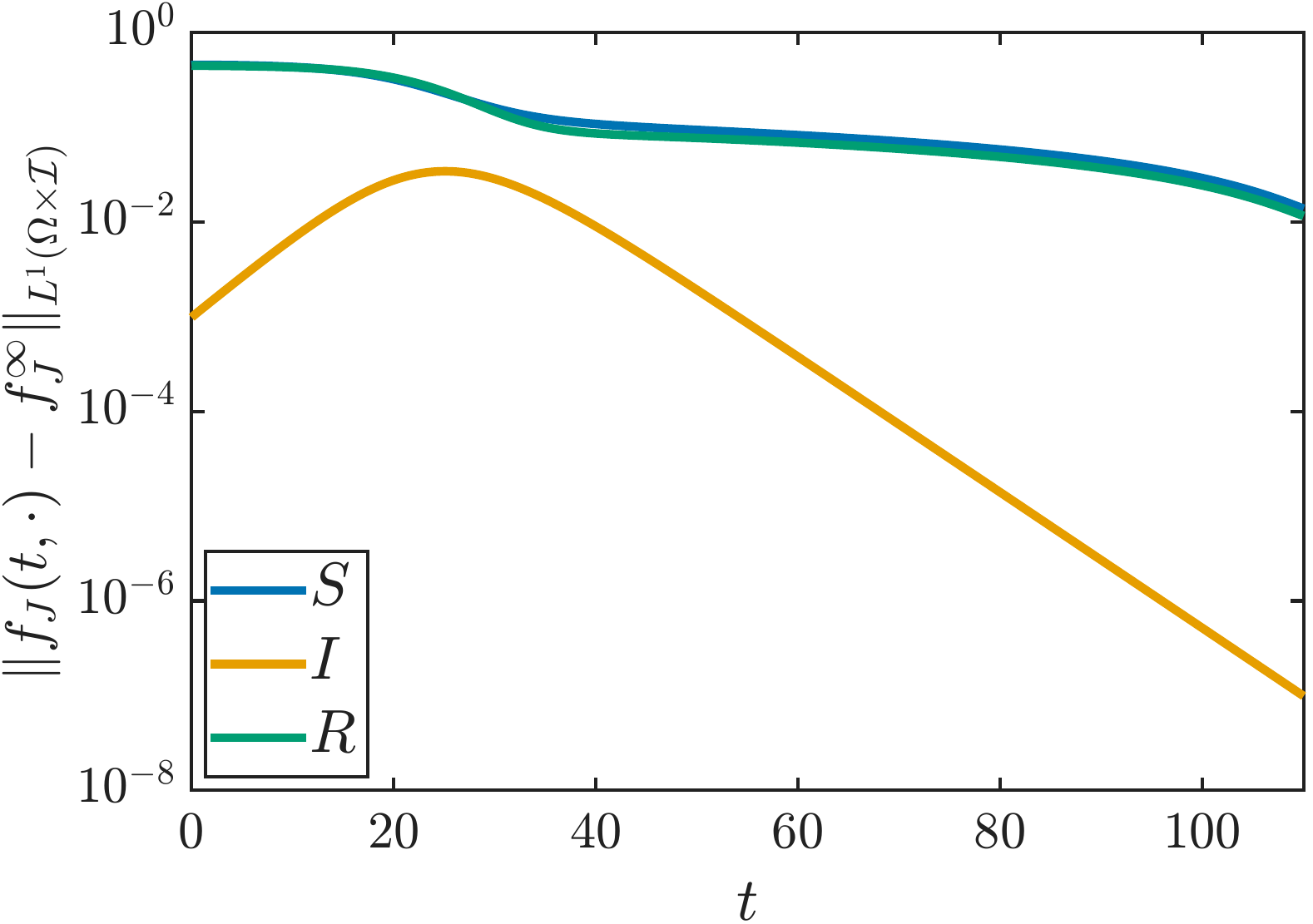}
        \caption*{(D)}
    \end{subfigure}

  \caption{\textbf{Test 1.} Evolution of the simplified model \eqref{eq:vectorial model simplified} with epidemiological transition rate $\beta_T$ defined by \eqref{eq:function beta_T} and graphon $\mathcal{B}$ given by \eqref{eq:fat-tailed graphon}. Opinion distributions of susceptible (A), infected (B), and removed (C) individuals at different time instants and different fixed positions $x \in \Omega$ on the graphon. Note that we only plot $f_S^\init$, since $f_I^\init$ and $f_R^\init$ are uniform distributions with a very small mass. The dashed curves represent the global equilibria \eqref{eq:f_inf}. The last plot (D) shows the time-evolution of the $L^1(\Omega \times \mathcal{I})$ distance between each distribution $f_J$ and its corresponding equilibrium $f_J^\infty$. Values of the parameters: $\beta=0.8$, $\alpha = 0$ \eqref{eq:function beta_T}, $\gamma=0.6$, $\lambda=1$, $\sigma^2 = 0.16$, $\tau=1$, $\xi=0.05$ \eqref{eq:fat-tailed graphon}, $\chi=0.5$ \eqref{eq:P}, and $a=1$ \eqref{eq:g_P}. }
\label{IC5 P2 POLARIZZ}
\end{figure}

\begin{figure}[h!] 
    \centering
    
    \begin{subfigure}[t]{0.49\textwidth}
        \centering
        \includegraphics[width=\linewidth]{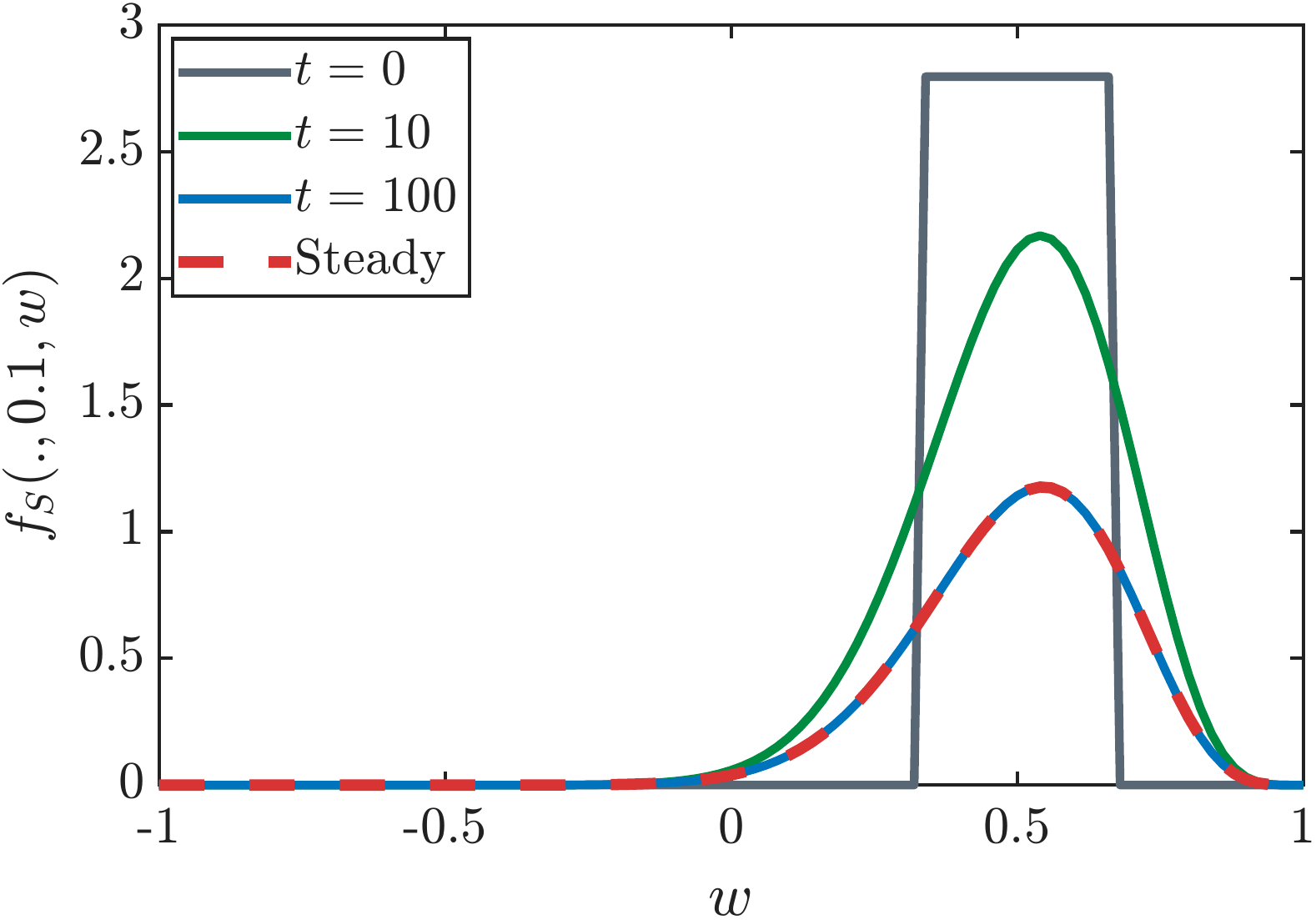}
        \caption*{(A)}
    \end{subfigure}\hfill
    \begin{subfigure}[t]{0.49\textwidth}
        \centering
        \includegraphics[width=\linewidth]{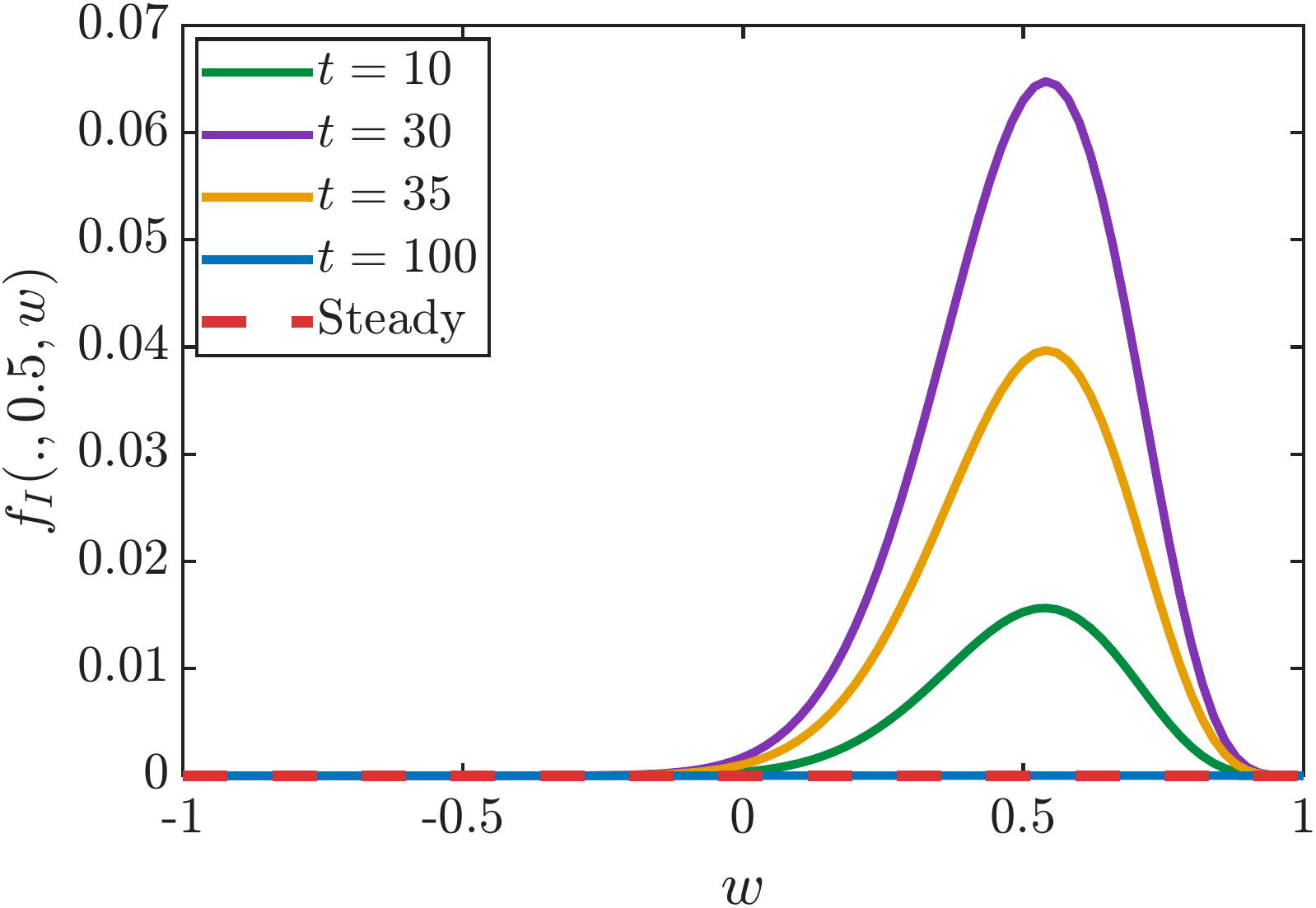}
        \caption*{(B)}
    \end{subfigure}

    \vspace{0.8em}

    \begin{subfigure}[t]{0.49\textwidth}
        \centering
        \includegraphics[width=\linewidth]{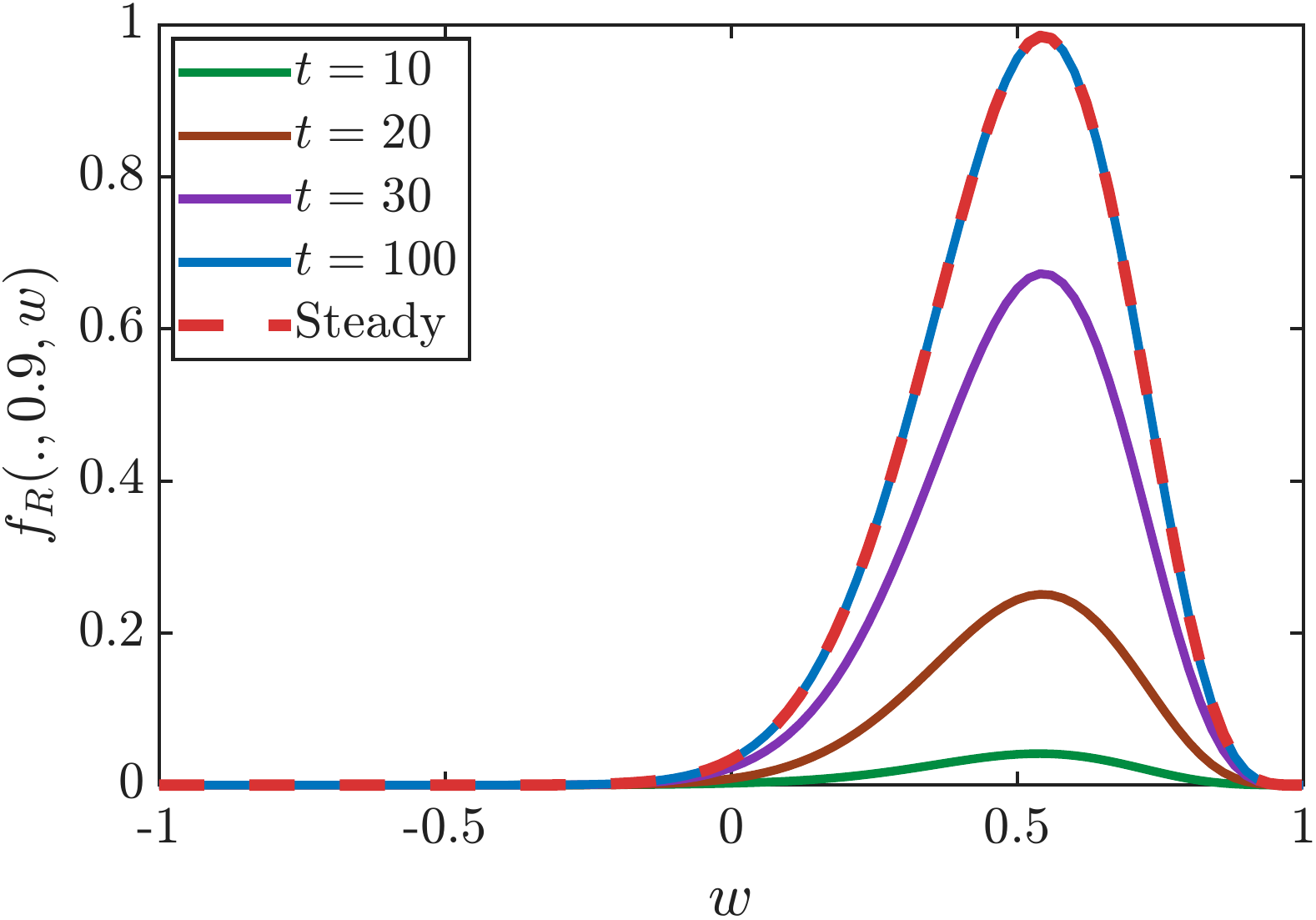}
        \caption*{(C)}
    \end{subfigure}\hfill
    \begin{subfigure}[t]{0.49\textwidth}
        \centering
        \includegraphics[width=\linewidth]{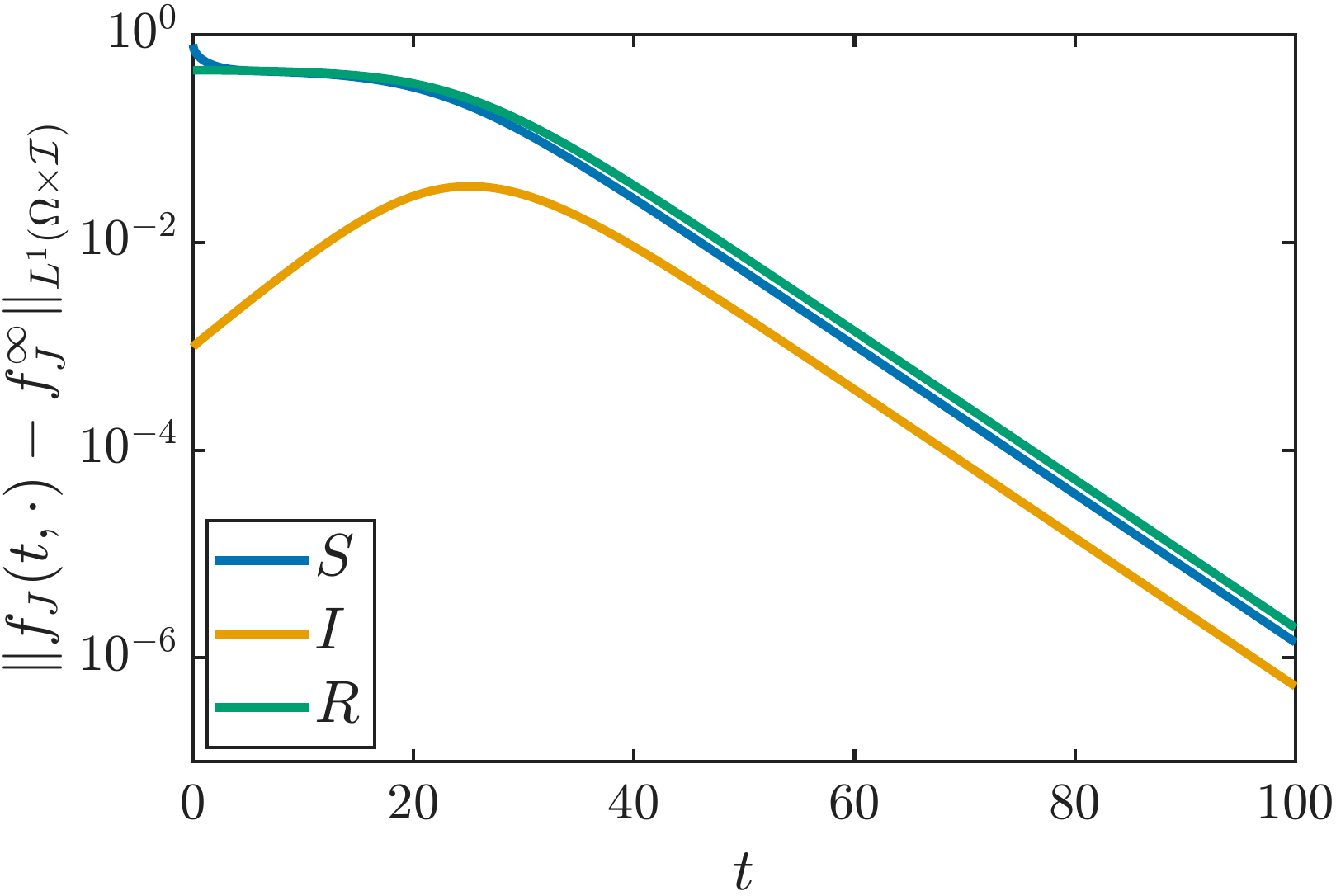}
        \caption*{(D)}
    \end{subfigure}

  \caption{\textbf{Test 1.} Evolution of the simplified model \eqref{eq:vectorial model simplified} with epidemiological transition rate $\beta_T$ defined by \eqref{eq:function beta_T} and graphon $\mathcal{B}$ given by \eqref{eq:graphon_app}. Opinion distributions of susceptible (A), infected (B), and removed (C) individuals at different time instants and different fixed positions $x \in \Omega$ on the graphon. Note that only $f_S^\init$ is displayed, since $f_I^\init$ and $f_R^\init$ are uniform distributions with very small mass. The dashed curves represent the global equilibria \eqref{eq:f_inf}. The last plot (D) shows the time-evolution of the $L^1(\Omega \times \mathcal{I})$ distance between each distribution $f_J$ and its corresponding equilibrium $f_J^\infty$. Values of the parameters: $\beta=0.8$, $\alpha = 0$ \eqref{eq:function beta_T}, $\gamma=0.6$, $\lambda=1$, $\sigma^2 = 0.01$, $\tau=1$, $r = 0.2$ \eqref{eq:graphon_app}, $\chi=0.5$ \eqref{eq:P}, and $a=0.5$ \eqref{eq:g_P}. }
\label{IC4 P6 CONSENSO}
\end{figure}

\subsection{Test 2 -- Trends to equilibrium for the popularity}

We proceed by investigating the trends to equilibrium for the popularity model \eqref{eq:FP-SIR and popularity}. We start from the previous simulations to reproduce the epidemic spread. The discretization parameters introduced in Section \ref{code_popularity} are fixed as follows: the tail tolerance is set to $\varepsilon_{\mathrm{tail}} = 10^{-14}$, the minimum admissible domain size is $L_{\min} = 8\,v_{\mathrm{peak}}$, and the number of grid points in the variable $v$ is constrained within the range $N_{\min} = 101$ and $N_{\max} = 2001$ (see the published code \cite{MATLAB} for more details). 

Figure \ref{fig:trend_pop} shows the results of this analysis. Looking at the second test, it is visible that $\mathcal{F}[f]$ converges exponentially fast toward $\mathcal{F}^\infty$, implying the $L^1(\R_+)$ convergence of $h$ toward $h^\infty$, as anticipated by Remark \ref{rmk:L1_conv}. Moreover, this relaxation appears to be exponential and the $L^1(\R_+)$ norm of the difference between the solution and the corresponding equilibrium reaches values of order $10^{-3}$. The first test shows a similar behavior, but a less precise converge to equilibrium due to the propagation of errors in the simulation of the kinetic SIR model in presence of opinion polarization. 
Note that, since in general a closed form of the equilibrium is not available for system \eqref{eq:vectorial model}, the steady value $\mathcal{F^\infty}$ is approximated using the numerical solution of the kinetic SIR model obtained at the final time $T$ of our simulations, meaning that we take $\mathcal{F^\infty} = \mathcal{F}[f](T)$.

\begin{figure}[h!] 
    \centering
    
    \begin{subfigure}[t]{0.5\textwidth}
        \centering
        \includegraphics[width=\linewidth]{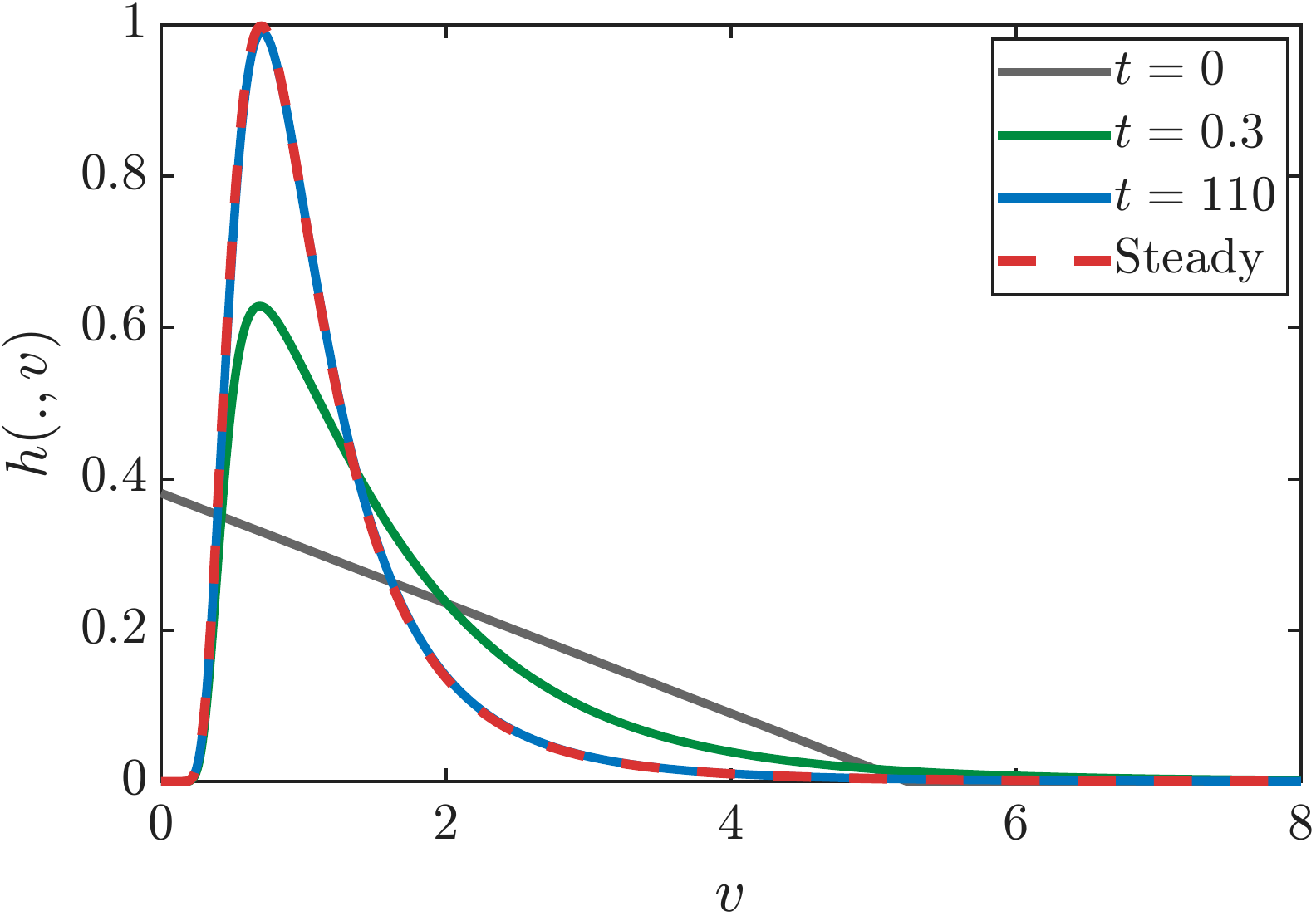}
        \caption*{(A)}
    \end{subfigure}\hfill
    \begin{subfigure}[t]{0.5\textwidth}
        \centering
        \includegraphics[width=\linewidth]{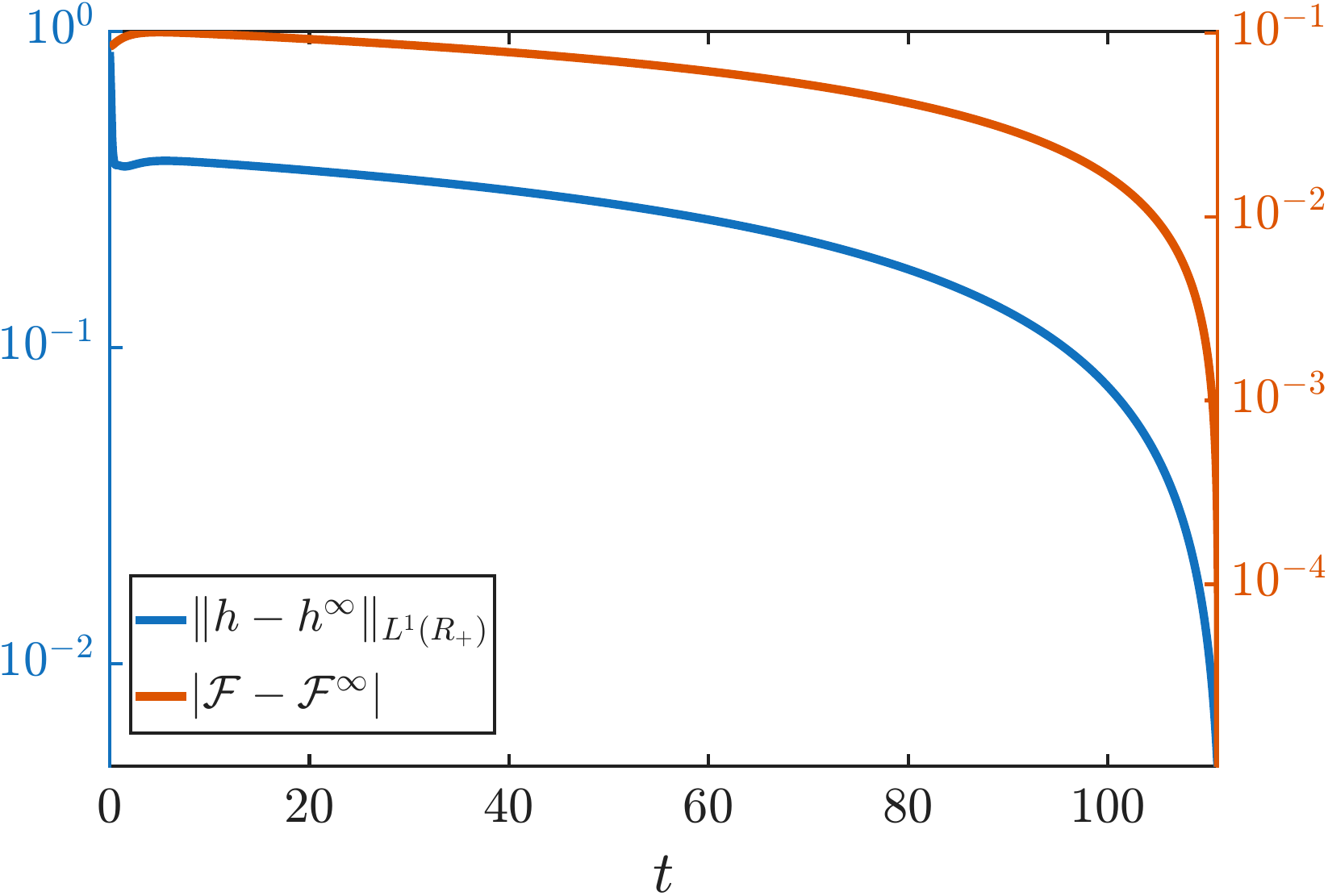}
        \caption*{(B)}
    \end{subfigure}

    \vspace{0.8em}

    \begin{subfigure}[t]{0.49\textwidth}
        \centering
        \includegraphics[width=\linewidth]{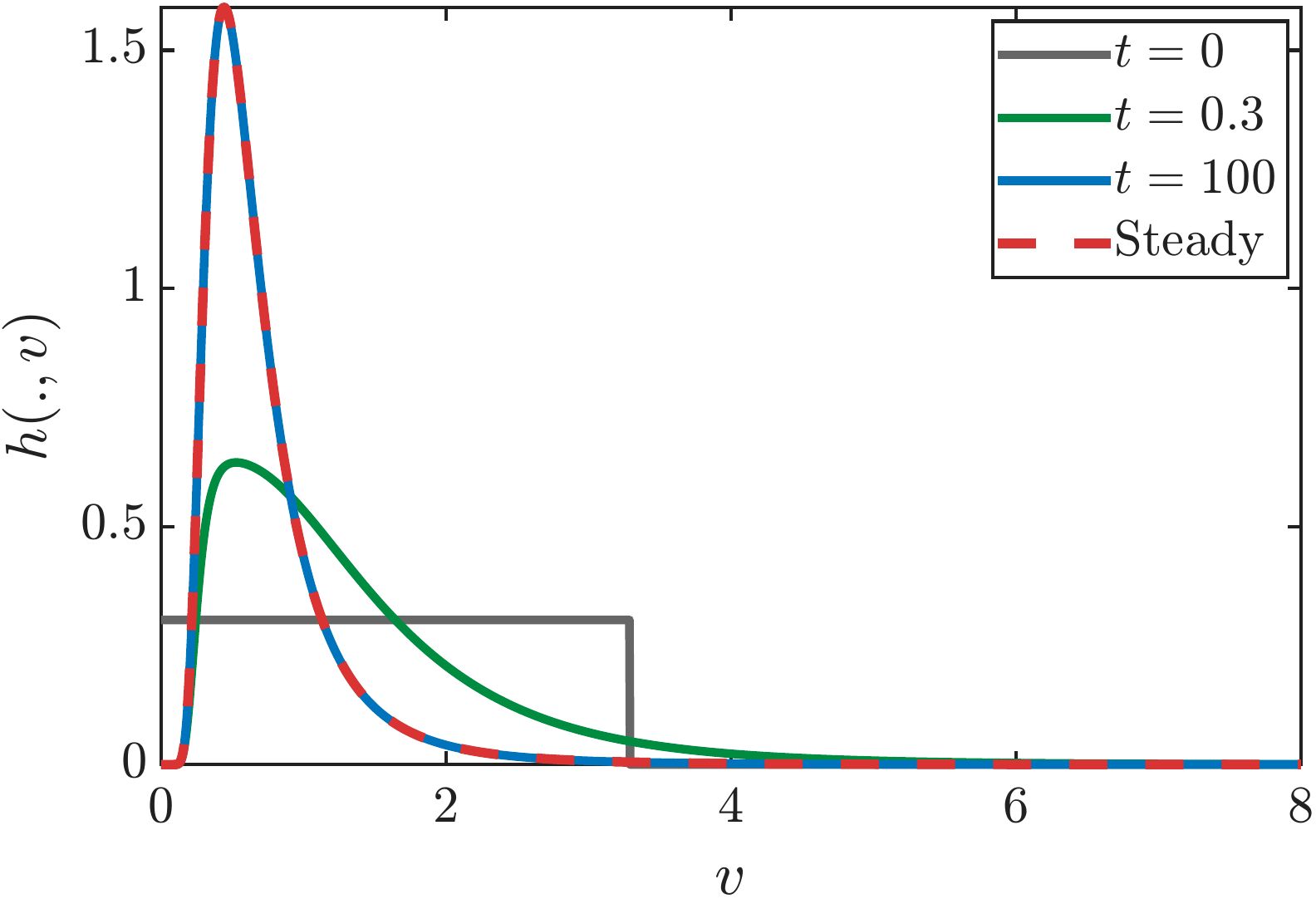}
        \caption*{(C)}
    \end{subfigure}\hfill
    \begin{subfigure}[t]{0.49\textwidth}
        \centering
        \includegraphics[width=\linewidth]{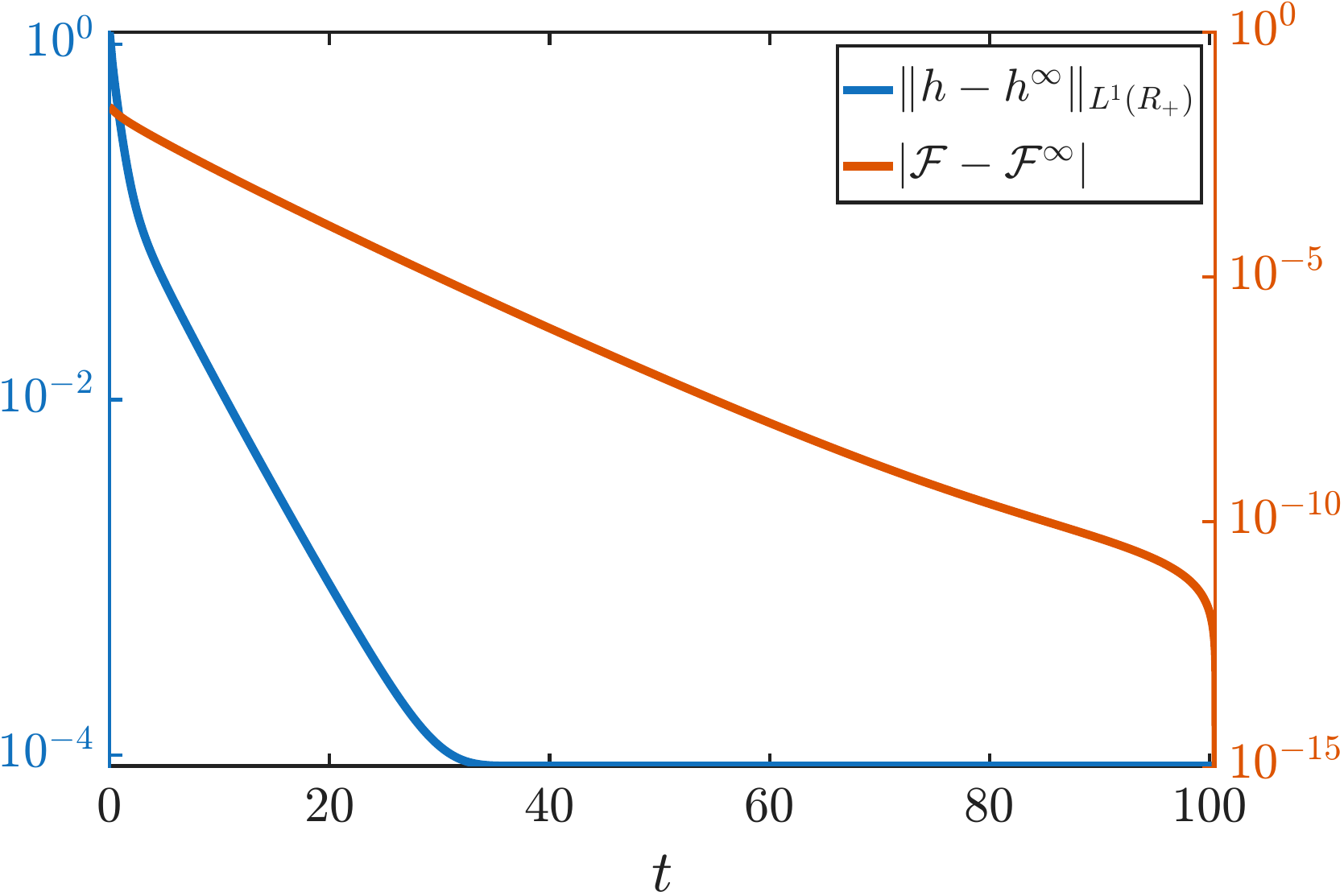}
        \caption*{(D)}
    \end{subfigure}

  \caption{\textbf{Test 2.} Evolution of the popularity model \eqref{eq:FP-SIR and popularity} depending on that of $\mathbf{f}$, which varies as in Figure \ref{IC5 P2 POLARIZZ} ((A) and (B) with $\hat{w}=-0.3$) or as in Figure \ref{IC4 P6 CONSENSO} ((C) and (D) with $\hat{w}=0.3$). In figures (A) and (C) we plot the distribution of the popularity at different time instants. The dashed curve represents the global equilibrium \eqref{eq:global equilibria popularity FP}. Figures (B) and (D) show the evolution of the quantity $|\mathcal{F}[f](t)-\mathcal{F}^\infty|$ and of the $L^1(\R_+)$ distance between the distribution $h$ from its corresponding equilibrium $h^\infty$. Values of the parameters: $\mu=1.5$, $\theta=5$, $\zeta^2 =1$, and $\tau_p=1$.}
\label{fig:trend_pop}
\end{figure}

\subsection{Test 3 -- Oscillations of the reproduction number and waves of infection}

We next aim to show that, despite its reduced complexity, our simplified model \eqref{eq:vectorial model simplified} is able to produce rich dynamics like the formation of epidemic waves without resorting to the use of external opinion controls \cite{BonTosZan}. For this, we choose $\alpha = 1$ in the epidemiological transition rate function $\beta_T$ defined by \eqref{eq:function beta_T}, in order to recover at the macroscopic level the generalized SIR system \eqref{eq:generalized SIR} which is associated with the time-dependent effective reproduction number \eqref{eq:effective reproduction number}. We then consider the fat-tailed graphon $\mathcal{B}$ given by \eqref{eq:fat-tailed graphon} and the interaction function $P$ defined by \eqref{eq:P}, associated with the presence of agents with a low propensity to interact (blue curve in Figure \ref{fig:propensity to interact}; note that, due to the cutoff, $0< \min_{x \in \Omega} p(x) \sim 10^{-3}$, hence the model is well-defined). Moreover, we assume that the exchanges in $\mathcal{I}$ are modulated by a bounded confidence-type function \cite{HegKra} of the form
\begin{equation} \label{eq:bounded confidence}
    G(w,w_*)=
        \left\{
        \begin{aligned}
            & 1 \quad \textnormal{if}\ \  |w-w_*| \leq \Delta, \\[2mm] 
            & 0 \quad \textnormal{if} \ \   |w-w_*| > \Delta,
        \end{aligned}
        \right.
\end{equation}
depending on a parameter $\Delta \in (0,2)$, which allows for interactions only between individuals that have close enough opinions. We initialize the distributions $f_S$, $f_I$, and $f_R$ as
\begin{gather*}
    f_S^\init(x,w) = \rho_S^\init \left( \frac{c}{4}\mathbb{I}_{A_1}(x,w) + c \mathbb{I}_{A_2}(x,w)\right), \\[4mm] f_I^\init(x,w)=\frac{\rho_I^\init}{2}\mathbb{I}_{\Omega \times \mathcal{I}}(x,w), \qquad f_R^\init(x,w)=\frac{\rho_R^\init}{2} \mathbb{I}_{\Omega \times \mathcal{I}}(x,w),
\end{gather*}
where $\rho_I^\init = \rho_R^\init = 10^{-3}$ and $\rho_S^\init = 1-\rho_I^\init-\rho_R^\init$, $A_1$ and $A_2$ are congruent rectangular subsets of $\Omega \times \mathcal{I}$ centered around the points $(x_1,w_1) = (0.8,-0.6)$ and $(x_2,w_2) = (0.2,0.6)$ respectively, and $c > 0$ is a constant chosen such that $f_S^\init$ has mass $\rho_S^\init$. In particular, compared to the previous setting we consider an asymmetric initial datum for the population of susceptible individuals, with two main groups of opposite opinions and different levels of connectivity (agents with positive opinions being four times more numerous).

Figure \ref{fig:reduced model epidemic waves} shows the results of this test. Looking at the evolution of the density $\rho_I$, we deduce that the epidemic spreads over multiple consecutive waves of infection, whose impact decreases over time before fading out completely \cite{BonTosZan}. This phenomenon stems directly from the oscillating behavior of the effective reproduction number $\mathcal{R}_{\mathrm{eff}}$ \eqref{eq:effective reproduction number}, and is the result of non-equilibrium competitive interactions between highly connected susceptible and infected individuals that are strongly polarized on extreme positive and negative stances, respectively, causing the incidence rate (which depends on their mean opinions $m_S$ and $m_I$) in the generalized SIR system \eqref{eq:generalized SIR} to oscillate.

\begin{figure}
    \centering
    
    \begin{subfigure}[t]{0.5\textwidth}
        \centering
        \includegraphics[width=\linewidth]{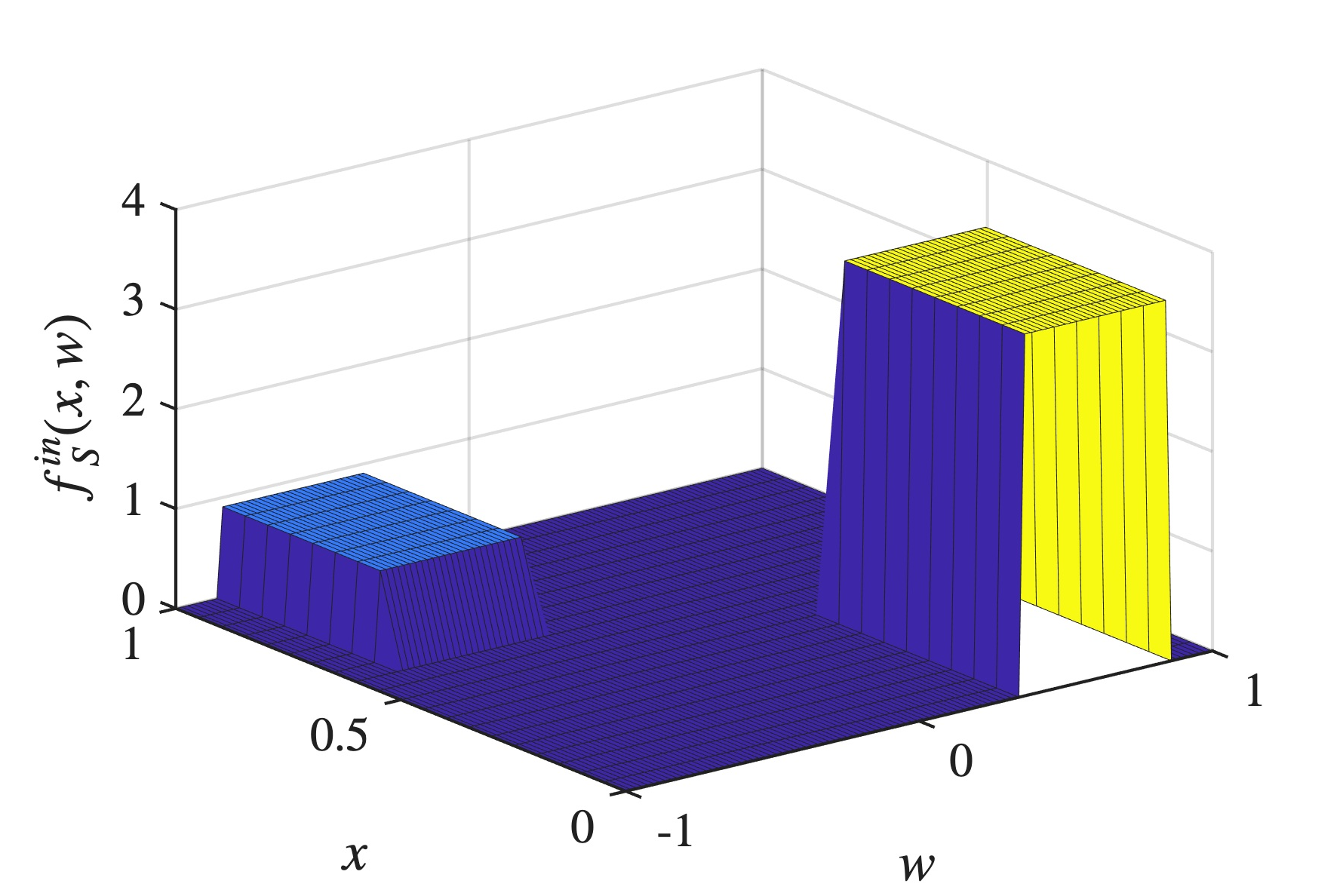}
        \caption*{(A)}
    \end{subfigure}\hfill
    \begin{subfigure}[t]{0.5\textwidth}
        \centering
        \includegraphics[width=\linewidth]{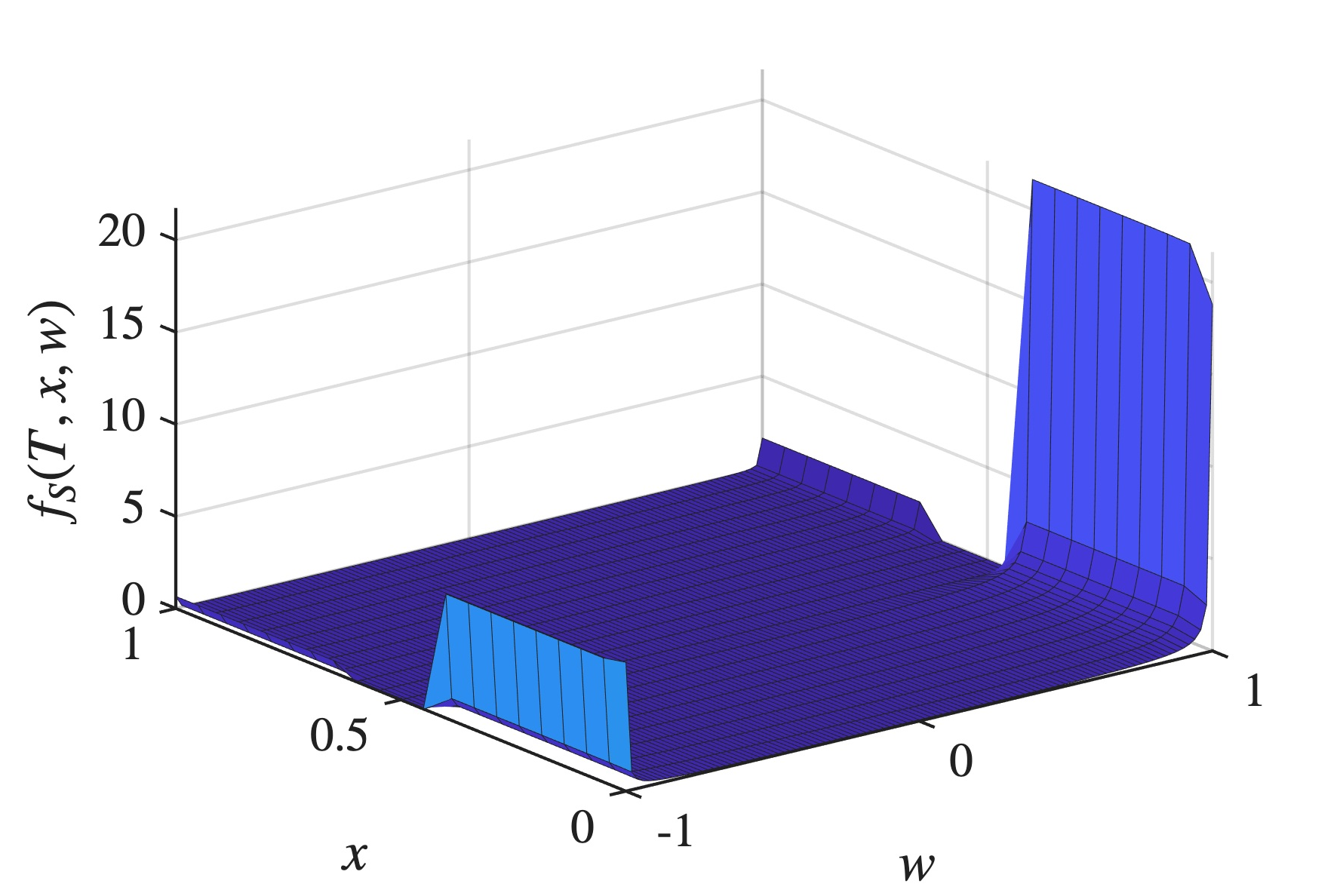}
        \caption*{(B)}
    \end{subfigure}

    \vspace{0.8em}

    \begin{subfigure}[t]{0.49\textwidth}
        \centering
        \includegraphics[width=\linewidth]{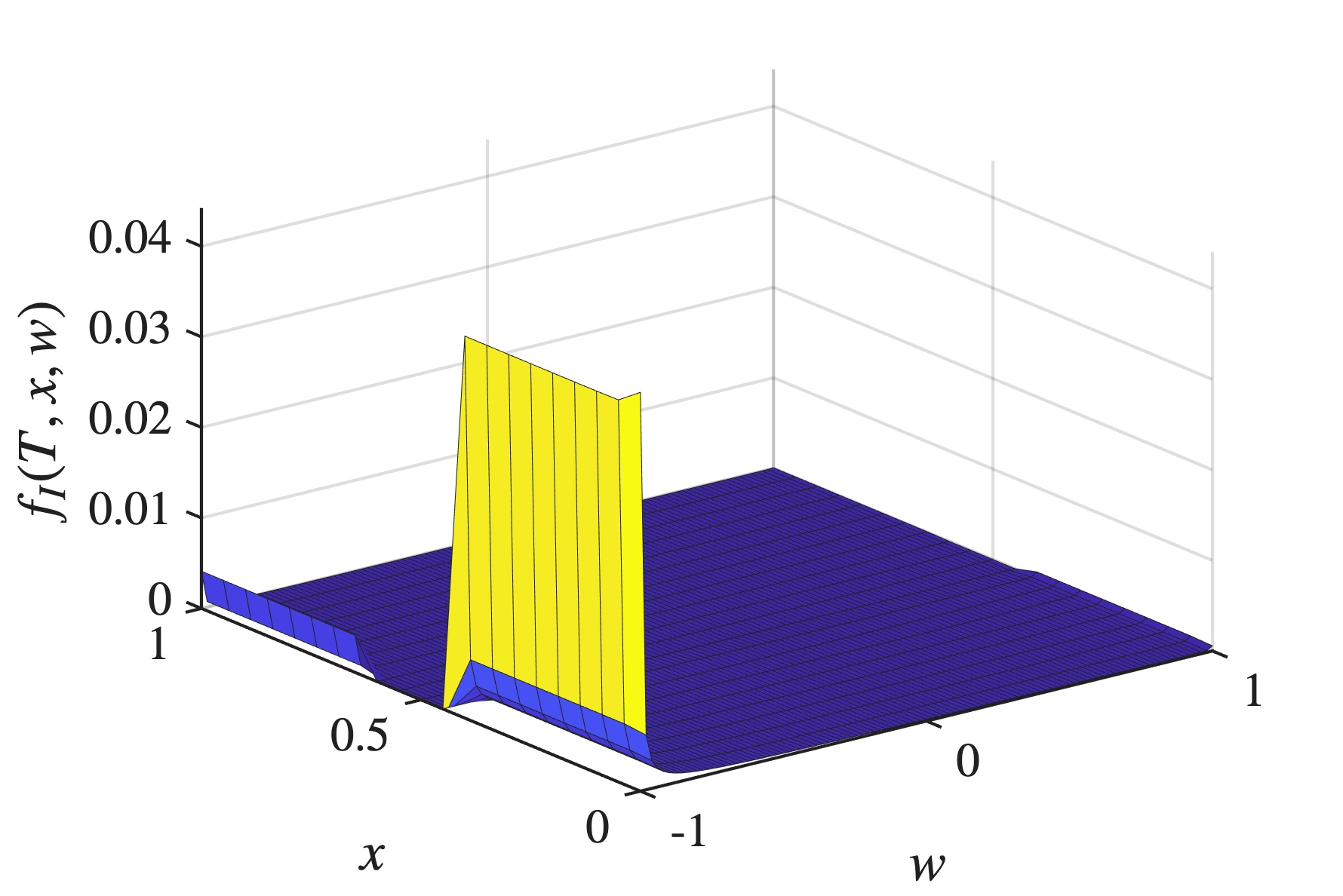}
        \caption*{(C)}
    \end{subfigure}\hfill
    \begin{subfigure}[t]{0.49\textwidth}
        \centering
        \includegraphics[width=\linewidth]{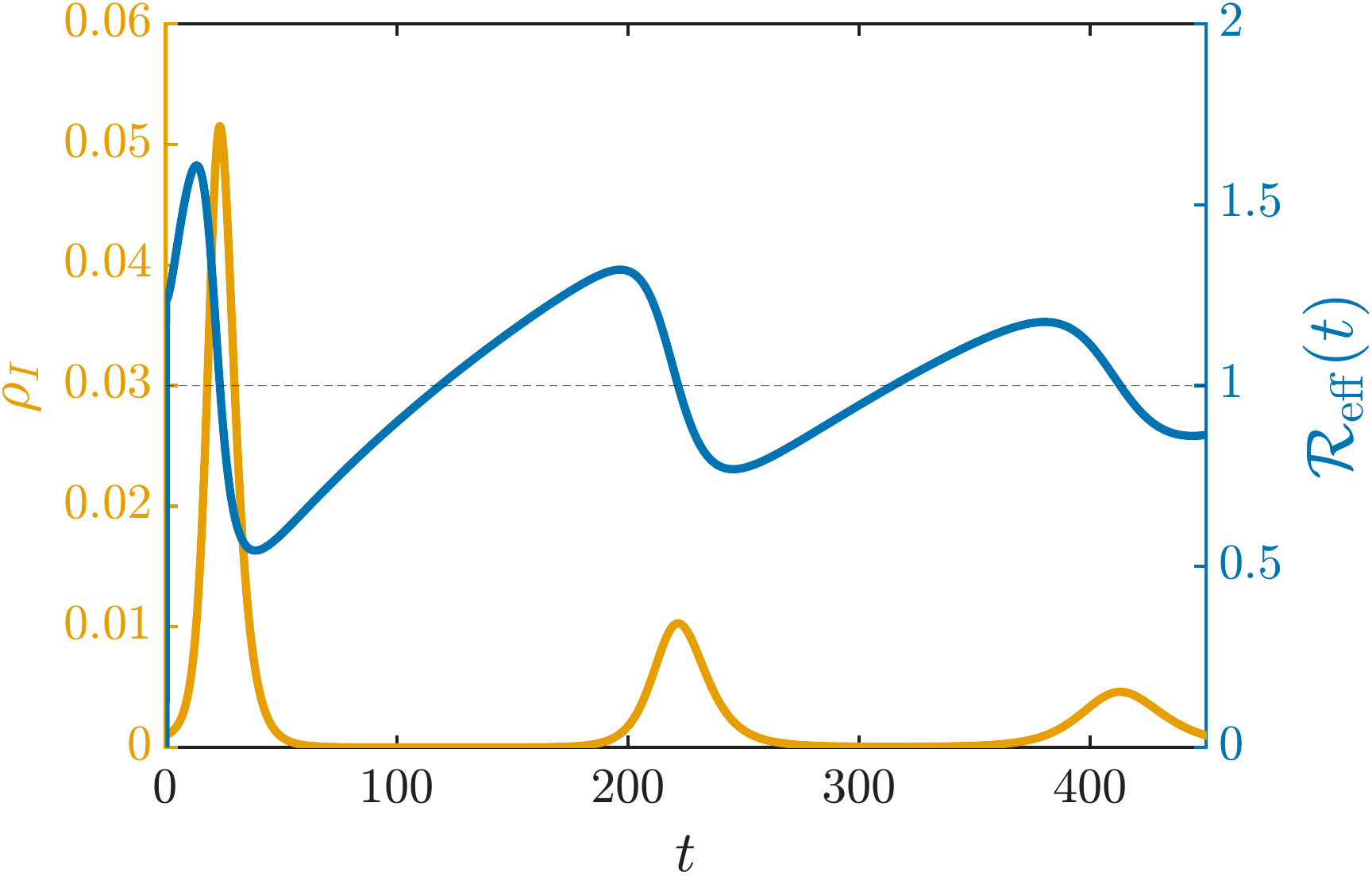}
        \caption*{(D)}
    \end{subfigure}
      \caption{\textbf{Test 3.} Evolution of the simplified model \eqref{eq:vectorial model simplified} with epidemiological transition rate $\beta_T$ defined by \eqref{eq:function beta_T}, graphon $\mathcal{B}$ given by \eqref{eq:fat-tailed graphon}, and opinion interaction function $G$ of the form \eqref{eq:bounded confidence}. Figure (A) shows the initial distribution of susceptible agents. At the final simulation time $T = 450$, the distributions of susceptible (B) and infected (C) populations display a polarization of their highly connected individuals around extreme positive and negative opinions, respectively. The effect of this competition becomes visible in the evolution of the effective reproduction number $\mathcal{R}_{\mathrm{eff}}$ which exhibits a strong oscillating behavior (D), reflected in the evolution of the density $\rho_I$. Values of the parameters: $\beta = 0.8$, $\alpha = 1$ \eqref{eq:function beta_T}, $\gamma = 0.4$, $\lambda = 1$, $\sigma_S^2 = 0.05$, $\sigma_I^2 = 0.03$, $\sigma_R^2 = 0.01$, $\tau = 1$, $\xi = 0.25$ \eqref{eq:fat-tailed graphon}, $\chi = 2$ \eqref{eq:P}, $a = 1$ \eqref{eq:g_P}, and $\Delta = 0.5$ \eqref{eq:bounded confidence}.}
\label{fig:reduced model epidemic waves}
\end{figure}

\subsection{Test 4 -- Epidemic dynamics in absence and presence of opinion leaders}

We now turn to the original kinetic SIR model \eqref{eq:vectorial model}, to investigate the impact of leaders on the epidemic dynamics. Let us still fix $\alpha = 1$ in the function $\beta_T$ defined by \eqref{eq:function beta_T} and consider a bounded confidence-type interaction $G$ of the form \eqref{eq:bounded confidence}. We initialize the distributions $f_S$, $f_I$, and $f_R$ as
\begin{gather*}
    f_S^\init(x,w) = \rho_S^\init \left( c \mathbb{I}_{B_1}(x,w) + \frac{c}{4} \mathbb{I}_{B_2}(x,w)\right), \\[4mm] f_I^\init(x,w)=\frac{\rho_I^\init}{2}\mathbb{I}_{\Omega \times \mathcal{I}}(x,w), \qquad f_R^\init(x,w)=\frac{\rho_R^\init}{2} \mathbb{I}_{\Omega \times \mathcal{I}}(x,w),
\end{gather*}
where $\rho_I^\init = \rho_R^\init = 10^{-3}$ and $\rho_S^\init = 1-\rho_I^\init-\rho_R^\init$, $B_1$ and $B_2$ are rectangular subsets of $\Omega \times \mathcal{I}$ (with unequal area) centered around the points $(x_1,w_1) = (0.8,-0.6)$ and $(x_2,w_2) = (0.2,0.7)$ respectively, and $c > 0$ is a constant chosen such that $f_S^\init$ has mass $\rho_S^\init$. In particular, we assume that poorly connected susceptible individuals are four times more numerous than highly connected ones, and that they are distributed around a negative opinion over a larger band compared to the latter, which are concentrated in the nearer vicinity of a strong positive stance.

In this setting, we perform two numerical tests by considering the fat-tailed graphon $\mathcal{B}$ given by \eqref{eq:fat-tailed graphon} and the interaction function $P$ defined by \eqref{eq:P}, in the two opposite cases modeling the absence or the presence of opinion leaders (exemplified by the red and blue curves of Figure \ref{fig:propensity to interact}, in the context of the simplified model \eqref{eq:vectorial model simplified}). In particular, between the two tests we solely modify the values of the parameters $\xi$ and $\chi$, which characterize the structure of the functions $\mathcal{B}$ and $P$, while maintaining fixed the values of all other parameters. Notice that we now make use of a stronger cutoff $10^{-3}$ to deal with the singularities of the graphon $\mathcal{B}$, in order to recover reasonable values for the opinion step $\Delta w$ based on the CFL condition \eqref{eq:CFL kinetic SIR}.

We report in Figure \ref{fig:full model no leaders} the results of the first test where no leaders are considered, namely when highly connected individuals, concentrated around a positive opinion, are not able to influence the bulk of the population which instead forms an unfavorable opinion cluster. In this case the former, less numerous group of individuals is driven toward the region of negative opinions, and the population reaches a global consensus opposing the use of protective behaviors. This dynamic directly affects the overall number of infected individuals, resulting in a greater spread of the epidemic. 

Figure \ref{fig:full model leaders} shows the results of the second test where leaders are taken into account. In this case, the individuals with positive opinions, despite being much less numerous than those belonging to the unfavorable cluster, have strong enough connections to steer the dynamics of the population toward reaching a global consensus in favor of the use of protective behaviors. In particular, compared to the previous test, the presence of leaders leads to a considerable reduction of the infection peak $\rho_I^\textnormal{max}$, making the epidemic less impactful.

This phenomenon is confirmed by the evolution of the corresponding effective reproduction numbers \eqref{eq:effective reproduction number}, displayed in Figure \ref{fig:reproduction numbers full model}. Indeed, the presence of leaders  has the effect of reducing and shrinking the peak of $\mathcal{R}_{\mathrm{eff}}$, which explains the substantial decrease of infected individuals compared to the case when leaders are absent.

\begin{figure}
    \centering
    
    \begin{subfigure}[t]{0.5\textwidth}
        \centering
        \includegraphics[width=\linewidth]{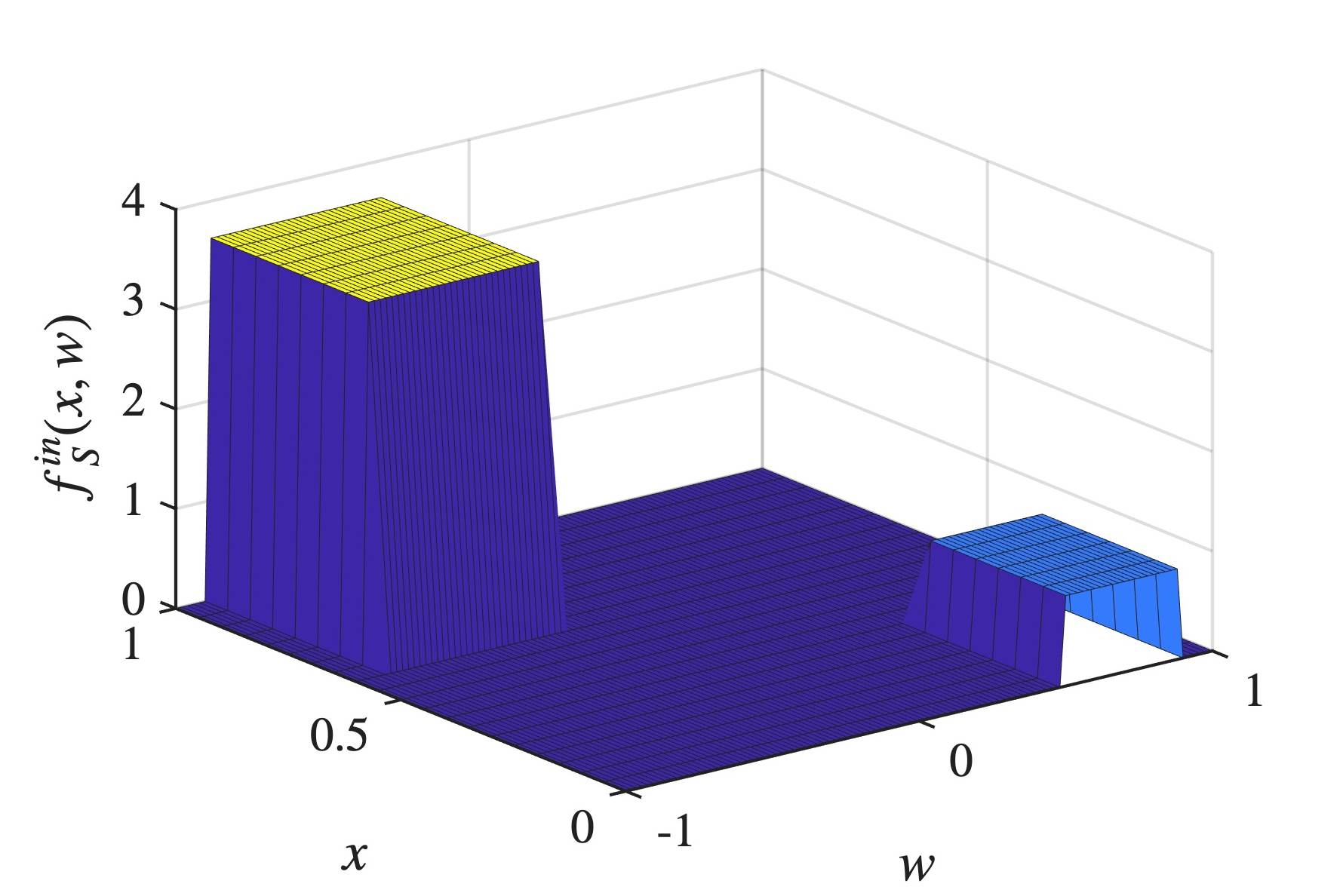}
        \caption*{(A)}
    \end{subfigure}\hfill
    \begin{subfigure}[t]{0.5\textwidth}
        \centering
        \includegraphics[width=\linewidth]{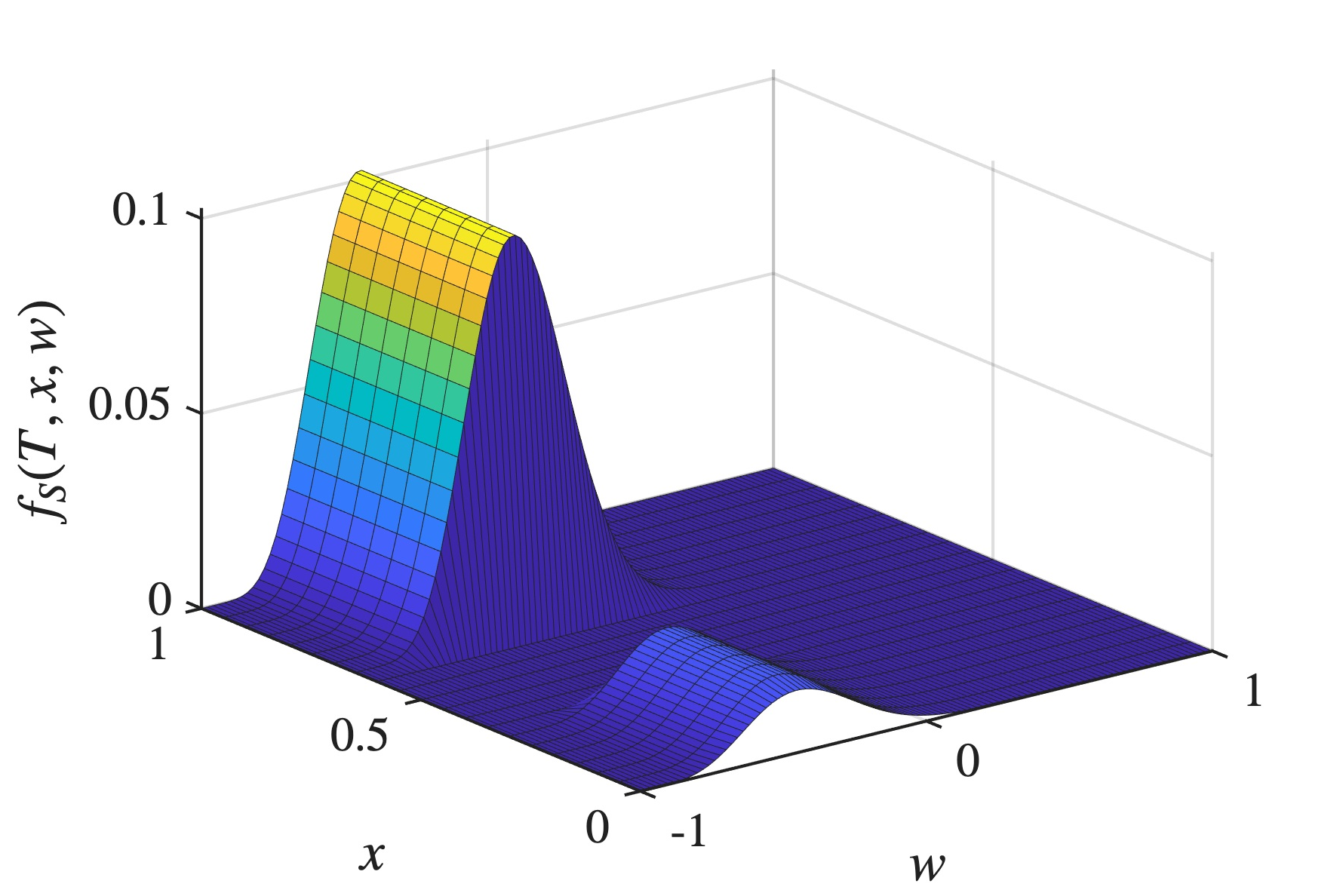}
        \caption*{(B)}
    \end{subfigure}

    \vspace{0.8em}

    \begin{subfigure}[t]{0.49\textwidth}
        \centering
        \includegraphics[width=\linewidth]{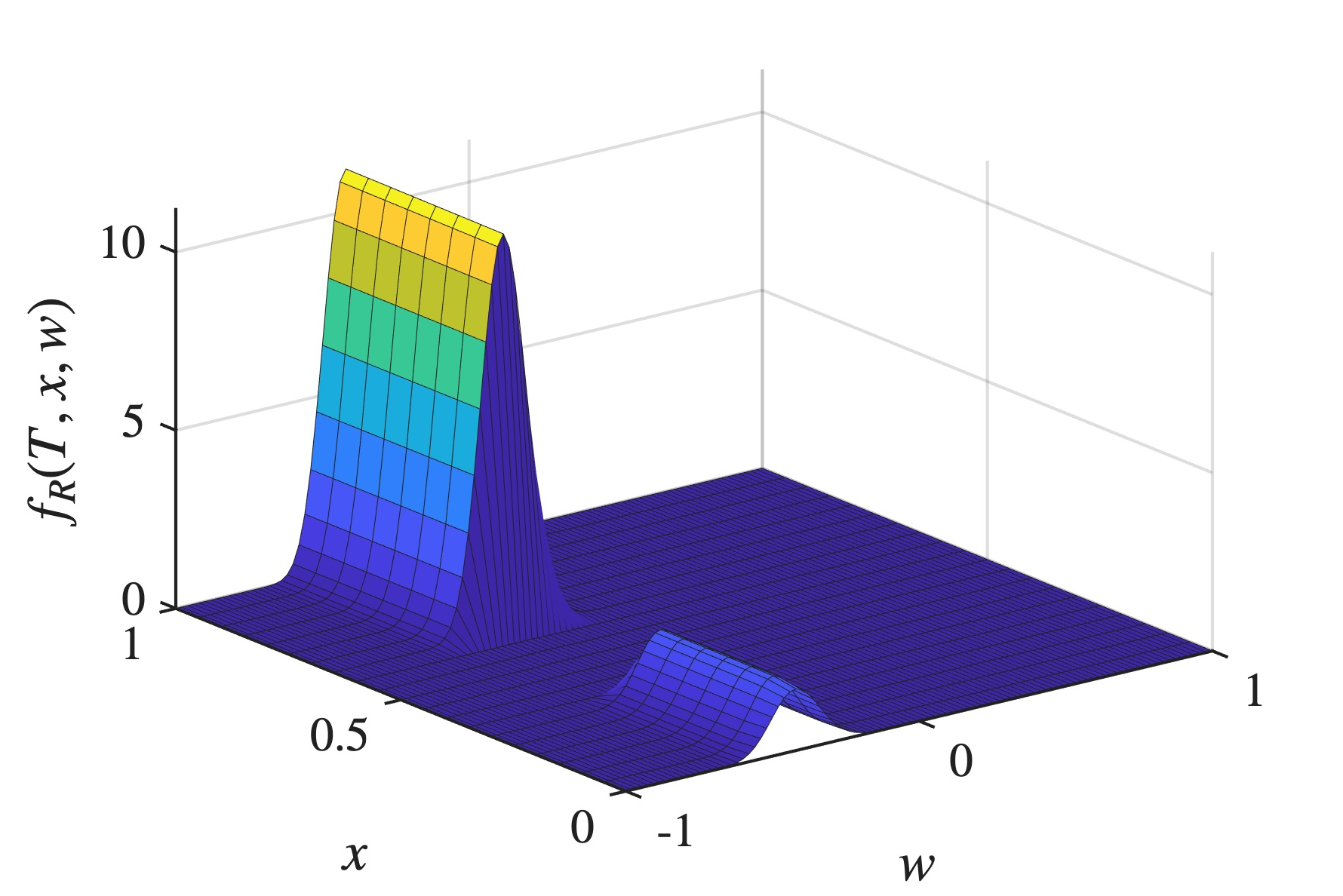}
        \caption*{(C)}
    \end{subfigure}\hfill
    \begin{subfigure}[t]{0.49\textwidth}
        \centering
        \includegraphics[width=\linewidth]{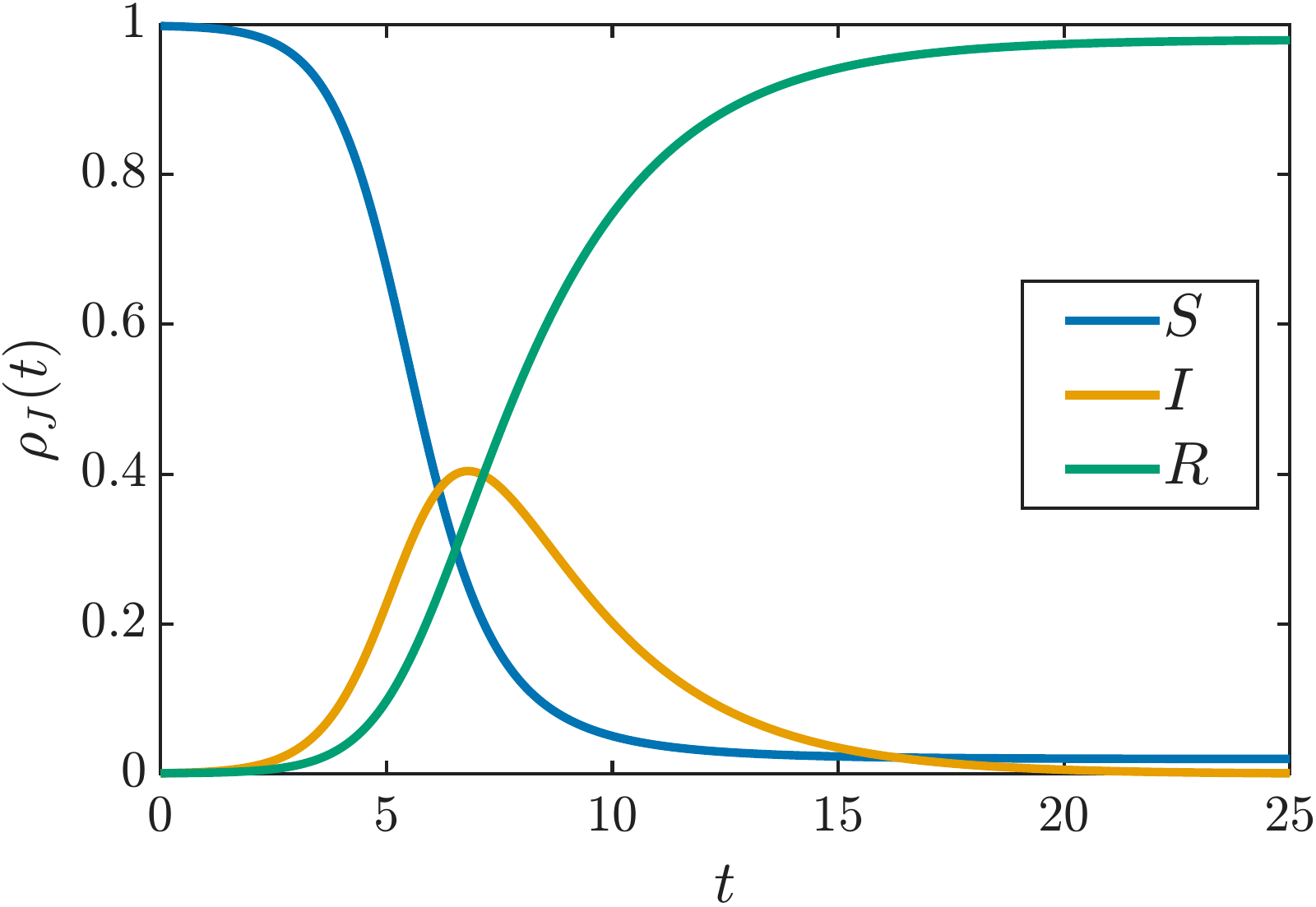}
        \caption*{(D)}
    \end{subfigure}
      \caption{\textbf{Test 4.} Evolution of the original model \eqref{eq:vectorial model} with epidemiological transition rate $\beta_T$ defined by \eqref{eq:function beta_T}, graphon $\mathcal{B}$ given by \eqref{eq:fat-tailed graphon}, and opinion interaction function $G$ of the form \eqref{eq:bounded confidence}. The initial population of susceptible individuals (A) is split between two opposite opinions, depending on the connectivity levels, with the unfavorable part being poorly connected and four times larger than the highly connected, favorable one. We plot the distributions of susceptible (B) and recovered (C) individuals at the final simulation time $T = 25$. The last figure (D) shows the evolution of the associated macroscopic SIR system \eqref{eq:generalized SIR} determined by computing the zeroth-order moment of the kinetic model. The absence of leaders leads to the formation of a global consensus around negative opinions for both susceptible and recovered compartments, and we observe a greater spread of the epidemic. Values of the parameters: $\beta = 0.8$, $\alpha = 1$ \eqref{eq:function beta_T}, $\gamma = 0.4$, $\lambda = 1$, $\sigma_S^2 = 0.05$, $\sigma_I^2 = 0.03$, $\sigma_R^2 = 0.01$, $\tau = 1$, $\xi = 0.05$ \eqref{eq:fat-tailed graphon}, $\chi = 0.5$ \eqref{eq:P}, and $\Delta = 1.2$ \eqref{eq:bounded confidence}.}
\label{fig:full model no leaders}
\end{figure}

\begin{figure}
    \centering
    
    \begin{subfigure}[t]{0.5\textwidth}
        \centering
        \includegraphics[width=\linewidth]{figures/5.6/FullModel_NoLeaders_fS_in.pdf}
        \caption*{(A)}
    \end{subfigure}\hfill
    \begin{subfigure}[t]{0.5\textwidth}
        \centering
        \includegraphics[width=\linewidth]{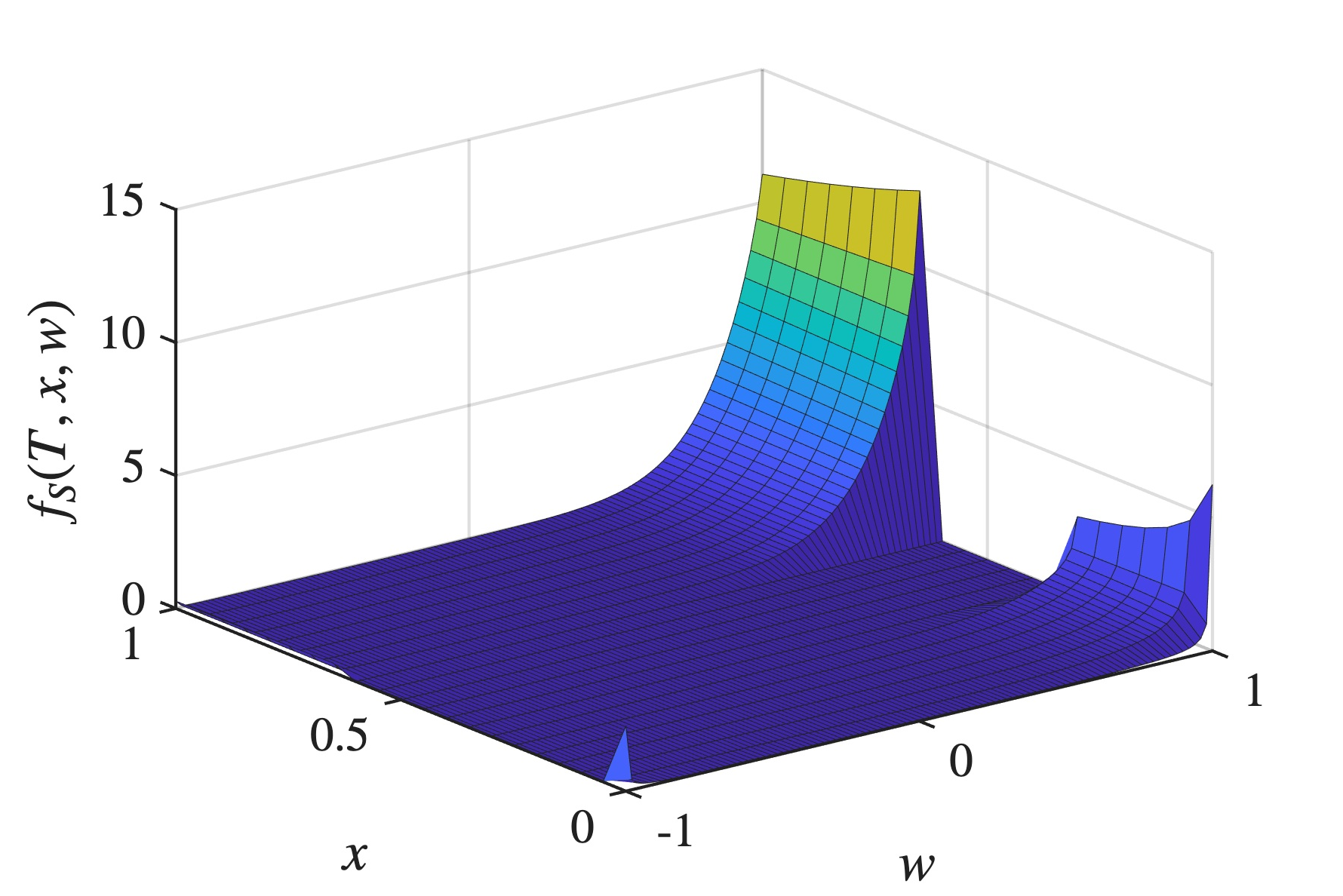}
        \caption*{(B)}
    \end{subfigure}

    \vspace{0.8em}

    \begin{subfigure}[t]{0.49\textwidth}
        \centering
        \includegraphics[width=\linewidth]{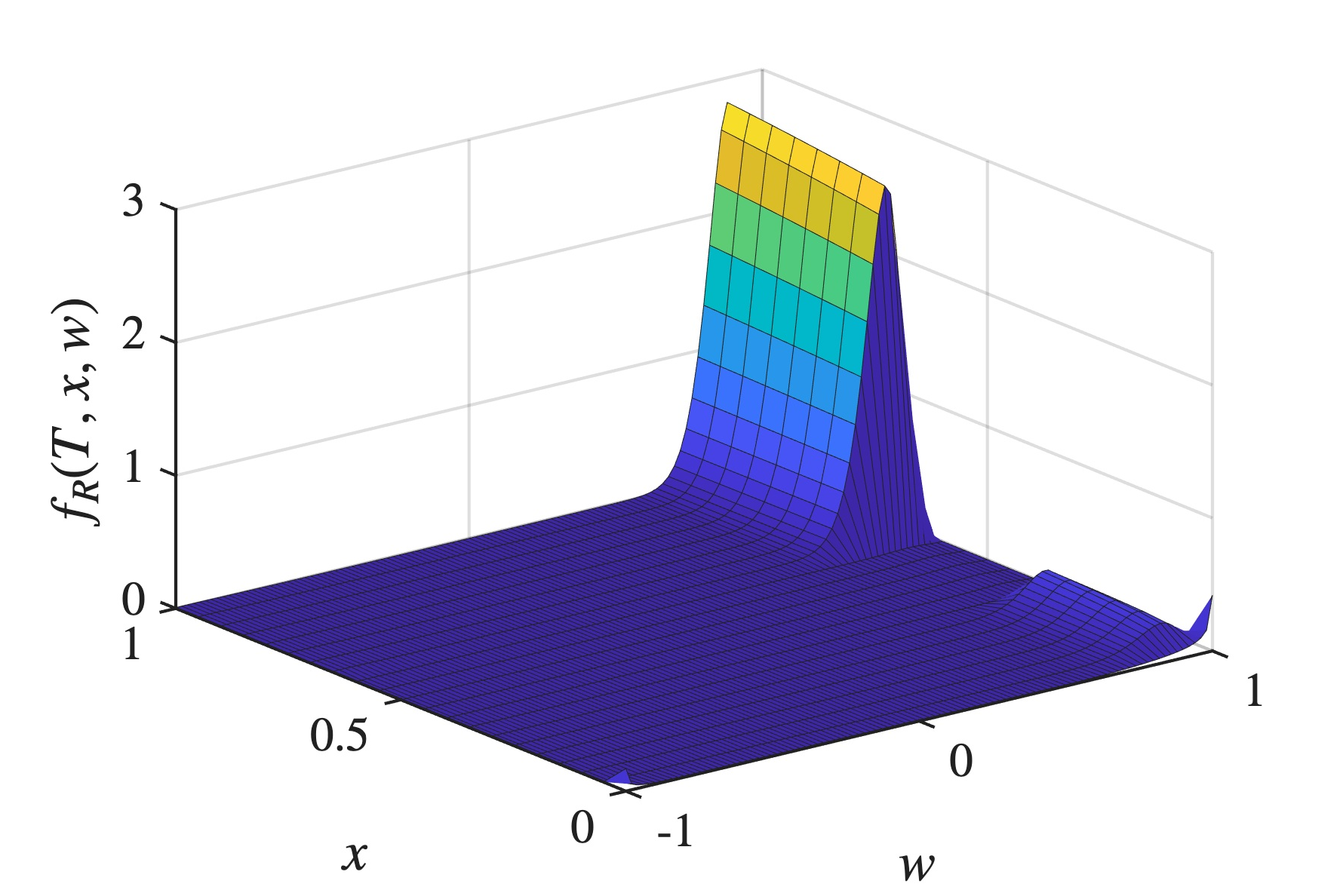}
        \caption*{(C)}
    \end{subfigure}\hfill
    \begin{subfigure}[t]{0.49\textwidth}
        \centering
        \includegraphics[width=\linewidth]{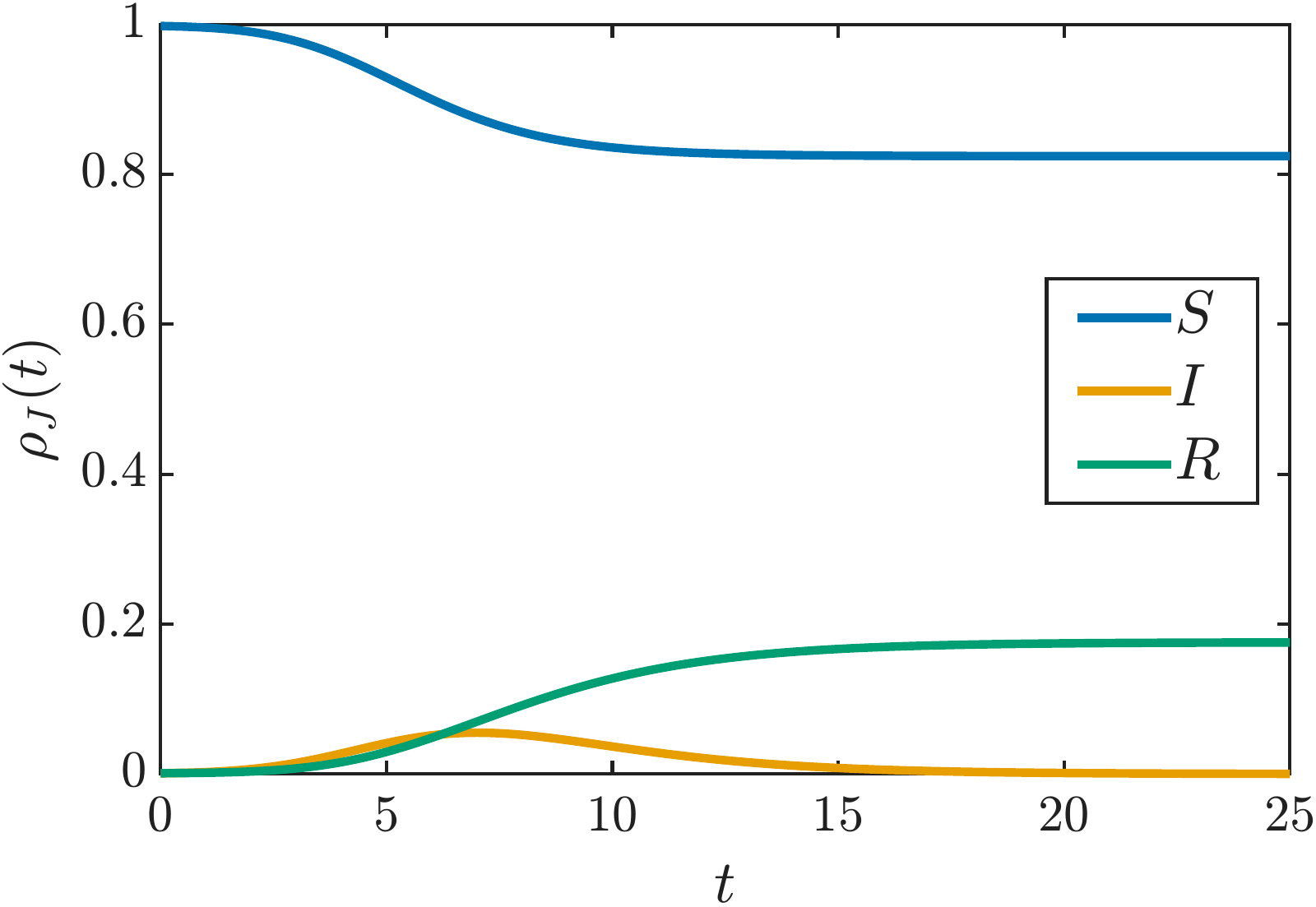}
        \caption*{(D)}
    \end{subfigure}
      \caption{\textbf{Test 4.} Evolution of the original model \eqref{eq:vectorial model} with epidemiological transition rate $\beta_T$ defined by \eqref{eq:function beta_T}, graphon $\mathcal{B}$ given by \eqref{eq:fat-tailed graphon}, and opinion interaction function $G$ of the form \eqref{eq:bounded confidence}. The initial population of susceptible individuals (A) is split between two opposite opinions, depending on the connectivity levels, with the unfavorable part being poorly connected and four times larger than the highly connected, favorable one. We plot the distributions of susceptible (B) and recovered (C) individuals at the final simulation time $T = 25$. The last figure (D) shows the evolution of the associated macroscopic SIR system \eqref{eq:generalized SIR} determined by computing the zeroth-order moment of the kinetic model. Leaders are able to steer the populations of susceptible and recovered agents toward a global consensus around positive opinions, and the impact of the epidemic is considerably reduced. Values of the parameters: $\beta = 0.8$, $\alpha = 1$ \eqref{eq:function beta_T}, $\gamma = 0.4$, $\lambda = 1$, $\sigma_S^2 = 0.05$, $\sigma_I^2 = 0.03$, $\sigma_R^2 = 0.01$, $\tau = 1$, $\xi = 0.25$ \eqref{eq:fat-tailed graphon}, $\chi = 2$ \eqref{eq:P}, and $\Delta = 1.2$ \eqref{eq:bounded confidence}.}
\label{fig:full model leaders}
\end{figure}

\begin{figure}
    \centering
    
    \begin{subfigure}[t]{0.5\textwidth}
        \centering
        \includegraphics[width=\linewidth]{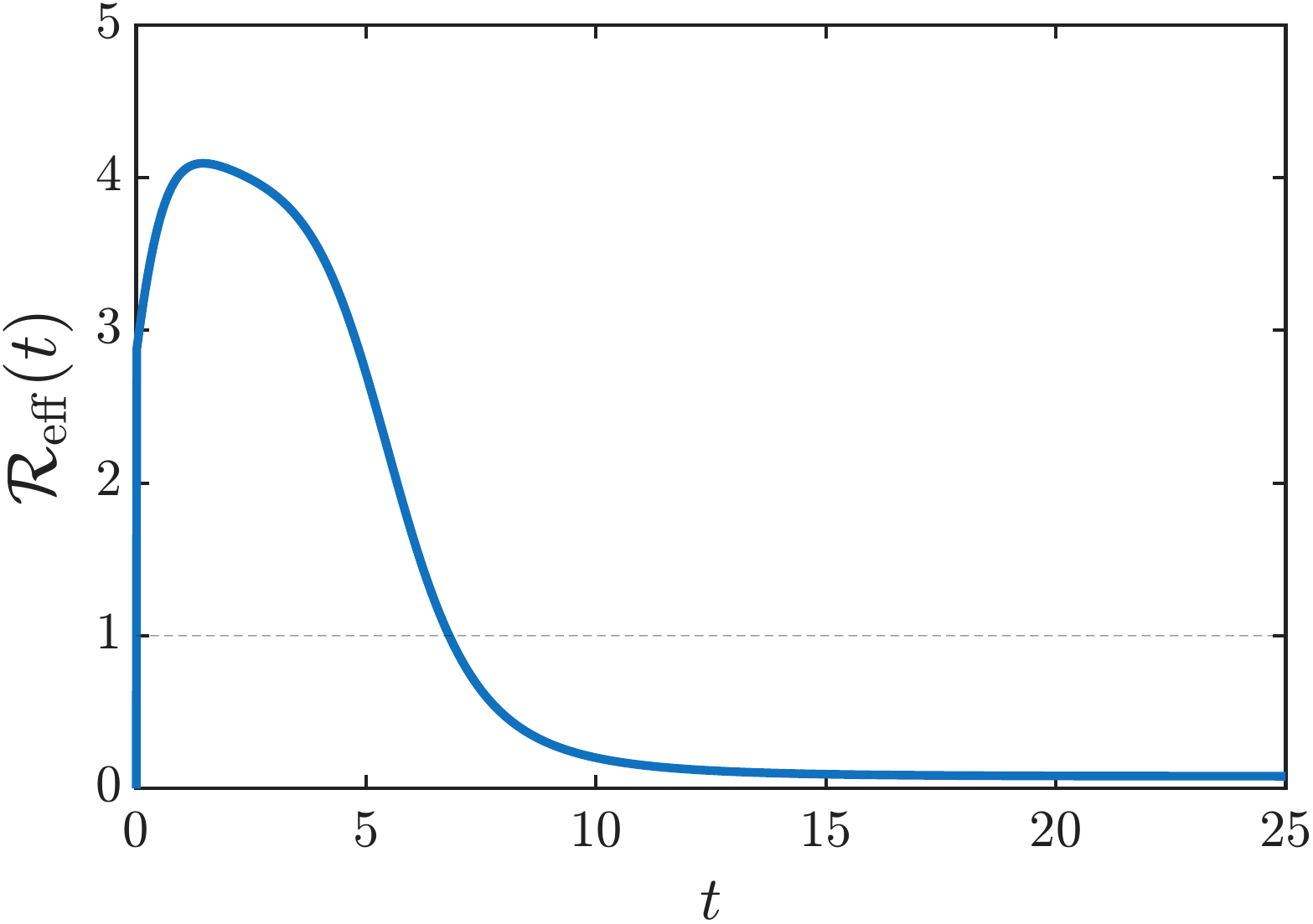}
        \caption*{(A)}
    \end{subfigure}\hfill
    \begin{subfigure}[t]{0.5\textwidth}
        \centering
        \includegraphics[width=\linewidth]{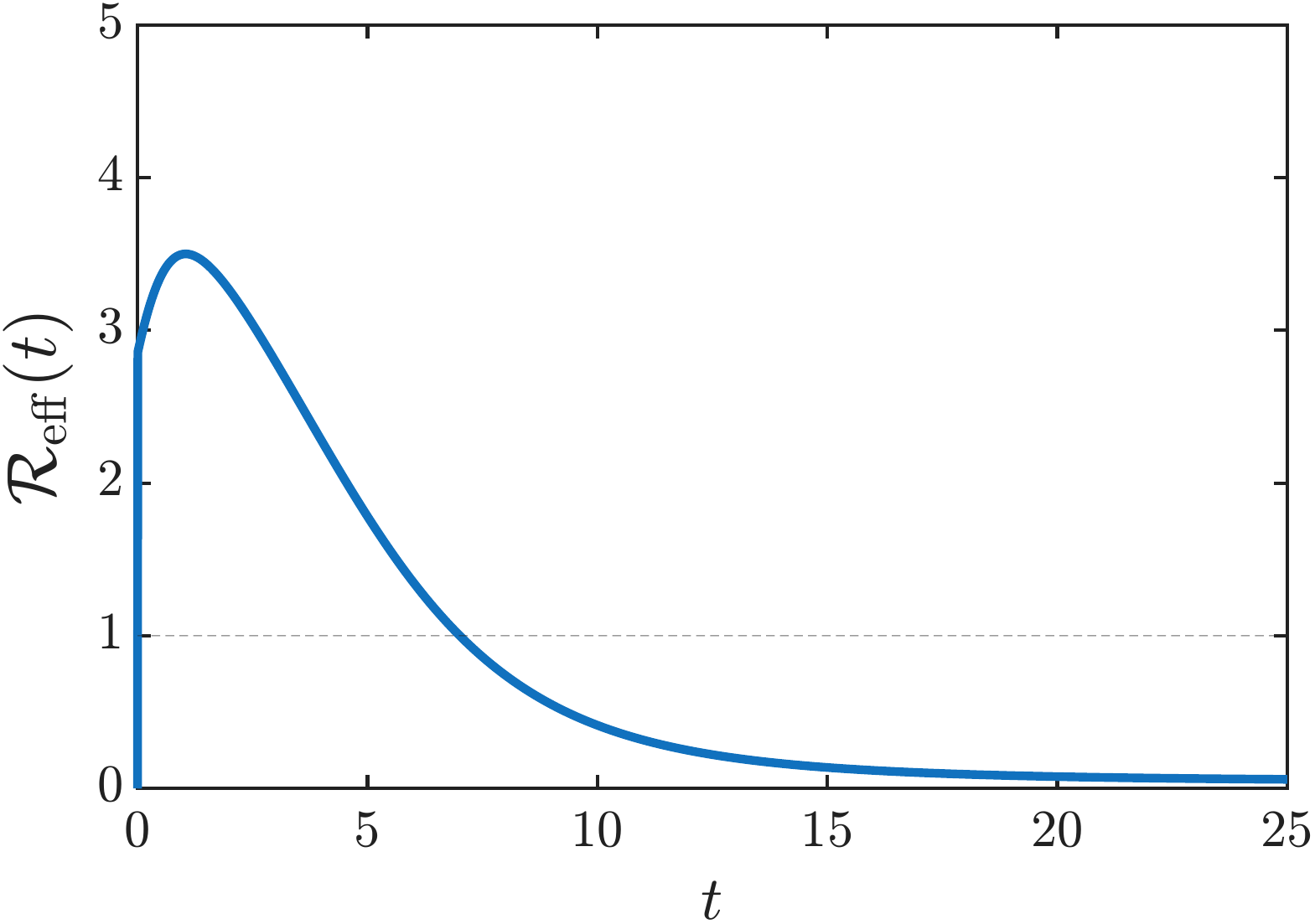}
        \caption*{(B)}
    \end{subfigure}
      \caption{\textbf{Test 4.} Evolution of the effective reproduction number \eqref{eq:effective reproduction number}, for the original model \eqref{eq:vectorial model} with epidemiological transition rate $\beta_T$ defined by \eqref{eq:function beta_T}, graphon $\mathcal{B}$ given by \eqref{eq:fat-tailed graphon}, and opinion interaction function $G$ of the form \eqref{eq:bounded confidence}. Temporal variations of $\mathcal{R}_{\mathrm{eff}}$ for the model without leaders (A) and for the model with leaders (B). The parameters are those used for Figure \ref{fig:full model no leaders} and Figure \ref{fig:full model leaders}, respectively.}
\label{fig:reproduction numbers full model}
\end{figure} 

\section{Conclusions} \label{section6}

    \noindent In this work we introduced kinetic SIR-type models to simulate the spread of an epidemic while explicitly accounting for the opinions that individuals hold regarding the adoption of protective behaviors. Opinion exchanges were mediated by a social network, represented by an underlying graphon \cite{DurFraWolZan}. In particular, each agent occupied a position on the graphon, which determined with which other individuals he was more prone to interact and by who he was more strongly influenced (these two properties were independent of each other). We also introduced a Fokker--Planck-type equation to simulate the spread of awareness regarding the disease's danger as the epidemic evolves.  

    Our kinetic epidemiological model is a generalization of the classical macroscopic SIR system, and we showed how to derive the latter starting from the former. We also introduced the propensity to interact of each agent as a combination of the above discussed properties depending on the position on the graphon. The importance of this new quantity is twofold. Firstly, it allowed us to distinguish between the presence and the absence of opinion leaders (such as politicians and influencers), highlighting the influence that a social network has on the dynamics occurring inside a population. This is an important novelty of our work since in previous articles (see \cite{BonBor,BonTosZan}) they had to be introduced separately from the rest of the population due to mathematical complexities. Secondly, the propensity to interact allowed us to to derive a simplification of the original model, which admits explicit global equilibria. These equilibria showed that agents prone to interact tend to be characterized by opinion polarization, while socially active ones can be characterized by consensus formation, highlighting the importance of social interactions to prevent the formation of extreme opinions, in line with the results presented in \cite{BonBor} by the first and second author. Moreover, we demonstrated strong $L^1$-convergence to equilibrium results toward these equilibria, also providing an explicit rate of convergence through relative entropy methods \cite{AurTosZan, BonBor}. 

    The spread of the popularity of the idea of the dangerousness of the disease did not influence the evolution of the epidemics, however it depended on the opinion of the individuals. This Fokker--Planck-type model admitted explicit equilibria, which showed that the population tends to consider the disease dangerous (respectively, not dangerous) not only if there are many agents with an opinion in favor of (respectively, against) a protective behavior, but they must be also prone to interact in order to diffuse their opinion among other agents. Importantly, we were able to demonstrate a convergence to equilibrium result in the homogeneous Sobolev space $\dot{H}^{-s}$, $s \in \big(\frac{1}{2}, 1\big)$.  Moreover, an additional converge to equilibrium result in $L^1$ was derived under a suitable additional assumption, which was proven to be realistic by the results of the numerical simulations.  

    To investigate the complete behaviors of the proposed models, we performed several numerical simulations based on a novel structure-preserving scheme \cite{ParZan}. In particular, the numerical scheme for the SIR model combines a Chang--Cooper type discretization for the Fokker--Planck opinion dynamics with a Runge--Kutta integration of the epidemiological operators, preserving key structural properties such as nonnegativity and mass conservation. Firstly, the numerical simulations reproduced the trends to equilibrium previously mentioned. Secondly, they allowed us to illustrate how the interplay between opinion dynamics and the graphon-mediated interactions affects the epidemic evolution, both in presence and absence of opinion leaders in favor or a protective behavior. For example, the presence of the latter reduced considerably the epidemic spread. Moreover, we introduced a time-dependent quantity, named effective reproduction number, resembling the classical basic reproduction number, whose temporal evolution may exhibit an oscillatory behavior, reflecting in epidemic waves. Notably, these waves naturally arose from our model, without the need of introducing an explicit external forcing (see, e.g., \cite{BonTosZan}). 

\appendix 
\section{A spatial adjacency graphon} \label{appA}

    \noindent Define the bounded symmetric graphon  
    \begin{align} \label{eq:graphon_app}
        \mathcal{B}(x,y)=
        \left\{
        \begin{aligned}
            & 1 \quad \textnormal{if}\ \  |x-y| \leq r, \\[2mm] 
            & 0 \quad \textnormal{if} \ \   |x-y| > r,
        \end{aligned}
        \right.
    \end{align}
    for some $0 < r <\frac{1}{4}$, meaning that each agent is able to communicate only with individuals $r$-close to them. Note that the agents with the largest number of connections (i.e., such that $d_\textnormal{in}$ is maximum) have a position on the graphon $x \in [r, 1-r]$. Consider the interaction function \eqref{eq:P}, from direct calculations it follows that  
    \begin{align} \label{eq:P_app}
        P(x,y)=
        \begin{cases}
            \left(1+\frac{x+r}{y+r}\right)^{-\chi} \quad &  \textnormal{if}\ \  x,y \in [0,r), \\[2mm] 
            \left(1+\frac{x+r}{2r}\right)^{-\chi}  &\textnormal{if} \ \   x \in [0,r), y \in [r,1-r], \\[2mm] 
            \left(1+\frac{x+r}{1-y+r}\right)^{-\chi} &\textnormal{if} \ \   x \in [0,r), y \in (1-r,1], \\[2mm] 
            \left(1+\frac{2r}{y+r}\right)^{-\chi} \quad & \textnormal{if}\ \  x \in [r,1-r],y \in [0,r), \\[2mm] 
            2^{-\chi} & \textnormal{if} \ \   x,y \in [r,1-r], \\[2mm]  
            \left(1+\frac{2r}{1-y+r}\right)^{-\chi} &\textnormal{if} \ \   x \in [r,1-r], y \in (1-r,1], \\[2mm] 
            \left(1+\frac{1-x+r}{y+r}\right)^{-\chi} \quad &\textnormal{if}\ \  x \in (1-r,1],y \in [0,r), \\[2mm] 
            \left(1+\frac{1-x+r}{2r}\right)^{-\chi} &\textnormal{if} \ \   x \in (1-r,1],y \in [r,1-r], \\[2mm]  
            \left(1+\frac{1-x+r}{1-y+r}\right)^{-\chi} &\textnormal{if} \ \   x,y \in (1-r,1]. 
        \end{cases}
    \end{align}
    Figure \ref{fig:appendix_graphon} shows the graphon \eqref{eq:graphon_app} and the interaction function \eqref{eq:P_app}. 

    \begin{figure}
    \centering
    \begin{subfigure}[h!]{0.5\textwidth}
        \centering
        \includegraphics[width=\linewidth]{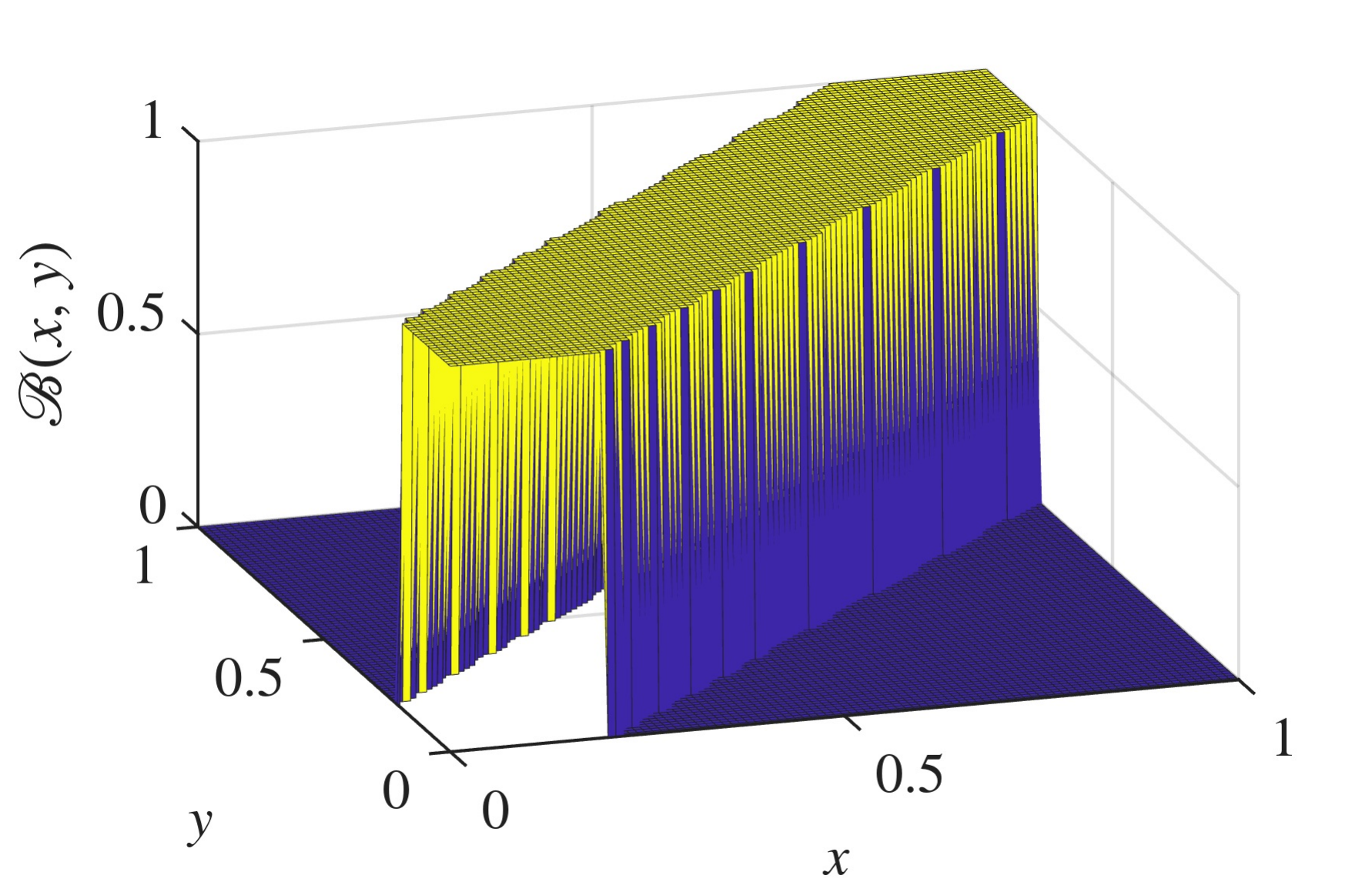}
        \caption*{(A)}
    \end{subfigure}\hfill
    \begin{subfigure}[h!]{0.5\textwidth}
        \centering
        \includegraphics[width=\linewidth]{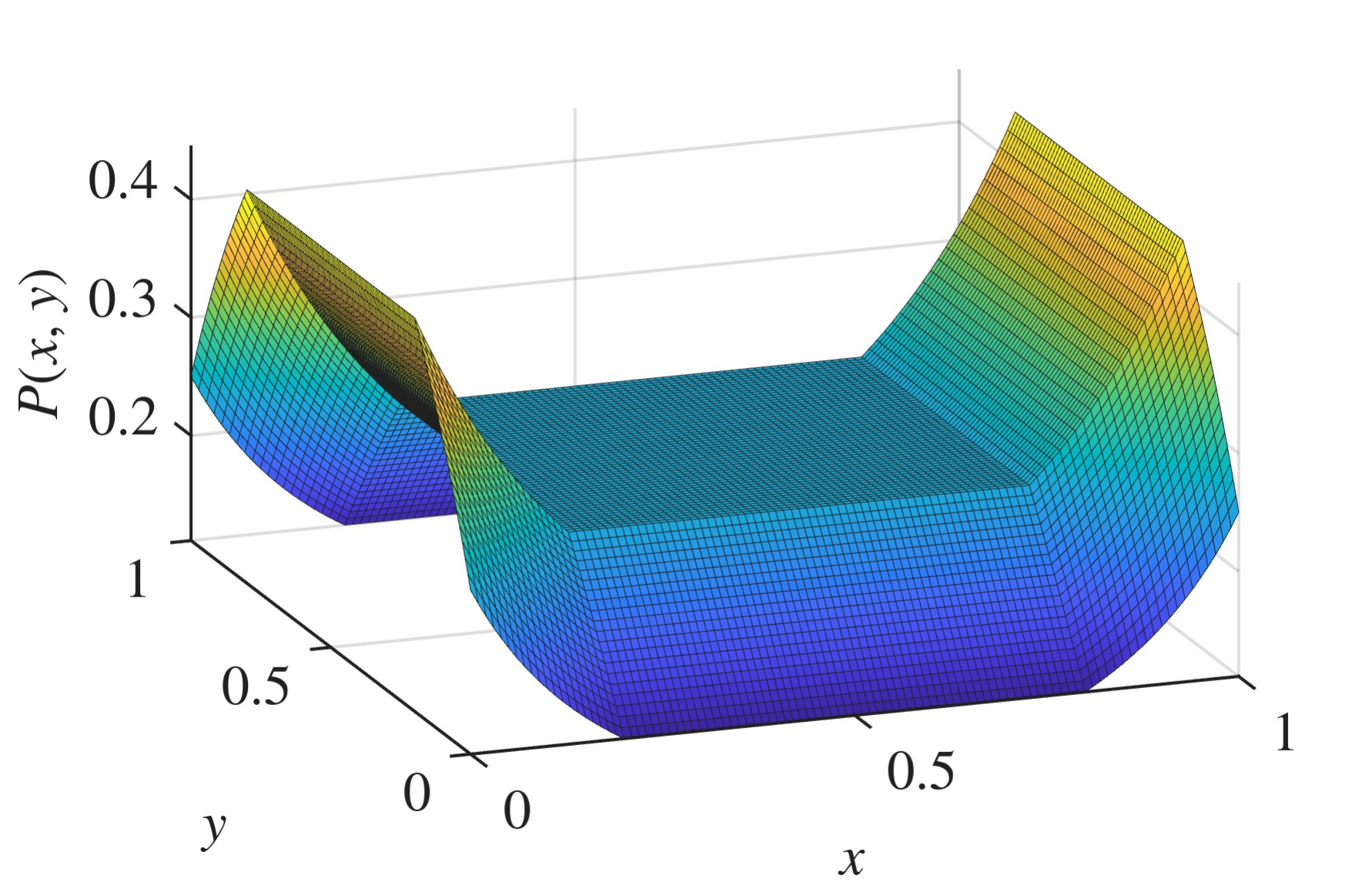}
        \caption*{(B)}
    \end{subfigure}
      \caption{(A) Plot of the graphon $\mathcal{B}$ given by \eqref{eq:graphon_app}. (B) Plot of the interaction function $P$ defined by \eqref{eq:P_app}. Values of the parameters: $r=0.2$ and $\chi=2$.}
    \label{fig:appendix_graphon}
    \end{figure}
    
    Cumbersome calculations show that 
    \begin{align} \label{eq:propensity_app}
        p(x)=
        \begin{cases}
             x\left(1+\frac{x+r}{2r}\right)^{-\chi} + \int_0^r \left(1+\frac{x+r}{y+r}\right)^{-\chi} \dd y \quad & \textnormal{if}\ \  x \in [0,r), \\[2mm] 
             x 2^{-\chi} + \int_{x-r}^r\left(1+\frac{2r}{y+r}\right)^{-\chi} \dd y \quad & \textnormal{if}\ \  x \in [r,2r), \\[2mm]
             r 2^{1-\chi} \quad & \textnormal{if} \ \ x \in [2r,1-2r], \\[2mm]
             (1-x) 2^{-\chi} + \int_{1-r}^{x+r} \left(1+\frac{2r}{1-y+r}\right)^{-\chi} \dd y \quad & \textnormal{if}\ \  x \in (1-2r,1-r], \\[2mm]
             (1-x)\left(1+\frac{1-x+r}{2r}\right)^{-\chi} + \int_{1-r}^1\left(1+\frac{1-x+r}{1-y+r}\right)^{-\chi} \dd y \quad & \textnormal{if}\ \  x \in (1-r,1]. 
        \end{cases}
    \end{align}
    Notice that $p(x)=p(1-x)$. 

    \begin{figure}[h!]
    \includegraphics[width=8cm]{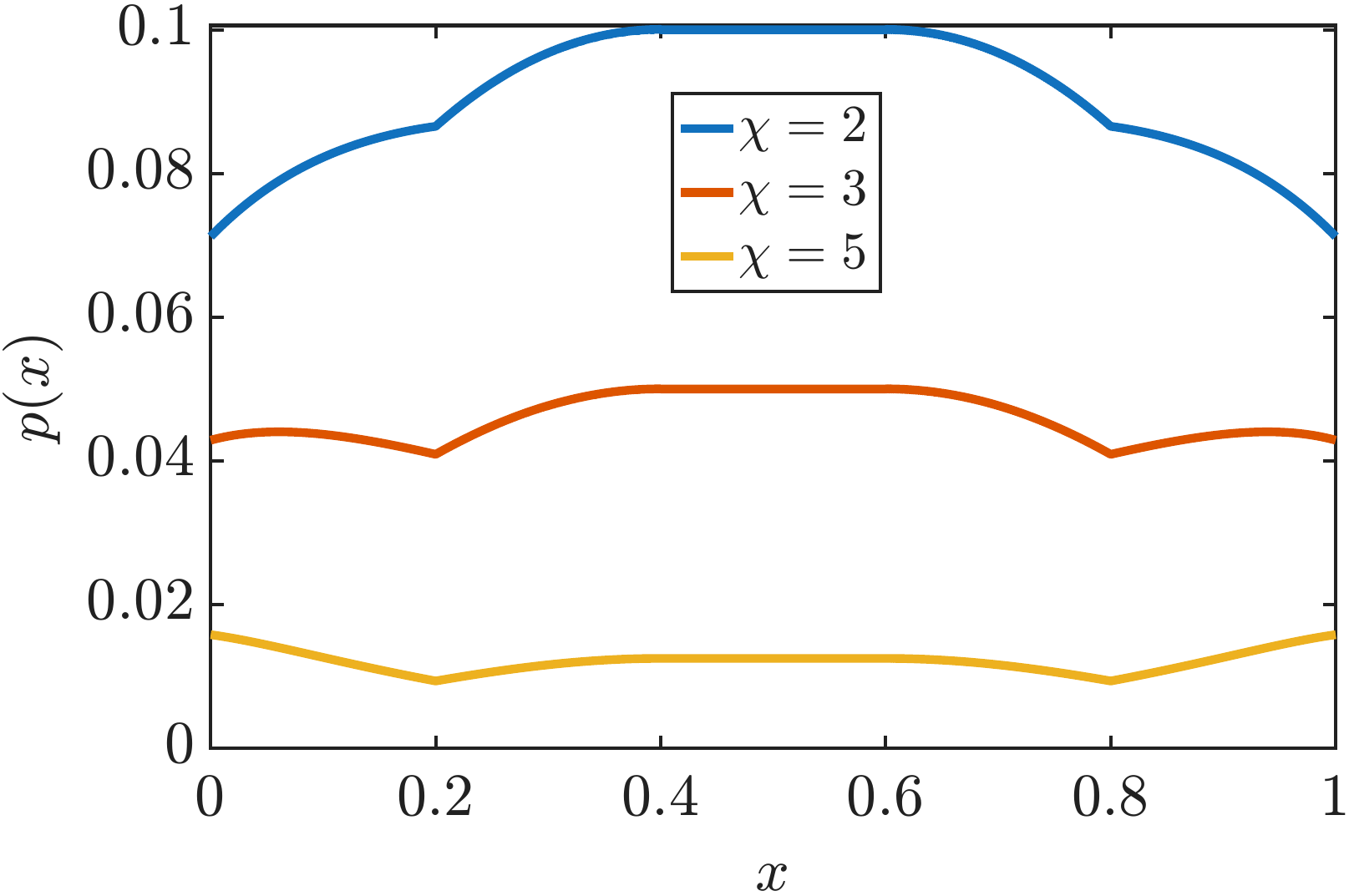}
    \caption{Graph of the propensity to interact $p$ defined by \eqref{eq:propensity_app}, for $r=0.2$ and different values of $\chi$.} \label{f4}
    \end{figure}
    
    Figure \ref{f4} shows the graph of the function $p(x)$ for $r=\frac{1}{5}$ and different values of $\chi$. If $\chi=2$ the agents with the largest propensity to interact have a position $x \in [2r, 1-2r]$, meaning that not only they have the largest number of connections, but they also communicate only with agents with again the largest number of connections (i.e., with position $x \in [r, 1-r]$). For example, agents with a position $x \in [r,2r)$ are among the ones with the largest number of connections but they communicate with individuals with position $x < r$ with a low number of connections. This situation is similar to the one described by the orange curve in Figure \ref{fig:propensity to interact}. The cases $\chi=3$ and $\chi=5$ are more interesting and complex. In both situations the agents with the lowest number of connections ($x=0$ and $x=1$) do not have the lowest propensity to interact, meaning that they are strongly influenced by their few social interactions. In particular, they have the highest propensity to interact if $\chi=5$. 

\bigskip
\bigskip
\noindent \textbf{Code availability.} The MATLAB code is available at \url{https://doi.org/10.5281/zenodo.19500723} \cite{MATLAB}.

\bigskip
\bigskip
\noindent \textbf{Acknowledgments.} AB and JB are members and acknowledge the support of {\it Gruppo Nazionale di Fisica Matematica} (GNFM) of {\it Istituto Nazionale di Alta Matematica} (INdAM). AB acknowledges the support from the INdAM-GNFM project CUP E5324001950001 (Multi-species non-Maxwellian Fokker-Planck models inferred from local non-equilibrium distributions), and from the European Union’s Horizon Europe research and innovation programme under the Marie Skłodowska-Curie grant agreement 101110920, project MesoCroMo (A Mesoscopic approach to Cross-diffusion Modelling in population dynamics). JB thanks the support of the project PRIN 2022 PNRR ``Mathematical modeling for a Sustainable Circular Economy in Ecosystems'' (project code P2022PSMT7, CUP D53D23018960001) funded by the European Union - NextGenerationEU, PNRR-M4C2-I 1.1, and by MUR-Italian Ministry of Universities and Research. JB and MF acknowledge the support of the University of Parma through the action Bando di Ateneo 2022 per la ricerca, cofunded by MUR-Italian Ministry of Universities and Research - D.M. 737/2021 - PNR - PNRR - NextGenerationEU (project ``Collective and Self-Organised Dynamics: Kinetic and Network Approaches''). MF acknowledges the support from the European Union’s Horizon Europe research and innovation programme under the Marie Skłodowska-Curie Doctoral Network DATAHYKING (Grant No. 101072546), and is supported by the French government, through the UniCA$_{JEDI}$ Investments in the Future project managed by the National Research Agency (ANR) with the reference number ANR-15-IDEX-01.

\bigskip
\bigskip
\noindent \textbf{Disclaimer.} Funded by the European Union. Views and opinions expressed are however those of the author(s) only and do not necessarily reflect those of the European Union or of the European Research Executive Agency (REA). Neither the European Union nor the granting authority can be held responsible for them.

\begin{figure}[h!]
\begin{flushleft}
\includegraphics[scale=0.3]{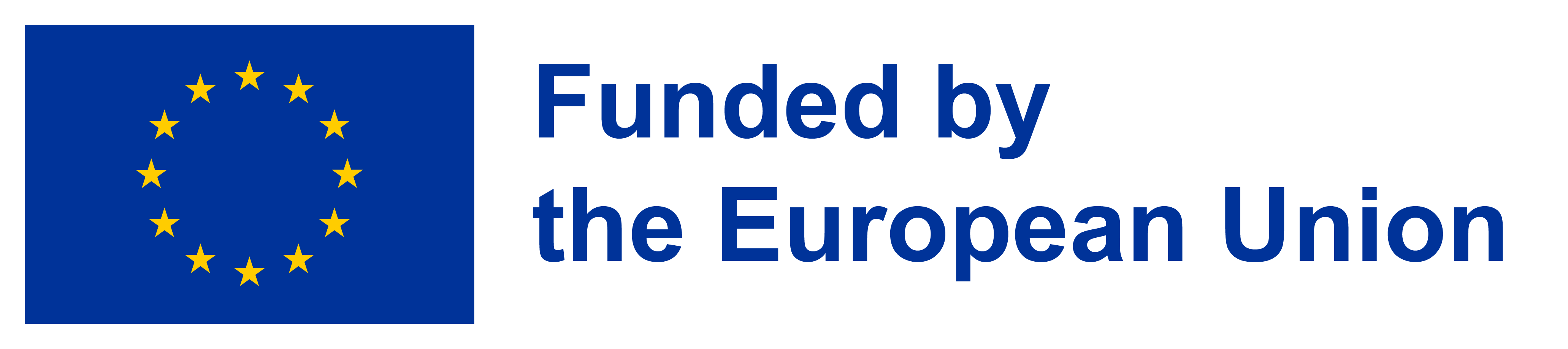}
\end{flushleft}
\end{figure}

\bigskip
\noindent \textbf{Funding.} This research has been supported by the European Union’s Horizon Europe research and innovation programme (MSCA-PF grant agreement 101110920, project MesoCroMo ``A Mesoscopic approach to Cross-diffusion Modelling in population dynamics'', and MSCA-DN grant agreement 101072546, project DATAHYKING), by INdAM--GNFM (project ``Multi-species non-Maxwellian Fokker--Planck models inferred from local non-equilibrium distributions'', CUP E5324001950001), by the European Union - NextGenerationEU and by the MUR-Italian Ministry of Universities and Research (project “Collective and Self-Organised Dynamics: Kinetic and Network Approaches”, action ``Bando di Ateneo 2022 per la ricerca'' D.M. 737/2021 - PNR - PNRR, and project ``Mathematical Modelling for a Sustainable Circular Economy in Ecosystems'', PRIN 2022 PNRR, code P2022PSMT7, CUP D53D23018960001).


\bigskip
\bibliographystyle{plain}
\bibliography{Bibliography_KSIR}
\bigskip
\bigskip

\setlength\parindent{0pt}

\end{document}